\documentclass[11pt]{article}
\usepackage[a4paper, margin=0.9in]{geometry}
\usepackage{float}
\usepackage{amsmath}
\usepackage{mathtools}
\usepackage{mathrsfs}
\usepackage{bm}
\usepackage{amssymb}
\usepackage{graphicx} 
\usepackage{setspace}
\usepackage{xcolor}
\usepackage{algorithm,algpseudocode}
\usepackage{amsthm}
\usepackage{caption}
\usepackage{subcaption}
\usepackage{comment}
\usepackage{booktabs}
\usepackage{multirow}
\usepackage{longtable,tabularx}
\usepackage{makecell}
\usepackage{bbm}
\usepackage{enumitem}
\usepackage{dsfont}

\usepackage{url}

\usepackage[american]{babel}

\usepackage{chngcntr}
\counterwithout{figure}{section}  
\counterwithout{table}{section}

\newcommand{\TheoremHeaderFont}{\bfseries}
\newtheoremstyle{plainreading}
  {6pt}{6pt}{\itshape}{}{\TheoremHeaderFont}{.}{0.5em}{}
\theoremstyle{plainreading}
\newtheorem{theorem}{Theorem}[section]    
\newtheorem{proposition}[theorem]{Proposition}
\newtheorem{lemma}[theorem]{Lemma}        
\newtheorem{assumption}{Assumption}   
\newtheorem{definition}[theorem]{Definition}  
\newtheorem{corollary}[theorem]{Corollary}   

\newtheorem{example}[theorem]{Example}

\def\mb{\mathbb}

\def\bs{\boldsymbol}

\def\ind{\mathbb{I}}
\def\Re{\mathbb{R}}

\def\conv{\mathrm{conv}}
\def\diam{\mathrm{diam}}
\def\dist{\mathrm{dist}}

\def\txi{\tilde{\xi}}

\def\A{\mathcal{A}}
\def\C{\mathcal{C}}
\def\I{\mathcal{I}}
\def\J{\mathcal{J}}
\def\M{\mathcal{M}}
\def\N{\mathcal{N}}
\def\T{\mathcal{T}}
\def\U{\mathcal{U}}

\def\X{\mathcal{X}}

\def\kyfan{\mathrm{T}}

\def\0{\bs{0}}
\def\1{\bs{1}}

\usepackage{soul} 

\usepackage[numbers,square]{natbib}
\usepackage[bookmarksnumbered=true]{hyperref}
\hypersetup{
  colorlinks=true,
  linkcolor=blue,
  citecolor=blue,
  urlcolor=blue,
  pdfauthor={Zhiping Li, Nan Jiang, Rujun Jiang},
  pdftitle={Solving Chance Constrained Programs via a Penalty based Difference of Convex Approach}
}

\algdef{SE}[DOWHILE]{Do}{doWhile}{\algorithmicdo}[1]{\algorithmicwhile\ #1}
\algnewcommand{\Or}{\textbf{or}}
\algnewcommand{\And}{\textbf{and}}
\algnewcommand{\Input}{\item[\textbf{Input:}]}

\makeatletter
\let\RevisionFloatMakeBox\float@makebox
\renewcommand{\float@makebox}[1]{\vbox{\color{black}\RevisionFloatMakeBox{#1}}}
\makeatother

\title{Solving Chance Constrained Programs via a Penalty based Difference of Convex Approach}

\date{}
\author{Zhiping Li\textsuperscript{1}\quad\quad
        Nan Jiang\textsuperscript{2}\quad\quad
        Rujun Jiang\textsuperscript{3, $\dag$}}

\begin{document}
\maketitle

\begingroup
\begin{NoHyper}
\renewcommand\thefootnote{1}
\footnotetext{Department of Industrial and Systems Engineering, University of Minnesota. Email: \href{mailto:li004950@umn.edu}{li004950@umn.edu}.}
\renewcommand\thefootnote{2}
\footnotetext{Department of Industrial Engineering and Decision Analytics, The Hong Kong University of Science and Technology. Email: \href{mailto:nanjiang@ust.hk}{nanjiang@ust.hk}.}
\renewcommand\thefootnote{3}
\footnotetext{School of Data Science, Fudan University. Email: \href{mailto:rjjiang@fudan.edu.cn}{rjjiang@fudan.edu.cn}.}
\renewcommand\thefootnote{$\dag$}
\footnotetext{Corresponding author.}
\end{NoHyper}
\endgroup

\onehalfspacing

\begin{abstract}
    We develop two penalty based difference of convex (DC) algorithms for solving sample average approximations of chance constrained programs. Using a rank-based DC decomposition of the empirical chance constraint, we propose a proximal DC algorithm in the primal space that does not require a feasible starting point. To improve numerical stability in nonlinear settings, we derive a lifted penalty formulation via Toland duality and show that minimizing over the primal variables yields a concave minimization problem over a polytope. Based on this structure, we develop a proximal DC algorithm in the lifted space. We establish exact penalty and stationarity guarantees under suitable regularity or polyhedral assumptions. We also characterize local minimizers of both formulations and provide a sufficient condition under which a lifted local minimizer projects to a primal local minimizer. Numerical experiments show that the lifted method delivers strong performance in both solution quality and runtime, with substantial runtime improvements over state-of-the-art methods.
\end{abstract}


\section{Introduction}

As a powerful tool for addressing uncertainty in decision-making, the chance constrained program has been extensively studied in the literature. It has been widely applied across fields such as finance, energy systems, and supply chain management~\citep{ahmed2008solving,kuccukyavuz2022chance}. In general, a chance constrained program can be written as follows:
\begin{equation} \label{p:ccp}
    \min_{\bs{x} \in \X} \left\{ f(\bs{x}): \Pr\{g(\bs{x}, \txi) \leq 0\} \geq 1-\alpha \right\},
\end{equation}
where $\txi$ is a random vector supported on $\Xi\subseteq\Re^{d_\xi}$ and $\alpha\in(0,1)$ is a prescribed violation risk level. The set $\X$ is a deterministic set contained in an open set $\U\subseteq\Re^d$, and the functions $f:\U\to\Re$, $g:\U\times\Xi \to \Re$ are real-valued functions. We use $\txi$ to denote the random vector and $\xi$ to denote its realization. The goal of Problem~\eqref{p:ccp} is to minimize the objective function $f(\bs{x})$ subject to the requirement that the constraint $g(\bs{x}, \txi) \leq 0$ be satisfied with probability at least $1-\alpha$. Problem~\eqref{p:ccp} is generally difficult to solve, particularly when the distribution of $\txi$ is unknown. Even checking whether a candidate solution satisfies the chance constraint can be computationally demanding. The feasible region is also generally nonconvex, even when the set $\X$ is convex, making it difficult to find an optimal solution.

Given the challenges mentioned above, there are two major strategies to address the chance constraints. The first strategy uses tractable conservative approximations, where the chance constrained program is formulated as a convex optimization problem that can be efficiently solved and yields a feasible solution to the original chance constraints~\citep{rockafellar2002conditional, ahmed2018relaxations,cheung2012linear}. The second approach is the Sample Average Approximation (SAA) and has been extensively studied by \citet{ahmed2008solving} and \citet{luedtke2008sample}. It is a popular method to approximately solve Problem~\eqref{p:ccp}, especially when the distribution of $\txi$ is unknown, and we can only access a sample of $S$ i.i.d.~realizations $\{\xi^s\}^S_{s=1}$.

In this paper, we consider the SAA formulation of Problem~\eqref{p:ccp} written as
\begin{equation}\label{p:ccp-saa}
    v^* = \min_{\bs x\in\X}\left\{ f(\bs x): \frac{1}{S} \sum^S_{s=1}\ind\{g(\bs x, \xi^s) \leq 0\} \geq 1 - \alpha \right\},
\end{equation}
where the indicator function $\ind\{g(\bs x, \xi^s) \leq 0\} = 1$ if $g(\bs x, \xi^s) \leq 0$ and $0$ otherwise. We assume that the feasible region of~\eqref{p:ccp-saa} is nonempty and the objective function is bounded below over the feasible region. Moreover, we focus on the convex setting and make the following assumptions. Note that for a positive integer $n$, we write $[n]:=\{1,2,\ldots,n\}$.

\begin{assumption}\label{ass:X-convex}
    The set $\X\subseteq \Re^d$ is compact and convex with a nonempty relative interior.
\end{assumption}

\begin{assumption}\label{ass:f-g-convex}
    The function $f: \U \to \Re$ is continuously differentiable and convex. The function $g:\U \times \Xi \to \Re$ takes the form $g(\bs{x}, \xi) = \max_{i \in [I]} h_i(\bs{x}, \xi)$, where $h_i : \U \times \Xi \to \Re$ is continuously differentiable and convex in $\bs{x}$ for every $\xi \in \Xi$ and $i \in [I]$. Moreover, $\U \subseteq\Re^d$ is an open set that contains $\X$.
\end{assumption}

Assumptions~\ref{ass:X-convex} and~\ref{ass:f-g-convex} are common in the literature~\citep{wang2025proximal, jiang2022also, bai2021augmented} and cover several standard settings, e.g., the objective function $f(\bs{x})$ is linear, and the feasible set $\X$ is a polytope~\citep{luedtke2010integer}. Note that Assumptions~\ref{ass:X-convex} and~\ref{ass:f-g-convex} imply that $\nabla f(\bs{x})$ is bounded on $\X$ and thus $f$ is Lipschitz continuous on $\X$. That is, there exists $L_f > 0$ such that $|f(\bs{x}) - f(\bs{x}')| \leq L_f \|\bs{x} - \bs{x}'\|$ holds for all $\bs{x}, \bs{x}' \in \X$.

Problem~\eqref{p:ccp-saa} is known as a single chance constrained program if $I=1$ and a joint chance constrained program otherwise. Given a sample of $S$ realizations $\{\xi^s\}^S_{s=1}$, we use $g_s(\bs{x})$ to denote $g(\bs{x}, \xi^s)$ and $h_{i, s}(\bs{x})$ to denote $h_i(\bs{x}, \xi^s)$ for all $s\in[S]$, $i \in [I]$ in the rest of the paper.

\subsection{Relevant Literature}

\noindent\textbf{Chance Constrained Programs.}
Chance constrained programming has a long history as a framework for incorporating uncertainty directly into decision models. Since its early development~\citep{charnes1958cost,charnes1963deterministic}, it has become a central tool in applications where decisions must remain feasible with high probability. In energy systems, it handles renewable generation, demand fluctuations, and network contingencies~\citep{cho2023exact}. In finance, probabilistic constraints help balance portfolio performance against downside risk~\citep{deng2021scenario}. In logistics and supply chain management, they enforce service, routing, or inventory requirements despite uncertain demand or disruptions~\citep{dinh2018exact}. We refer interested readers to \citet{ahmed2008solving} and \citet{kuccukyavuz2022chance} for a comprehensive review.

A main difficulty of chance constrained programs~\eqref{p:ccp} is that their feasible regions are typically nonconvex, limiting the direct use of convex optimization tools. A common remedy is SAA, which replaces the probabilistic constraint with finitely many scenarios and yields a deterministic mixed-integer formulation. Although modern mixed-integer solvers can handle this formulation, it is usually practical only for moderate-scale instances~\citep{luedtke2008sample,ruszczynski2002probabilistic}. Another approach constructs convex inner approximations of the chance constraint~\citep{nemirovski2007convex}. A representative example is the conditional value-at-risk (CVaR) reformulation, which provides a convex conservative approximation of the chance constraint. Such approximations preserve feasibility but can be conservative and generally do not recover the optimal value of the original problem. \citet{hong2011sequential} approximated the joint chance constraint by a difference of convex (DC) function and solved the resulting problem by a sequential convex approximation (SCA) algorithm. \citet{adam2016nonlinear} formulated the SAA problem as a mathematical program with complementarity constraints (MPCC) and solved it by an iterative regularization method that enlarges the feasible region of the complementarity constraints.

More recent work mitigates the conservatism of single-shot convex approximations through iterative refinement procedures~\citep{jiang2022also,jiang2025also}, where the ALSO-X method is a representative example. When the deterministic set $\X$ is convex, its iterative updates can improve upon CVaR-based solutions. Building on ALSO-X, the enhanced ALSO-X+ method often attains better objective values than ALSO-X and CVaR in numerical experiments, but these studies also report substantially longer runtimes, especially on larger instances. Moreover, the theoretical properties of ALSO-X+ remain incomplete. Existing analyses by \citet{jiang2022also} do not establish convergence to a stationary point of the underlying chance constrained problem, leaving the quality of its limiting solutions without a rigorous optimality guarantee.

Recently, penalty based methods have also been developed for solving chance constrained programs. \citet{curtis2018sequential} proposed a penalty based trust-region algorithm which requires solving quadratic subproblems with linear cardinality constraints. Subsequently, \citet{pena2020solving} reformulated chance constraints through quantile functions and approximated the quantile function by a differentiable sample-based nonlinear approximation. More recently, \citet{cui2022nonconvex} studied a generalized version of chance constrained programs and proposed an algorithm via sequentially solving penalized convex surrogates obtained through sampling and continuous approximation.

\noindent\textbf{Difference of Convex (DC) Program.}
DC programs are optimization problems that aim to minimize a DC function subject to DC constraints. DC programs play an important role in nonconvex optimization and have been extensively studied for decades~\citep{horst1999dc, tao1997convex}. One of the most important methods for solving DC programs is the DC algorithm (DCA)~\citep{tao1997convex}, which solves a sequence of convex subproblems by replacing the concave part with its first-order approximation at the current iterate.

Recently, \citet{wang2025proximal} proposed an exact rank-based DC reformulation of the SAA chance constraint and a corresponding DCA, with subproblems that are easy to implement and are amenable to off-the-shelf solvers. Starting from a feasible point, their method successively solves convex inner approximations that preserve feasibility. However, finding such an initialization can be difficult, and these inner approximations may restrict the search for better solutions.

Another popular method for solving DC programs is the exact penalty method~\citep{lipp2016variations,gotoh2018dc, jara2018study}, which penalizes the DC constraint in the objective function and solves a sequence of DC programs using DCA with increasing penalty parameters. Compared with the DCA, the exact penalty method does not require an initial feasible solution. It allows constraints to be violated in the initial updates, which can sometimes lead to a lower objective value~\citep{lipp2016variations}.

\subsection{Summary of Contributions}

In this paper, we study penalty based DC algorithms for solving chance constrained programs. Our contributions can be summarized as follows:
\begin{itemize}
    \item[(i)]
    Building on the rank-based DC decomposition of the SAA chance constraint proposed by \citet{wang2025proximal}, we first propose a penalty based DC method in the primal space. The method does not require an initial feasible solution and can alleviate over-conservatism caused by overly restrictive convex inner approximations. This algorithm can improve both solution quality and runtime in the polyhedral setting, where both $\X$ and $g$ are polyhedral, but can exhibit longer running time in the nonlinear setting.

    \item[(ii)]
    To reduce this running time, we derive an equivalent lifted formulation with complementarity constraints, inspired by the Toland duality for DC programs~\citep{toland1978duality}. This leads to a tractable DC structure over a polytope and yields a proximal DC algorithm in the lifted space. For a fixed penalty parameter, the inner loop with unregularized subproblems admits finite termination, is amenable to effective warm starts, and exhibits substantially improved numerical stability.

    \item[(iii)]
    We establish global exact penalty properties for both the primal and lifted formulations and local exactness for the primal formulation under suitable regularity or polyhedral assumptions. Under a relative Slater condition or polyhedral assumptions, we show that a feasible point of the original lifted problem is strongly stationary if and only if it is stationary for its penalized counterpart for all sufficiently large penalty parameters. We further characterize lifted and primal local optimality through scenario subproblems, explain why lifted local minimizers may fail to project to primal local minimizers, and provide a sufficient condition ensuring that lifted local optimality implies primal local optimality.

    \item[(iv)]
    Numerical experiments demonstrate that the proposed methods can reduce over-conservatism. Moreover, the algorithm in the lifted space offers a favorable trade-off between solution quality and runtime on the tested instances.
\end{itemize}

\noindent\textbf{Organization.}
The remainder of the paper is organized as follows. Section~\ref{sec:alg} develops a penalty based DC algorithm in the primal space. It then derives an equivalent lifted formulation with complementarity constraints, develops the corresponding penalty based DC algorithm, and analyzes the convergence properties of both algorithms. Section~\ref{sec:exact-penalty-framework} establishes global exact penalty properties for both formulations, local exactness for the primal formulation, and the relationship between strong stationarity of the original lifted problem and stationarity of its penalized counterpart at points feasible for the former. Section~\ref{sec:iden-localmin} characterizes local optimality in both formulations and provides a sufficient condition under which lifted local optimality implies primal local optimality. Section~\ref{sec:numerical} reports numerical experiments on real and synthetic instances and compares the proposed methods with existing methods. Section~\ref{sec:conclusion} concludes the paper. Proofs are presented immediately after the corresponding results. Appendix~\ref{sec:preliminaries} provides preliminaries on exact penalization.

\noindent\textbf{Notation.}\phantomsection\label{sec:notations} We use $\Re^n_+$ to denote the nonnegative orthant. For a vector $\bs{x}\in\Re^d$, $\|\bs{x}\|$ denotes the Euclidean norm and $\langle \cdot,\cdot\rangle$ denotes the standard inner product. We use $\|\bs{x}\|_1$ to denote the $\ell_1$-norm and $\|\bs{x}\|_0$ to denote the $\ell_0$ pseudo-norm, i.e., the number of nonzero elements. For a scalar $t\in\Re$, define $[t]_+:=\max\{t,0\}$. For a scalar variable $\tau$, $\tau\downarrow0$ means $\tau\to0$ with $\tau>0$. For a scalar sequence $\{a_t\}_{t\ge1}$, we write $a_t\downarrow0$ if $a_t>0$ and $a_{t+1}\le a_t$ for all $t\ge1$ and $a_t\to0$. For $\bar{\bs{x}}\in\Re^d$ and $\varepsilon>0$, $B(\bar{\bs{x}},\varepsilon):=\{\bs{x}\in\Re^d:\ \|\bs{x}-\bar{\bs{x}}\|<\varepsilon\}$	denotes the open neighborhood with radius $\varepsilon$. Given a set $\mathcal{S}\subseteq\Re^d$, we use $\mathrm{ri}(\mathcal S)$ to denote its relative interior, and its indicator function is
\[
\delta_{\mathcal{S}}(\bs{x})
:= \begin{cases}
    0, & \bs{x}\in\mathcal{S},\\
    +\infty, & \bs{x}\notin\mathcal{S}.
\end{cases}
\]
The distance from a point to a set is
$\dist(\bs{x},\mathcal{S}):=\inf\{\|\bs{x}-\bs{y}\|:\ \bs{y}\in\mathcal{S}\}$,
and the diameter of the set is $\diam(\mathcal{S}):=\sup\{\|\bs{x}-\bs{y}\|:\ \bs{x},\bs{y}\in\mathcal{S}\}$. For a closed convex set $\X$, $\T_{\X}(\bar{\bs{x}})$ and $\N_{\X}(\bar{\bs{x}})$ denote the tangent cone and normal cone at $\bar{\bs{x}}\in\X$. For an extended-real-valued function $f:\Re^d\to\Re\cup\{+\infty\}$ and a point $\bar{\bs{x}}\in\operatorname{dom}f$, the directional derivative of $f$ at $\bar{\bs{x}}$ in direction $\bs{d}$ is defined by
\[
f'(\bar{\bs{x}};\bs{d})\ :=\ \lim_{t\downarrow 0}\ \frac{f(\bar{\bs{x}}+t\bs{d})-f(\bar{\bs{x}})}{t},
\] whenever the limit exists.
If $f$ is locally Lipschitz, its Clarke directional derivative is
\[
f^\circ(\bar{\bs{x}};\bs{d})\ :=\ \limsup_{\bs{x}\to\bar{\bs{x}}, t\downarrow 0}\ \frac{f(\bs{x}+t\bs{d})-f(\bs{x})}{t},
\]
and its Clarke subdifferential is
\[
\partial f(\bar{\bs{x}})
\ :=\
\Bigl\{\bs{v}\in\Re^d:\ \langle \bs{v},\bs{d}\rangle\le f^\circ(\bar{\bs{x}};\bs{d})\ \ \forall \bs{d}\in\Re^d\Bigr\}.
\]
For locally Lipschitz $f$, the Clarke directional derivative is the support function of its Clarke subdifferential~\citep[Proposition~2.1.2(b)]{clarke1990optimization}:
\begin{equation}\label{eq:notation-clarke-support}
f^\circ(\bar{\bs{x}};\bs{d})=\max_{\bs{v}\in\partial f(\bar{\bs{x}})}\langle\bs{v},\bs{d}\rangle.
\end{equation}
When $f$ is convex, the Clarke subdifferential coincides with the usual convex subdifferential, and $f^\circ(\bar{\bs{x}};\bs{d})=f'(\bar{\bs{x}};\bs{d})$. For a closed proper convex function $f$, its Fenchel conjugate is $f^*(\bs{p})\ :=\ \sup_{\bs{x}\in\Re^d}\ \{\langle \bs{p},\bs{x}\rangle-f(\bs{x})\}$. For a DC program in the form $\min_{\bs{x}}\{\phi(\bs{x})-\psi(\bs{x})\}$ with $\phi,\psi$ closed proper convex, its Toland dual~\citep{toland1978duality} is $\min_{\bs{p}}\ \{\psi^*(\bs{p})-\phi^*(\bs{p})\}$, which is also a DC program. A DC program is called a polyhedral DC program if either $\phi$ or $\psi$ is polyhedral, i.e., the epigraph is a polyhedron. The polyhedral DC program has some interesting properties on local optimality and finite convergence of the DC algorithm. The reader may refer to \citet{tao1997convex} for a detailed review.

For a vector $\bs{y}\in\Re^S$, let $y_{(1)}\le\cdots\le y_{(S)}$ denote the nondecreasing rearrangement of $(y_1,\ldots,y_S)$ and $|y|_{(1)}\le\cdots\le |y|_{(S)}$ denote the nondecreasing rearrangement of $(|y_1|,\ldots,|y_S|)$. The sum of the $r$ largest components is defined as $\kyfan_r(\bs{y}) = \sum_{i=S-r+1}^S y_{(i)}$. The Ky Fan norm is defined as $\|\bs{y}\|_{(r)}\ :=\ \sum_{i=S-r+1}^S |y|_{(i)}$.

\section{Algorithmic Frameworks}\label{sec:alg}

In this section, we present the overall algorithmic foundations of our penalty based approach to solving the SAA formulation of the chance constrained program~\eqref{p:ccp-saa}. We begin by reviewing the DC reformulation of the chance constraint in the primal space and discuss why the corresponding exact penalty formulation can be fragile in practice, especially in nonlinear settings. Motivated by these limitations and inspired by the general Toland duality of DC programs~\citep{toland1978duality}, we propose an alternative exact penalty approach and derive an equivalent lifted formulation as an MPCC. We then show that the associated value-function representation yields a tractable DC structure. The resulting polyhedral DC formulation supports an algorithm with a finite convergence guarantee and effective warm starts.

\subsection{Algorithm in the Primal Space}\label{sec:primal-algorithm}

To obtain a tractable reformulation of Problem~\eqref{p:ccp-saa}, \citet{wang2025proximal} recently proposed a DC formulation of Problem~\eqref{p:ccp-saa} along with a DCA-based solution method. The algorithm is easy to implement, as its subproblems can be solved directly using off-the-shelf solvers, and is particularly efficient in the polyhedral setting. We begin this section by briefly reviewing its key elements. Let $m:=\lfloor \alpha S\rfloor$ and let $g_{(1)}(\bs{x})\le\cdots\le g_{(S)}(\bs{x})$ denote the nondecreasing rearrangement of the vector $(g_1(\bs{x}),\dots,g_S(\bs{x}))$. Throughout the paper, we assume that $m > 0$. Define
\begin{equation}\label{eq:G-def}
    G_1(\bs{x}) := \sum_{s=S-m}^{S} g_{(s)}(\bs{x}), \qquad G_2(\bs{x}) := \sum_{s=S-m+1}^{S} g_{(s)}(\bs{x}).
\end{equation}

The following equivalence is standard, and we include it for completeness.

\begin{proposition}[Lemma~1 of \citealt{wang2025proximal}]
    \label{prop:equiv-cc}
    For any $\bs{x}\in\X$, the sample-based chance constraint $\sum_{s=1}^S \mb{I}\{g_s(\bs{x})\le 0\} \;\ge\; S-m$ holds if and only if $G_1(\bs{x}) - G_2(\bs{x}) \leq 0$.
\end{proposition}

By Proposition~\ref{prop:equiv-cc}, Problem~\eqref{p:ccp-saa} is equivalent to
\begin{equation} \label{p:ccp-dc}
    v^* = \min_{\bs{x}} \left\{ f(\bs{x}) : \bs{x} \in \hat\X\right\},
\end{equation}
where the feasible region $\hat\X$ of~\eqref{p:ccp-dc} is defined as $
\hat\X := \{\bs{x}\in\X : G(\bs{x}) := G_1(\bs{x}) - G_2(\bs{x}) \leq 0\}.
$

We next summarize several properties of the functions $G_1$ and $G_2$ established by \citet{wang2025proximal}.

\begin{proposition}[Lemmas~1 and~2 of \citealt{wang2025proximal}] \label{prop:G1G2-subdiff}
    Under Assumptions~\ref{ass:X-convex} and~\ref{ass:f-g-convex}, both $G_1(\bs{x})$ and $G_2(\bs{x})$ defined in~\eqref{eq:G-def} are convex and continuous on the set $\X$. Moreover, let $\M_{G_2}(\bs{x})$ denote the active index set of $G_2$ at $\bs{x}$, defined by
    \begin{equation*}
    \M_{G_2}(\bs{x})
    =
    \left\{(s_1,\ldots,s_m):
    1\le s_1<\cdots<s_m\le S,\
    \sum_{t=1}^{m}g_{s_t}(\bs{x})=G_2(\bs{x})
    \right\}.
    \end{equation*}
    Then for any $\bs{x}\in\X$, the subdifferential of $G_2$ admits the explicit representation
    \begin{equation*}
        \begin{aligned}
        \partial G_2(\bs{x})&=\conv\left\{\bigcup_{(s_1,\ldots,s_m)\in \M_{G_2}(\bs{x})}\sum_{t=1}^{m} \partial g_{s_t}(\bs{x})\right\},\\
        \partial g_s(\bs{x})&=\conv\Bigl\{ \nabla h_{i,s}(\bs{x}): i\in [I],\ h_{i,s}(\bs{x})=g_s(\bs{x}) \Bigr\},\ s\in[S].
        \end{aligned}
    \end{equation*}
\end{proposition}

We make the following remarks on Propositions~\ref{prop:equiv-cc} and~\ref{prop:G1G2-subdiff}.
\begin{itemize}
    \item[(i)] The sample-based chance constraint is equivalent to an empirical value-at-risk (VaR) constraint~\citep{wang2025proximal}.
    The requirement that at least $(S-m)$ scenarios satisfy $g_s(\bs{x})\le 0$ holds if and only if $ g_{(S-m)}(\bs{x}) \le 0, $ i.e., the empirical $(1-\alpha)$-quantile of $\{g_s(\bs{x})\}_{s=1}^S$ is nonpositive.

    \item[(ii)] Under Assumptions~\ref{ass:X-convex}--\ref{ass:f-g-convex}, each $h_{i,s}$ has a bounded gradient on the compact convex set $\X$ and is therefore Lipschitz continuous on $\X$. Since each $g_s$ is a finite maximum of the functions $h_{i,s}$, and $G_1$ and $G_2$ are finite maxima of sums of the functions $g_s$, it follows that $G_1$, $G_2$, and their difference $G$ are Lipschitz continuous on $\X$. These functions are also locally Lipschitz on the open set $\U$ containing $\X$, so their Clarke subdifferentials are well-defined at every point of $\X$~\citep[Section~2.1]{clarke1990optimization}.
\end{itemize}

\citet{wang2025proximal} applied DCA to Problem~\eqref{p:ccp-dc}. At the $k$th iteration, given $\bs{x}^k\in\hat\X$, the algorithm selects a subgradient $\bs{n}^k\in\partial G_2(\bs{x}^k)$ to linearize the concave part of the constraint, and computes the next iterate by solving the following subproblem:
\begin{equation} \label{eq:dca-update}
    \bs{x}^{k+1} \in \arg\min_{\bs{x} \in \X} \{f(\bs{x}):G_1(\bs{x}) - G_2(\bs{x}^k) - \langle \bs{n}^k, \bs{x}-\bs{x}^k \rangle \leq 0\}.
\end{equation}
We note that the above DCA scheme has several intrinsic limitations despite its strong numerical performance:
\begin{itemize}
    \item[(i)] DCA constructs a sequence of conservative convex inner approximations of Problem~\eqref{p:ccp-saa}. Although this construction preserves feasibility when initialized at a feasible point, the affine majorization of the concave part may overly restrict the search and lead to solutions with poor objective values. Other conservative approximation techniques, such as the CVaR approximation~\citep{rockafellar2002conditional}, may even yield an infeasible approximation when the original chance constrained problem is feasible.
    \item[(ii)] The DCA requires an initial feasible point (i.e., $\bs{x}^0\in\hat\X$) to linearize the constraint and to preserve feasibility of its successive convex subproblems. Finding such a point can be computationally demanding, especially when the convex CVaR approximation is infeasible.
    \item[(iii)] Problem~\eqref{p:ccp-dc} is a polyhedral DC program when $h_{i,s}$ are affine for all $i\in[I]$ and $s\in[S]$. While DCA can be computationally efficient in the polyhedral setting, it often exhibits longer running time with nonlinear constraints. For example, \citet[Section~5.4]{wang2025proximal} consider a problem with a linear objective and a joint nonlinear chance constraint, and DCA requires substantially longer runtimes in this setting, which we also observe in our numerical experiments in Section~\ref{sec:numerical-nonlinear}.
\end{itemize}

To address the first two limitations, we propose a penalty based approach for the DC program~\eqref{p:ccp-dc}, which has been extensively studied in the literature~\citep{lipp2016variations, le2012exact} and does not require an initial feasible solution $\bs{x}^0\in\hat\X$. Let $\sigma > 0$ be the penalty parameter. We consider the following penalized problem:
\begin{equation} \label{p:ccp-dc-pen}
    \min_{\bs{x}\in\X}\ \bigl\{P_\sigma(\bs{x}):=f(\bs{x})+\sigma[G(\bs{x})]_+\bigr\}.
\end{equation}

As established in \citet[Proposition~2.1]{horst1999dc}, the $\ell_1$ exact penalty associated with a DC constraint admits an explicit DC decomposition. Accordingly, Problem~\eqref{p:ccp-dc-pen} can be written as
\begin{equation}\label{p:ccp-dc-pen-explicit}
    \min_{\bs{x} \in \X} \quad
    \left\{
        f(\bs{x}) + \sigma \max\{G_1(\bs{x}), G_2(\bs{x})\}
        - \sigma G_2(\bs{x})
    \right\}.
\end{equation}

For a fixed $\sigma > 0$, we can apply a proximal DC algorithm to Problem~\eqref{p:ccp-dc-pen-explicit}. As in the standard DCA, at iteration $k$, we select a subgradient $\bs{n}^k \in \partial G_2(\bs{x}^k)$ using Proposition~\ref{prop:G1G2-subdiff} and solve the following subproblem:
\begin{equation} \label{p:ccp-dc-pen-explicit-sub}
    \min_{\bs{x} \in \X} \left\{ f(\bs{x}) + \sigma \max\{G_1(\bs{x}), G_2(\bs{x})\} - \sigma \langle \bs{n}^k, \bs{x}\rangle + \frac{\rho}{2}\|\bs{x}-\bs{x}^k\|^2 \right\}.
\end{equation}

We remark that the subproblem~\eqref{p:ccp-dc-pen-explicit-sub} is feasible under the mild assumption that the set $\X$ is nonempty. Here $\rho\ge 0$ is a proximal parameter, and taking $\rho=0$ recovers the unregularized DCA subproblem. By introducing an auxiliary epigraph variable $\tau \in \Re$ and using the dual representation of CVaR, we obtain the following reformulation:
\begin{equation} \label{p:ccp-dc-pen-explicit-sub-equiv}
    \begin{aligned}
        \min_{\bs{x}\in \X,\ \tau,\ \eta_1,\eta_2,\ \bs{u}\in\Re_+^S,\ \bs{v}\in\Re_+^S}\quad
        & f(\bs{x}) + \sigma \tau - \sigma \langle \bs{n }^k,\bs{x}\rangle + \frac{\rho}{2}\|\bs{x} - \bs{x}^k\|^2\\
        \text{s.t.}\quad
        & \tau \ \ge\ (m+1)\eta_1 + \sum_{s=1}^S u_s, \quad u_s \ \ge\ g_s(\bs{x}) - \eta_1,\quad \forall s\in[S], \\
        & \tau \ \ge\ m\,\eta_2 + \sum_{s=1}^S v_s, \quad v_s \ \ge\ g_s(\bs{x}) - \eta_2,\quad  \forall s\in[S].
    \end{aligned}
\end{equation}

The complete framework for Problem~\eqref{p:ccp-dc-pen-explicit} is presented in Algorithm~\ref{alg:penalty-dc-primal}. For a given penalty parameter $\sigma>0$, the inner loop, Algorithm~\ref{alg:pdca-primal}, repeatedly applies the update~\eqref{p:ccp-dc-pen-explicit-sub} to find a critical point of the penalized DC problem. Given $\varepsilon\ge0$, we say $\bs{x}\in\X$ is an \emph{$\varepsilon$-critical point} of Problem~\eqref{p:ccp-dc-pen-explicit} if there exist $\bs{x}'\in\X$, $\bs{n}\in\partial G_2(\bs{x}')$, $\bs{h}\in\partial(\max\{G_1,G_2\})(\bs{x})$, and $\bs{\nu}\in\N_{\X}(\bs{x})$ such that
\begin{equation}\label{eq:criticality-dc-pen-primal}
    \sigma\bigl[G_2(\bs{x})-G_2(\bs{x}')-\langle\bs{n},\bs{x}-\bs{x}'\rangle\bigr]\le\varepsilon,
    \qquad
    \bigl\|\nabla f(\bs{x})+\sigma\bs{h}-\sigma\bs{n}+\bs{\nu}\bigr\|\le\varepsilon.
\end{equation}
When $\varepsilon=0$, we call $\bs{x}$ a \emph{critical point} of Problem~\eqref{p:ccp-dc-pen-explicit}. The following lemma shows that the iterates generated by Algorithm~\ref{alg:pdca-primal} produce an $\varepsilon$-critical point of Problem~\eqref{p:ccp-dc-pen-explicit} after finitely many iterations for any $\varepsilon>0$.

\begin{lemma}\label{lem:primal-fixed-penalty}
    Suppose that Assumptions~\ref{ass:X-convex} and~\ref{ass:f-g-convex} hold. Fix $\sigma>0$ and $\rho\ge0$, and let $\{\bs{x}^k\}_{k\geq0}$ be the sequence generated by Algorithm~\ref{alg:pdca-primal}. Then for every $\varepsilon>0$ there exists $K$ such that $\bs{x}^{k+1}$ is an $\varepsilon$-critical point of Problem~\eqref{p:ccp-dc-pen-explicit} for all $k\ge K$.
\end{lemma}

\begin{proof}
Write $\phi(\bs{x}):=f(\bs{x})+\sigma\max\{G_1(\bs{x}),G_2(\bs{x})\}$, so that $P_\sigma=\phi-\sigma G_2$ on $\X$, and let $\bs{\Delta}^k:=\bs{x}^{k+1}-\bs{x}^k$. By Step~3, $\bs{x}^{k+1}$ minimizes the $\rho$-strongly convex function $\phi(\bs{x})-\sigma\langle\bs{n}^k,\bs{x}\rangle+\frac{\rho}{2}\|\bs{x}-\bs{x}^k\|^2$ over $\X$. Comparing its values at $\bs{x}^k$ and $\bs{x}^{k+1}$ and subtracting $\sigma G_2(\bs{x}^{k+1})-\sigma G_2(\bs{x}^k)$ from both sides gives
\begin{equation}\label{eq:primal-inner-descent}
    P_\sigma(\bs{x}^{k+1})\ \le\ P_\sigma(\bs{x}^k)-e_k-\rho\|\bs{\Delta}^k\|^2,
    \qquad
    e_k:=\sigma\bigl[G_2(\bs{x}^{k+1})-G_2(\bs{x}^k)-\langle\bs{n}^k,\bs{\Delta}^k\rangle\bigr].
\end{equation}
Since $\bs{n}^k\in\partial G_2(\bs{x}^k)$ and $G_2$ is convex on $\X$ by Proposition~\ref{prop:G1G2-subdiff}, we have $e_k\ge0$. As $P_\sigma$ is continuous on the compact set $\X$, summing~\eqref{eq:primal-inner-descent} over $k$ yields $\sum_{k\ge0}(e_k+\rho\|\bs{\Delta}^k\|^2)<\infty$. Hence $e_k\to0$ and $\rho\|\bs{\Delta}^k\|=\sqrt{\rho}\,(\rho\|\bs{\Delta}^k\|^2)^{1/2}\to0$, so $\max\{e_k,\rho\|\bs{\Delta}^k\|\}\le\varepsilon$ for all sufficiently large $k$.

Fix such a $k$. The first-order optimality condition of the convex subproblem in Step~3 gives $\bs{h}^{k+1}\in\partial(\max\{G_1,G_2\})(\bs{x}^{k+1})$ and $\bs{\nu}^{k+1}\in\N_{\X}(\bs{x}^{k+1})$ such that $\nabla f(\bs{x}^{k+1})+\sigma\bs{h}^{k+1}-\sigma\bs{n}^k+\bs{\nu}^{k+1}=-\rho\bs{\Delta}^k$. Taking $\bs{x}=\bs{x}^{k+1}$, $\bs{x}'=\bs{x}^k$, $\bs{n}=\bs{n}^k$, $\bs{h}=\bs{h}^{k+1}$, and $\bs{\nu}=\bs{\nu}^{k+1}$ in~\eqref{eq:criticality-dc-pen-primal}, the first inequality reads $e_k\le\varepsilon$, and the second follows from $\rho\|\bs{\Delta}^k\|\le\varepsilon$. Hence $\bs{x}^{k+1}$ is an $\varepsilon$-critical point of Problem~\eqref{p:ccp-dc-pen-explicit}.
\end{proof}

The outer loop, Algorithm~\ref{alg:penalty-dc-primal}, calls Algorithm~\ref{alg:pdca-primal} to solve Problem~\eqref{p:ccp-dc-pen-explicit} with $\sigma=\sigma^t$ in its $t$th iteration and then increases the penalty parameter by a constant factor $\beta>1$. Following \citet[Chapter~17]{nocedal2006numerical}, we start with a small initial penalty parameter $\sigma^1$ and increase it progressively to penalize constraint violations more heavily.

\begin{algorithm}[ht]
    \caption{Proximal DCA for Problem~\eqref{p:ccp-dc-pen-explicit}}
    \label{alg:pdca-primal}
    \begin{algorithmic}[1]
        \Input $\sigma>0$, $\rho\ge0$, $\bs{x}^0\in\X$.
        \For{$k = 0,1,2,\dots$}
        \State Find $\bs{n}^k \in \partial G_2(\bs{x}^k)$ via Proposition~\ref{prop:G1G2-subdiff}.
        \State Solve subproblem~\eqref{p:ccp-dc-pen-explicit-sub-equiv} to obtain $\bs{x}^{k+1}$.
        \EndFor
    \end{algorithmic}
\end{algorithm}

\begin{algorithm}[ht]
    \caption{Penalty Based Proximal DC Algorithm in the Primal Space for Problem~\eqref{p:ccp-dc}}
    \label{alg:penalty-dc-primal}
    \begin{algorithmic}[1]
        \Input $\bs{x}^0\in\X$, $\sigma^1>0$, $\rho \geq 0$, $\beta>1$.
        \For{$t = 1,2,\dots$}
        \State Run Algorithm~\ref{alg:pdca-primal} on Problem~\eqref{p:ccp-dc-pen-explicit} with input $(\sigma^t,\rho,\bs{x}^{t-1})$ to obtain $\bs{x}^t$.
        \State Update $\sigma^{t+1}\gets \beta\,\sigma^t$.
        \EndFor
    \end{algorithmic}
\end{algorithm}

We next characterize the limit points of the sequence generated by Algorithm~\ref{alg:penalty-dc-primal} as $\sigma^t\to\infty$ in the polyhedral setting. Following \citet[Theorem~6.1.1 and Section~6.3]{clarke1990optimization}, we call $\bar{\bs{x}}\in\hat\X$ a \emph{stationary point} of Problem~\eqref{p:ccp-dc} if there exists $\bar\lambda\ge0$ such that
\begin{equation}\label{eq:stationarity-dc-primal}
    \bs{0}\in\nabla f(\bar{\bs{x}})+\bar\lambda\,\partial G(\bar{\bs{x}})+\N_{\X}(\bar{\bs{x}}),
    \qquad
    \bar\lambda\, G(\bar{\bs{x}})=0.
\end{equation}
The following theorem characterizes the feasible and infeasible limit points of the outer sequence under the $\varepsilon_t$-criticality condition in~\eqref{eq:criticality-dc-pen-primal}.

\begin{theorem}\label{thm:primal-outer}
    Suppose that Assumptions~\ref{ass:X-convex} and~\ref{ass:f-g-convex} hold, $\X$ is a polytope, each $h_{i,s}$ is affine, $\sigma^t\to\infty$, and $\varepsilon_t\downarrow0$. If, for every $t\ge1$, $\bs{x}^t$ is an $\varepsilon_t$-critical point of Problem~\eqref{p:ccp-dc-pen-explicit} with $\sigma=\sigma^t$, then every limit point $\bar{\bs{x}}$ of $\{\bs{x}^t\}_{t\ge1}$ satisfies the following statements.
    \begin{itemize}
        \item[(a)] If $\bar{\bs{x}}\notin\hat\X$, then there exist $\bar{\bs{h}}\in\partial G_1(\bar{\bs{x}})$ and $\bar{\bs{n}}\in\partial G_2(\bar{\bs{x}})$ such that $\bs{0}\in\bar{\bs{h}}-\bar{\bs{n}}+\N_{\X}(\bar{\bs{x}})$, and consequently
        \[
            \bigl\{\bs{x}\in\X:\ G_1(\bs{x})\le G_2(\bar{\bs{x}})+\langle\bar{\bs{n}},\bs{x}-\bar{\bs{x}}\rangle\bigr\}=\emptyset.
        \]
        \item[(b)] If $\bar{\bs{x}}\in\hat\X$ and $G_2$ is differentiable at $\bar{\bs{x}}$, then $\bar{\bs{x}}$ is a stationary point of Problem~\eqref{p:ccp-dc}.
    \end{itemize}
\end{theorem}

\begin{proof}
Passing to a subsequence, we may assume that $\bs{x}^t\to\bar{\bs{x}}$. Since each $h_{i,s}$ is affine, $G_2$ is the pointwise maximum of finitely many affine functions $\ell_j(\bs{x}):=\langle\bs{c}_j,\bs{x}\rangle+\kappa_j$, $j\in[q]$, defined on $\Re^d$, and $\partial G_2(\bs{x})=\conv\{\bs{c}_j:j\in\A(\bs{x})\}$, where $\A(\bs{x}):=\{j\in[q]:\ell_j(\bs{x})=G_2(\bs{x})\}$. Since $\bs{x}^t$ is an $\varepsilon_t$-critical point of Problem~\eqref{p:ccp-dc-pen-explicit} with $\sigma=\sigma^t$, there exist $\bs{x}'^t\in\X$, $\bs{n}^t\in\partial G_2(\bs{x}'^t)$, $\bs{h}^t\in\partial(\max\{G_1,G_2\})(\bs{x}^t)$, and $\bs{\nu}^t\in\N_{\X}(\bs{x}^t)$ such that
\[
    \sigma^t\bigl[G_2(\bs{x}^t)-G_2(\bs{x}'^t)-\langle\bs{n}^t,\bs{x}^t-\bs{x}'^t\rangle\bigr]\le\varepsilon_t,
    \qquad
    \bigl\|\nabla f(\bs{x}^t)+\sigma^t(\bs{h}^t-\bs{n}^t)+\bs{\nu}^t\bigr\|\le\varepsilon_t.
\]

We first replace $\bs{n}^t$ by a subgradient of $G_2$ at $\bar{\bs{x}}$. Write $\bs{n}^t=\sum_{j\in[q]}\lambda^t_j\bs{c}_j$, where $\lambda^t_j\ge0$, $\sum_{j\in[q]}\lambda^t_j=1$, and $\lambda^t_j=0$ for $j\notin\A(\bs{x}'^t)$. Since $\ell_j(\bs{x}'^t)=G_2(\bs{x}'^t)$ for $j\in\A(\bs{x}'^t)$ and each $\ell_j$ is affine, $G_2(\bs{x}'^t)+\langle\bs{n}^t,\bs{x}^t-\bs{x}'^t\rangle=\sum_{j\in[q]}\lambda^t_j\ell_j(\bs{x}^t)$. Hence the first inequality above becomes
\[
    \sum_{j\in[q]}\lambda^t_j\bigl(G_2(\bs{x}^t)-\ell_j(\bs{x}^t)\bigr)\ \le\ \frac{\varepsilon_t}{\sigma^t},
\]
where each summand is nonnegative. If $\A(\bar{\bs{x}})=[q]$, set $\hat{\bs{n}}^t:=\bs{n}^t\in\partial G_2(\bar{\bs{x}})$. Otherwise, let $\delta:=\min_{j\notin\A(\bar{\bs{x}})}\{G_2(\bar{\bs{x}})-\ell_j(\bar{\bs{x}})\}>0$. By continuity, $G_2(\bs{x}^t)-\ell_j(\bs{x}^t)\ge\delta/2$ for all $j\notin\A(\bar{\bs{x}})$ and all sufficiently large $t$, so $\mu_t:=\sum_{j\notin\A(\bar{\bs{x}})}\lambda^t_j\le 2\varepsilon_t/(\delta\sigma^t)\to0$. For all sufficiently large $t$, we have $\mu_t<1$, and
\[
    \hat{\bs{n}}^t:=\frac{1}{1-\mu_t}\sum_{j\in\A(\bar{\bs{x}})}\lambda^t_j\bs{c}_j\ \in\ \conv\{\bs{c}_j:j\in\A(\bar{\bs{x}})\}=\partial G_2(\bar{\bs{x}}).
\]
Moreover, $\bs{n}^t-\hat{\bs{n}}^t=\sum_{j\notin\A(\bar{\bs{x}})}\lambda^t_j\bs{c}_j-\mu_t\hat{\bs{n}}^t$, so $\sigma^t\|\bs{n}^t-\hat{\bs{n}}^t\|\le 2c_{\max}\sigma^t\mu_t\le4c_{\max}\varepsilon_t/\delta$, where $c_{\max}:=\max_{j\in[q]}\|\bs{c}_j\|$. In both cases, $\hat{\bs{n}}^t\in\partial G_2(\bar{\bs{x}})$ and $\sigma^t\|\bs{n}^t-\hat{\bs{n}}^t\|\to0$, so
\begin{equation}\label{eq:primal-outer-crit}
    \nabla f(\bs{x}^t)+\sigma^t(\bs{h}^t-\hat{\bs{n}}^t)+\bs{\nu}^t=\bs{r}^t,
    \qquad
    \|\bs{r}^t\|\le\varepsilon_t+\sigma^t\|\bs{n}^t-\hat{\bs{n}}^t\|\to0.
\end{equation}

\smallskip
\noindent\textbf{(a).} Suppose that $\bar{\bs{x}}\notin\hat\X$, i.e., $G_1(\bar{\bs{x}})>G_2(\bar{\bs{x}})$. The function $\max\{G_1,G_2\}$ is finite and convex on $\U$. Hence its subdifferentials are uniformly bounded on the compact set $\X$ and have closed graphs. Passing to a further subsequence, we have $\bs{h}^t\to\bar{\bs{h}}\in\partial(\max\{G_1,G_2\})(\bar{\bs{x}})$ and, since $\partial G_2(\bar{\bs{x}})$ is compact, $\hat{\bs{n}}^t\to\bar{\bs{n}}\in\partial G_2(\bar{\bs{x}})$. By continuity, $\max\{G_1,G_2\}=G_1$ on a neighborhood of $\bar{\bs{x}}$, so $\bar{\bs{h}}\in\partial G_1(\bar{\bs{x}})$. Dividing~\eqref{eq:primal-outer-crit} by $\sigma^t$ and using the boundedness of $\nabla f$ on $\X$ and $\bs{r}^t\to\bs{0}$, we obtain $\bs{\nu}^t/\sigma^t\to\bar{\bs{n}}-\bar{\bs{h}}$. Since $\bs{\nu}^t/\sigma^t\in\N_{\X}(\bs{x}^t)$ and the graph of $\N_{\X}$ is closed, $\bar{\bs{n}}-\bar{\bs{h}}\in\N_{\X}(\bar{\bs{x}})$, which proves the first claim. Consequently, $\bar{\bs{x}}$ minimizes the convex function $G_1(\bs{x})-\langle\bar{\bs{n}},\bs{x}\rangle$ over $\X$. Hence, for every $\bs{x}\in\X$,
\[
    G_1(\bs{x})-G_2(\bar{\bs{x}})-\langle\bar{\bs{n}},\bs{x}-\bar{\bs{x}}\rangle
    \ \ge\ G_1(\bar{\bs{x}})-G_2(\bar{\bs{x}})=G(\bar{\bs{x}})>0,
\]
and the set in part~(a) is empty.

\smallskip
\noindent\textbf{(b).} Define $\Gamma(\bar{\bs{x}}):=\partial(\max\{G_1,G_2\})(\bar{\bs{x}})-\partial G_2(\bar{\bs{x}})$. Since each $h_{i,s}$ is affine, both $\max\{G_1,G_2\}$ and $G_2$ are pointwise maxima of finitely many affine functions. For such a function $\phi(\bs{x})=\max_{j\in[p]}\{\langle\bs{a}_j,\bs{x}\rangle+b_j\}$, we have $\partial\phi(\bs{x})=\conv\{\bs{a}_j:j\in \A_\phi(\bs{x})\}$, where $\A_\phi(\bs{x})$ is the set of indices attaining the maximum at $\bs{x}$. By continuity, $\A_\phi(\bs{x})\subseteq \A_\phi(\bar{\bs{x}})$, and hence $\partial\phi(\bs{x})\subseteq\partial\phi(\bar{\bs{x}})$, for all $\bs{x}$ in a neighborhood of $\bar{\bs{x}}$. Similarly, since $\X$ is a polytope, every constraint of $\X$ that is active at a point $\bs{x}\in\X$ near $\bar{\bs{x}}$ is also active at $\bar{\bs{x}}$, so $\N_{\X}(\bs{x})\subseteq\N_{\X}(\bar{\bs{x}})$. Therefore, since $\bs{h}^t\in\partial(\max\{G_1,G_2\})(\bs{x}^t)$ and $\hat{\bs{n}}^t\in\partial G_2(\bar{\bs{x}})$, for all sufficiently large $t$, \eqref{eq:primal-outer-crit} yields
\[
    -\nabla f(\bs{x}^t)+\bs{r}^t\in\sigma^t\,\Gamma(\bar{\bs{x}})+\N_{\X}(\bar{\bs{x}})
    \subseteq\mathcal{K}:=\bigl\{\lambda\bs{u}:\lambda\ge0,\ \bs{u}\in\Gamma(\bar{\bs{x}})\bigr\}+\N_{\X}(\bar{\bs{x}}).
\]
The set $\Gamma(\bar{\bs{x}})$ is a polytope. It contains $\bs{0}$, because $G_2\le\max\{G_1,G_2\}$ with equality at $\bar{\bs{x}}\in\hat\X$, so that $\partial G_2(\bar{\bs{x}})\subseteq\partial(\max\{G_1,G_2\})(\bar{\bs{x}})$. Hence the first set in the definition of $\mathcal{K}$ is the convex cone generated by the finitely many vertices of $\Gamma(\bar{\bs{x}})$. As $\N_{\X}(\bar{\bs{x}})$ is also finitely generated, $\mathcal{K}$ is a finitely generated convex cone and is therefore closed. Letting $t\to\infty$ and using the continuity of $\nabla f$ and $\bs{r}^t\to\bs{0}$, we obtain $-\nabla f(\bar{\bs{x}})\in\mathcal{K}$. That is, there exist $\bar\mu\ge0$ and $\bar{\bs{u}}\in\Gamma(\bar{\bs{x}})$ such that
\begin{equation}\label{eq:primal-outer-limit}
    -\nabla f(\bar{\bs{x}})-\bar\mu\,\bar{\bs{u}}\in\N_{\X}(\bar{\bs{x}}).
\end{equation}

It remains to express $\bar{\bs{u}}$ through $\partial G(\bar{\bs{x}})$. Since $G_2$ is convex and differentiable at $\bar{\bs{x}}$, we have $\partial G_2(\bar{\bs{x}})=\{\nabla G_2(\bar{\bs{x}})\}$, so $\bar{\bs{u}}=\bar{\bs{h}}-\nabla G_2(\bar{\bs{x}})$ for some $\bar{\bs{h}}\in\partial(\max\{G_1,G_2\})(\bar{\bs{x}})$. If $G(\bar{\bs{x}})<0$, then $\max\{G_1,G_2\}=G_2$ on a neighborhood of $\bar{\bs{x}}$ by continuity, so $\bar{\bs{h}}=\nabla G_2(\bar{\bs{x}})$ and $\bar{\bs{u}}=\bs{0}$. If $G(\bar{\bs{x}})=0$, both $G_1$ and $G_2$ attain the maximum at $\bar{\bs{x}}$, and the subdifferential formula for the pointwise maximum of finite convex functions gives $\partial(\max\{G_1,G_2\})(\bar{\bs{x}})=\conv\bigl(\partial G_1(\bar{\bs{x}})\cup\{\nabla G_2(\bar{\bs{x}})\}\bigr)$. Since $\partial G_1(\bar{\bs{x}})$ is nonempty and convex, in both cases there exist $\theta\in[0,1]$ and $\bar{\bs{v}}\in\partial G_1(\bar{\bs{x}})$ such that
\[
    \bar{\bs{u}}=\theta\bigl(\bar{\bs{v}}-\nabla G_2(\bar{\bs{x}})\bigr),
    \qquad\text{with } \theta=0 \text{ if } G(\bar{\bs{x}})<0.
\]
We next show that $\partial G_1(\bar{\bs{x}})-\nabla G_2(\bar{\bs{x}})\subseteq\partial G(\bar{\bs{x}})$. Let $\bs{v}\in\partial G_1(\bar{\bs{x}})$ and $\bs{d}\in\Re^d$. Taking $\bs{x}=\bar{\bs{x}}$ in the $\limsup$ that defines $G^\circ(\bar{\bs{x}};\bs{d})$, and using the convexity of $G_1$ and the differentiability of $G_2$ at $\bar{\bs{x}}$, we obtain
\[
    G^\circ(\bar{\bs{x}};\bs{d})
    \ \ge\ G'(\bar{\bs{x}};\bs{d})
    \ =\ G_1'(\bar{\bs{x}};\bs{d})-\langle\nabla G_2(\bar{\bs{x}}),\bs{d}\rangle
    \ \ge\ \bigl\langle\bs{v}-\nabla G_2(\bar{\bs{x}}),\bs{d}\bigr\rangle.
\]
Since $\bs{d}$ is arbitrary, $\bs{v}-\nabla G_2(\bar{\bs{x}})\in\partial G(\bar{\bs{x}})$ by the definition of the Clarke subdifferential. The reverse inclusion $\partial G(\bar{\bs{x}})\subseteq\partial G_1(\bar{\bs{x}})-\nabla G_2(\bar{\bs{x}})$ follows from \citet[Propositions~2.3.1 and~2.3.3]{clarke1990optimization}. Together, the two inclusions give $\partial G(\bar{\bs{x}})=\partial G_1(\bar{\bs{x}})-\nabla G_2(\bar{\bs{x}})$.

Now set $\bar\lambda:=\theta\bar\mu\ge0$. Then $\bar\mu\,\bar{\bs{u}}=\bar\lambda\bigl(\bar{\bs{v}}-\nabla G_2(\bar{\bs{x}})\bigr)\in\bar\lambda\,\partial G(\bar{\bs{x}})$, and~\eqref{eq:primal-outer-limit} yields $\bs{0}\in\nabla f(\bar{\bs{x}})+\bar\lambda\,\partial G(\bar{\bs{x}})+\N_{\X}(\bar{\bs{x}})$. Moreover, $\bar\lambda\,G(\bar{\bs{x}})=0$, because either $G(\bar{\bs{x}})=0$ or $\theta=0$. Hence $\bar{\bs{x}}$ satisfies~\eqref{eq:stationarity-dc-primal} and is a stationary point of Problem~\eqref{p:ccp-dc}.
\end{proof}

We make the following remarks on Algorithm~\ref{alg:penalty-dc-primal} and Theorem~\ref{thm:primal-outer}.
\begin{itemize}
    \item[(i)]
    Algorithm~\ref{alg:penalty-dc-primal} warm-starts each iteration at the point returned by the previous iteration, which often accelerates convergence in practice. In our implementation, Algorithm~\ref{alg:pdca-primal} is stopped after finitely many iterations, and its last iterate serves as $\bs{x}^t$. Moreover, at early stages when $\sigma$ is small, it is typically unnecessary to run the inner loop to high accuracy. A practical strategy is therefore to use a loose inner stopping tolerance initially and tighten it progressively as $\sigma$ increases, thereby reducing the overall computational cost.

    \item[(ii)]
    In our experiments, when oscillations occur with $\rho=0$, we use a small positive proximal parameter $\rho\in[10^{-4},10^{-3}]$ to mitigate these oscillations and improve numerical stability.

    \item[(iii)]
    The differentiability of $G_2$ at $\bar{\bs{x}}$ in Theorem~\ref{thm:primal-outer}(b) is required in the proof. Since $G_2$ is convex, it is differentiable at $\bar{\bs{x}}$ if and only if $\partial G_2(\bar{\bs{x}})$ is a singleton. By Proposition~\ref{prop:G1G2-subdiff}, this is ensured by the following two conditions:
    \begin{itemize}
        \item[(a)] (\emph{Unique top-$m$ scenarios}) The $m$ largest components of $(g_1(\bar{\bs{x}}),\ldots,g_S(\bar{\bs{x}}))$ are separated from the remaining components by a strict gap:
        \begin{equation}\label{eq:unique-top-m}
            g_{(S-m)}(\bar{\bs{x}})
            <
            g_{(S-m+1)}(\bar{\bs{x}}).
        \end{equation}
        Consequently, exactly $m$ scenarios satisfy $g_s(\bar{\bs{x}})\ge g_{(S-m+1)}(\bar{\bs{x}})$.
        \item[(b)] (\emph{Unique active piece within each selected scenario}) For every $s\in [S]$ with $g_s(\bar{\bs{x}})\ge g_{(S-m+1)}(\bar{\bs{x}})$, the maximizer in
        $g_s(\bar{\bs{x}})=\max_{i\in[I]}h_{i,s}(\bar{\bs{x}})$ is unique:
        \begin{equation}\label{eq:unique-active-piece}
            \left|\arg\max_{i\in[I]}h_{i,s}(\bar{\bs{x}})\right|=1.
        \end{equation}
    \end{itemize}
    This assumption can be dispensed with by solving each penalized subproblem with the enhanced DCA proposed in \citet{pang2017computing}, which solves one subproblem of the form~\eqref{p:ccp-dc-pen-explicit-sub} for each nearly active piece of $G_2$. For a fixed penalty parameter, every accumulation point of the inner iterates is a $\mathrm{d}$-stationary point of Problem~\eqref{p:ccp-dc-pen-explicit}~\citep[Theorem~1(iv)]{lu2019enhanced}.
    If the point returned at each penalty level satisfies condition~(3.7) of \citet{lu2022penalty} with tolerance $\eta_t\to0$, then parts~(ii) and~(iv) of their Theorem~3.1 imply that every accumulation point $\bar{\bs{x}}\in\hat\X$ of the sequence generated by Algorithm~\ref{alg:penalty-dc-primal}, with this variant in place of Algorithm~\ref{alg:pdca-primal}, is a B-stationary point of Problem~\eqref{p:ccp-dc}~\citep[see][Section~4]{pang2017computing}, even if $G_2$ is not differentiable at $\bar{\bs{x}}$. However, each iteration of this variant requires solving additional subproblems, whose number can be large when many scenarios are tied, so we do not implement it for efficiency.

\end{itemize}

Similar to the DCA of~\citet{wang2025proximal}, Algorithm~\ref{alg:penalty-dc-primal} is particularly effective in polyhedral settings. For any $\sigma>0$, Algorithm~\ref{alg:pdca-primal} with $\rho=0$ enjoys finite convergence under the conditions of \citet[Theorem~6]{tao1997convex}. Problem~\eqref{p:ccp-dc-pen-explicit} is a polyhedral DC program whenever all $h_{i,s}$ are affine, because its concave part $\sigma G_2$ is then polyhedral. Our numerical experiments in Section~\ref{sec:numerical} show that, for problems with affine constraints, Algorithm~\ref{alg:penalty-dc-primal} can produce solutions with lower objective values than DCA. The two methods have comparable runtime when the objective is linear, whereas the proposed method achieves substantial speedups when the objective is quadratic.

For problems with nonlinear constraints, however, Algorithm~\ref{alg:penalty-dc-primal} can exhibit longer running time, similar to DCA, because Algorithm~\ref{alg:pdca-primal} requires substantially more iterations at each penalty level. In our experiments, it can also fail to find a feasible solution within a reasonable computational budget, even with a larger initial penalty parameter $\sigma^1$ and a more aggressive growth factor $\beta$. These observations motivate an alternative exact penalty approach that offers improved numerical stability in nonlinear settings.

\subsection{Equivalent Formulation in the Lifted Space}
To address these difficulties, we introduce auxiliary
variables that separate the scenario constraint values from the
restriction on the number of violated scenarios. This leads to a
lifted formulation with a cardinality constraint, whose structure
we exploit in the subsequent algorithmic development. We begin
with the following equivalent reformulation of the chance constraint:
\[
\sum^S_{s=1} \ind\{g_s(\bs{x})\leq 0\} \geq S-m \Longleftrightarrow \sum^S_{s=1} \ind\{[g_s(\bs{x})]_+> 0\} \leq m.
\]

Introducing an auxiliary variable $\bs{y} \in \Re^S_+$, we obtain a reformulation of Problem~\eqref{p:ccp-saa} with a cardinality constraint:
\begin{equation*}
    v^* = \min_{\bs{x}\in\X, \bs{y} \geq \bs{0}} \left \{ f(\bs{x}): \|\bs{y}\|_0 \leq m, \; g_s(\bs{x}) \leq y_s, \; \forall s\in [S]\right\},
\end{equation*}
where the cardinality constraint $\|\bs{y}\|_0 \leq m$ enforces that the number of violated scenarios cannot exceed $m$.

The cardinality constraint admits a well-known equivalent DC representation~\citep{cui2021modern, gotoh2018dc}. That is, $\|\bs{y}\|_0 \leq m \Longleftrightarrow \|\bs{y} \|_1 - \|\bs{y}\|_{(m)} = 0$, where $\|\bs{y}\|_{(m)}$ denotes the Ky Fan norm of $\bs{y}$. \citet{gotoh2018dc} developed a penalty approach based on this representation. In our setting, the nonnegativity of $\bs{y}$ simplifies the resulting penalized problem to
\begin{equation}
    \label{p:ccp-pen-kyfan}
    \min_{\bs{x}\in\X, \bs{y} \geq \bs{0}} \left \{ f(\bs{x}) + \sigma (\bs{1}^\top \bs{y} - \kyfan_m(\bs{y})):  \; g_s(\bs{x}) \leq y_s, \; \forall s\in [S]\right\},
\end{equation}
where $\kyfan_m(\bs{y})  = \max_{\bs{u}}\{\langle \bs{u}, \bs{y} \rangle : \bs{0} \leq \bs{u} \leq \bs{1},\; \bs{1} ^\top \bs{u} \leq m\}$ denotes the sum of the $m$ largest elements of $\bs{y}$. To further simplify this formulation, define the convex polytope
\begin{equation}\label{eq:C-def}
\C:=\{\bs{z}\in\Re^S:\ 0\le z_s\le 1\ \forall s\in[S],\ \bs{1}^\top\bs{z}\ge S-m\}.
\end{equation}
Using the support function $\pi_{\C}$ of the polytope $\C$, we can equivalently express Problem~\eqref{p:ccp-pen-kyfan} as
\begin{equation}
    \label{p:ccp-card-pen}
    \min_{\bs{x}\in\X, \bs{y} \geq \bs{0}} \left \{ f(\bs{x}) - \pi_{\C}(-\sigma \bs{y}):  \; g_s(\bs{x}) \leq y_s, \; \forall s\in [S]\right\},
\end{equation}
where $\pi_{\C}(\bs{v}):=\max_{\bs{z}\in\C}\bs{z}^\top\bs{v}$. A closely related formulation was studied by \citet{curtis2018sequential}, who developed a penalty based trust-region algorithm that solves a sequence of quadratic subproblems subject to linear cardinality constraints.

Our approach instead exploits the DC structure of the objective in~\eqref{p:ccp-card-pen}. The key observation is that this problem is a polyhedral DC program because the convex component $\pi_{\C}(-\sigma\bs{y})$ is polyhedral. Moreover, the Fenchel conjugate of $\pi_{\C}$ is simply the indicator function of $\C$. This structure motivates a dual reformulation via Toland duality for DC programs~\citep{toland1978duality,horst1999dc,banert2019general}, which transfers the nonsmoothness of the objective into simple polyhedral constraints. The resulting dual problem minimizes a concave function over a polyhedron, a structure that enables finite convergence of DCA under suitable conditions. This structure is preserved even when the original objective and constraints are nonlinear and may therefore improve numerical stability in these settings. We formally state this in the following proposition.

\begin{proposition}\label{prop:ts-dual-card-pen}
    Suppose that Assumptions~\ref{ass:X-convex}--\ref{ass:f-g-convex} hold. Consider Problem~\eqref{p:ccp-card-pen} in the unconstrained form $
    \min_{(\bs{x},\bs{y})\in\Re^d\times\Re^S}
    \Bigl\{\zeta(\bs{x},\bs{y})-\eta(\bs{x},\bs{y})\Bigr\},
    $ where
    \begin{align*}
        \zeta(\bs{x},\bs{y})
        := f(\bs{x})
        +\delta_{\X}(\bs{x})
        +\delta_{\Re^S_+}(\bs{y})
        +\delta_{\{(\bs{x},\bs{y}):\, g_s(\bs{x})\le y_s,\ \forall s\in[S]\}}(\bs{x},\bs{y}), \quad \eta(\bs{x},\bs{y})
        := \pi_{\C}(-\sigma \bs{y}).
    \end{align*}
    Let $(\bs{p},\bs{q})\in\Re^d\times\Re^S$ denote the dual variables associated with $(\bs{x},\bs{y})$. Then the Fenchel conjugates of $\eta$ and $\zeta$ satisfy
    \begin{align}
    \eta^*(\bs{p},\bs{q})&=\delta_{\{\bs{0}\}}(\bs{p})+\inf\{\delta_{\C}(\bs{z}):\bs{z}\in\Re^S,\ -\sigma\bs{z}=\bs{q}\},\label{eq:eta-conjugate}\\
    \zeta^*(\bs{p},\bs{q})&=\sup_{\substack{\bs{x}\in\X,\ \bs{y}\ge\bs{0}\\ g_s(\bs{x})\le y_s,\ \forall s\in[S]}}\{\langle \bs{p},\bs{x}\rangle+\langle \bs{q},\bs{y}\rangle-f(\bs{x})\}.\label{eq:zeta-conjugate}
    \end{align}

    \smallskip
    Consequently, by parameterizing the dual variable of the support function directly by $\bs{z}\in\C$, the Toland dual of Problem~\eqref{p:ccp-card-pen} can be written as
    \begin{equation}\label{eq:ts-dual-card-pen}
        \min_{\bs{z}\in\C}\ \Bigl(-\zeta^*(\bs{0},-\sigma \bs{z})\Bigr)
        =
        \min_{\bs{z}\in\C}\ \inf_{\substack{\bs{x}\in\X,\ \bs{y}\ge\bs{0}\\ g_s(\bs{x})\le y_s,\ \forall s\in[S]}}
        \Bigl\{
        f(\bs{x})+\sigma \bs{z}^\top \bs{y}
        \Bigr\}.
    \end{equation}
\end{proposition}

\begin{proof}
We derive the two Fenchel conjugates separately and then obtain the Toland dual formulation.

\smallskip
\noindent\textbf{(i): Computing $\eta^*$.} Since $\eta$ does not depend on $\bs{x}$, taking the supremum over $\bs{x}$ in the definition of the Fenchel conjugate gives
\[
\eta^*(\bs{p},\bs{q})=\delta_{\{\bs{0}\}}(\bs{p})+\sup_{\bs{y}\in\Re^S}\{\langle\bs{q},\bs{y}\rangle-\pi_{\C}(-\sigma\bs{y})\}.
\]
As $\sigma>0$, the change of variables $\bs{u}=-\sigma\bs{y}$ yields
\[
\sup_{\bs{y}\in\Re^S}\{\langle\bs{q},\bs{y}\rangle-\pi_{\C}(-\sigma\bs{y})\}=\pi_{\C}^*(-\bs{q}/\sigma)=\delta_{\C}(-\bs{q}/\sigma),
\]
where the last equality follows because $\C$ is nonempty, closed, and convex. This gives~\eqref{eq:eta-conjugate}.

\smallskip
\noindent\textbf{(ii): Computing $\zeta^*$.} By the definition of the Fenchel conjugate, the indicator functions in $\zeta$ restrict the supremum to the feasible set, yielding~\eqref{eq:zeta-conjugate}.

\smallskip
The effective domain of $\eta^*$ consists of pairs $(\bs{0},-\sigma\bs{z})$ with $\bs{z}\in\C$. At every such pair, $\zeta^*$ is finite because $f$ is bounded on $\X$ and $\bs{y},\bs{z}\ge\bs{0}$. Combining~\eqref{eq:eta-conjugate} and~\eqref{eq:zeta-conjugate} yields the Toland dual formulation~\eqref{eq:ts-dual-card-pen}.
\end{proof}

Expanding the feasible set $\C$ in~\eqref{eq:ts-dual-card-pen} by its definition in~\eqref{eq:C-def} and merging the two minimizations, we obtain the following problem in the lifted variables $(\bs{x},\bs{y},\bs{z})$:
\begin{equation} \label{p:ccp-bilinear-pen}
    v^*_\sigma = \min_{\bs{x} \in \X, \bs{y}\in\Re^S_+, \bs{z}\in[0,1]^S} \left\{ \underbrace{f(\bs{x}) + \sigma V(\bs{y}, \bs{z})}_{=:F_\sigma(\bs{x},\bs{y},\bs{z})}: \sum^S_{s=1}z_s \geq S-m, \; g_s(\bs{x}) \leq y_s, \;\forall s \in [S] \right\},
\end{equation}
where $V(\bs{y}, \bs{z}) := \sum_{s=1}^S y_sz_s$ and $\sigma > 0$ is the penalty parameter. The nonnegative slack variable $y_s$ upper bounds the violation $[g_s(\bs{x})]_+$ of scenario $s$, and the weight $z_s\in[0,1]$ indicates to what extent the scenario $s$ is required to be satisfied. The constraint $\sum_{s=1}^S z_s\ge S-m$ requires the total weight to be at least $S-m$, and the bilinear term $V(\bs{y},\bs{z})$ penalizes violations of the weighted scenarios. We take~\eqref{p:ccp-bilinear-pen} as the starting point of our analysis because it isolates the nonconvexity in the bilinear term $V(\bs{y},\bs{z})$ and leads directly to a concave value-function formulation in $\bs{z}$.

Problem~\eqref{p:ccp-bilinear-pen} is closely related to the following problem with complementarity constraints:
\begin{equation} \label{p:ccp-bilinear}
    v^* =  \min_{\bs{x} \in \X, \bs{y}\in\Re^S_+, \bs{z} \in [0,1]^S} \left\{ f(\bs{x}): \sum^S_{s=1}z_s \geq S-m, \; V(\bs{y}, \bs{z}) = 0, \; g_s(\bs{x}) \leq y_s, \;\forall s \in [S]  \right\}.
\end{equation}
\citet[Proposition~1]{jiang2022also} show that Problem~\eqref{p:ccp-bilinear}, which is an MPCC~\citep{scheel2000mathematical}, is equivalent to the SAA problem~\eqref{p:ccp-saa}. Moreover, Problem~\eqref{p:ccp-bilinear-pen} is exactly the penalty formulation of~\eqref{p:ccp-bilinear} obtained by moving the complementarity constraint $V(\bs{y},\bs{z})=0$ into the objective. \citet{jiang2022also} proposed an algorithm named ALSO-X+ to solve Problem~\eqref{p:ccp-bilinear} via bisection and alternating minimization. Although ALSO-X+ can outperform the CVaR approximation~\citep{rockafellar2002conditional}, it lacks a stationarity guarantee and the bisection procedure can be computationally expensive. The algorithm developed below can address the limitations of both Algorithm~\ref{alg:penalty-dc-primal} and ALSO-X+.

\subsection{Algorithm in the Lifted Space}

In this subsection, we develop an alternative exact penalty framework for solving the lifted MPCC formulation~\eqref{p:ccp-bilinear} through its penalized counterpart~\eqref{p:ccp-bilinear-pen}. For any fixed penalty parameter $\sigma>0$, the penalized model~\eqref{p:ccp-bilinear-pen} is still nonconvex because of the bilinear term $\sum_{s=1}^S y_s z_s$ appearing in the objective. Nevertheless, from Proposition~\ref{prop:ts-dual-card-pen}, we observe that once the remaining variables $(\bs{x},\bs{y})$ are minimized out, this nonconvexity induces a concave minimization structure in the dual variable $\bs{z}$.

Before proceeding, we introduce several pieces of notation used throughout the remainder of this paper. Using the scenario-wise representation $g_s(\bs{x})=\max_{i\in[I]}h_{i,s}(\bs{x})$, we write the feasible region of Problem~\eqref{p:ccp-bilinear-pen} as
\[
\Omega_0 :=\{(\bs{x},\bs{y},\bs{z})\in\X\times\Re_+^S\times\C:\ h_{i,s}(\bs{x})\le y_s,\ \forall s\in[S],\ \forall i\in[I]\}.
\]
Similarly, the feasible region of Problem~\eqref{p:ccp-bilinear} is
\[
\Omega:=\{(\bs{x},\bs{y},\bs{z})\in\Omega_0:\ y_s z_s=0,\ \forall s\in[S]\}.
\]
We next isolate the dependence on $\bs{z}$ by introducing the following value function.

\begin{proposition}
    \label{prop:Psi-convex}
    For each $\sigma>0$, let $\Psi_\sigma:\Re^S_+\to\Re$ be defined by
    \begin{equation*}
        \Psi_\sigma(\bs{z})
        :=
        \max_{\bs{x}\in\X,\bs{y}\in\Re^S_+}
        \left\{
        -f(\bs{x})-\sigma\sum_{s=1}^S y_s z_s
        :\ g_s(\bs{x})\le y_s,\  \forall s\in[S]
        \right\}.
    \end{equation*}
    Under Assumptions~\ref{ass:X-convex} and~\ref{ass:f-g-convex}, the function $\Psi_\sigma(\bs{z})$ is convex in $\bs{z}$ for every fixed $\sigma>0$.
\end{proposition}

\begin{proof}
For any $\bs{z}\ge\bs{0}$, taking $y_s=[g_s(\bs{x})]_+$ and using continuity of $f$ and each $g_s$ together with compactness of $\X$ shows that the maximum defining $\Psi_\sigma$ is finite and attained. Since the feasible set is independent of $\bs{z}$, $\Psi_\sigma$ is convex as a pointwise supremum of affine functions of $\bs{z}$.
\end{proof}

We remark here that the convexity of $\Psi_\sigma$ in $\bs{z}$ does not depend on the convexity of the function $f(\bs{x})$ or of the set $\X$ as long as the maximum is attained. Therefore, the algorithm framework can be naturally extended to the discrete setting, e.g., when $\X = \{0,1\}^d$, although the inner problem may then require solving a mixed-integer program. With this notation, we can rewrite the penalized problem~\eqref{p:ccp-bilinear-pen} as an optimization problem in $\bs{z}$ only. Indeed, for any fixed $\bs{z}$, the inner minimization over $(\bs{x},\bs{y})$ in~\eqref{p:ccp-bilinear-pen} equals $-\Psi_\sigma(\bs{z})$, and hence Problem~\eqref{p:ccp-bilinear-pen} is equivalent to the following DC program
\begin{equation}\label{p:ccp-concave}
    \min_{\bs{z}\in\C}\ - \; \Psi_\sigma(\bs{z}).
\end{equation}

To implement a proximal DC method for~\eqref{p:ccp-concave}, we need a computable subgradient of $\Psi_\sigma(\cdot)$ to linearize the concave part. The following proposition shows that such a subgradient can be obtained from an optimal solution of the inner problem. Here, $\partial\Psi_\sigma(\bs{z})$ denotes the convex subdifferential of $\Psi_\sigma$ relative to $\Re^S_+$.

\begin{proposition}
    \label{prop:Psi-subgrad}
    For any $\bs{z}\in\C$ and $\sigma>0$, let $(\bs{x}^*(\bs{z},\sigma),\bs{y}^*(\bs{z},\sigma))$ be an optimal solution to the following subproblem:
    \begin{equation} \label{p:inner-min-fixed-z}
        (\bs{x}^*(\bs{z},\sigma),\bs{y}^*(\bs{z},\sigma))
        \in
        \arg\min_{\bs{x}\in\X,\bs{y}\in\Re^S_+}
        \left\{
        f(\bs{x}) + \sigma \sum_{s=1}^S y_s z_s :
        g_s(\bs{x})\le y_s,\ \forall s\in[S]
        \right\}.
    \end{equation}
    Let $\bs{n}(\bs{z},\sigma) := -\sigma\bigl([g_1(\bs{x}^*(\bs{z},\sigma))]_+,\dots,[g_S(\bs{x}^*(\bs{z},\sigma))]_+\bigr)^\top.$ Then we have $\bs{n}(\bs{z},\sigma)\in\partial \Psi_\sigma(\bs{z})$.
\end{proposition}

\begin{proof}
Fix $\bs{z}\in\C$ and $\sigma>0$, and write
$(\bs{x}^*,\bs{y}^*)
=(\bs{x}^*(\bs{z},\sigma),\bs{y}^*(\bs{z},\sigma))$.
Define
\[
\bs{a}
:=\bigl([g_1(\bs{x}^*)]_+,\ldots,
        [g_S(\bs{x}^*)]_+\bigr)^\top.
\]
By feasibility, $\bs{0}\le\bs{a}\le\bs{y}^*$,
and $(\bs{x}^*,\bs{a})$ is also feasible for the subproblem.
Since $\bs{z}\ge\bs{0}$, replacing $\bs{y}^*$ with $\bs{a}$
cannot increase the objective value. The optimality of
$(\bs{x}^*,\bs{y}^*)$ therefore implies that
$(\bs{x}^*,\bs{a})$ is also optimal. Hence,
\[
\Psi_\sigma(\bs{z})
=-f(\bs{x}^*)-\sigma\bs{z}^\top\bs{a}.
\]
For any $\bs{z}'\in\Re_+^S$, the pair
$(\bs{x}^*,\bs{a})$ remains feasible for the subproblem
defining $\Psi_\sigma(\bs{z}')$. Consequently,
\[
\begin{aligned}
\Psi_\sigma(\bs{z}')
\ge -f(\bs{x}^*)-\sigma(\bs{z}')^\top\bs{a}=\Psi_\sigma(\bs{z})
  +\left\langle-\sigma\bs{a},\,\bs{z}'-\bs{z}\right\rangle
=\Psi_\sigma(\bs{z})
  +\left\langle\bs{n}(\bs{z},\sigma),\,\bs{z}'-\bs{z}\right\rangle.
\end{aligned}
\]
This is precisely the subgradient inequality, proving that
$\bs{n}(\bs{z},\sigma)\in\partial \Psi_\sigma(\bs{z})$.
\end{proof}

We now apply a DC algorithm to solve the DC program~\eqref{p:ccp-concave} with fixed $\sigma >0$. Given the current iterate $\bs{z}^k\in\C$ and $\sigma>0$, we first choose a subgradient $\bs{n }^k \in \partial \Psi_\sigma(\bs{z}^k)$, which can be computed via Proposition~\ref{prop:Psi-subgrad}. Linearizing the concave term $-\Psi_\sigma(\cdot)$ at $\bs{z}^k$ and adding a proximal regularizer yields the convex subproblem
\begin{equation*}
    \begin{aligned}
        \bs{z}^{k+1}
        \in
        \arg\min_{\bs{z}\in\C}
        \left\{ -\langle \bs{n}^k,\bs{z}\rangle
        +\frac{\rho}{2}\|\bs{z}-\bs{z}^k\|^2
        \right\},
    \end{aligned}
\end{equation*}
where $\rho\geq0$ is a proximal parameter to enhance numerical stability and avoid potential oscillations. Each subproblem can be solved efficiently via sorting and scales well for large $S$.

The complete penalty framework for Problem~\eqref{p:ccp-bilinear-pen} is presented in Algorithm~\ref{alg:penalty-dc-lifted}. For a given penalty parameter $\sigma>0$, the inner loop, Algorithm~\ref{alg:pdca-lifted}, repeatedly applies the above update to solve Problem~\eqref{p:ccp-concave}. The outer loop calls Algorithm~\ref{alg:pdca-lifted} to solve the penalized Problem~\eqref{p:ccp-bilinear-pen} with $\sigma=\sigma^t$ in its $t$th iteration and then increases the penalty parameter. As in Algorithm~\ref{alg:penalty-dc-primal}, we start with a relatively small penalty parameter $\sigma^1$ and increase it geometrically. Let $\bs{w}:=(\bs{x},\bs{y},\bs{z})$. Given $\varepsilon\ge0$, a point $\bar{\bs{w}}\in\Omega_0$ is called an \emph{$\varepsilon$-stationary point} of Problem~\eqref{p:ccp-bilinear-pen}~\citep{lacoste2016convergence} if
\begin{equation}\label{eq:stationarity-pen}
    \bigl\langle\nabla F_\sigma(\bar{\bs{w}}),\,\bs{w}-\bar{\bs{w}}\bigr\rangle\ \ge\ -\varepsilon ,
    \qquad\forall\,\bs{w}\in\Omega_0.
\end{equation}
When $\varepsilon=0$, we call $\bar{\bs{w}}$ a \emph{stationary point} of Problem~\eqref{p:ccp-bilinear-pen}. The following lemma shows that the iterates of Algorithm~\ref{alg:pdca-lifted} yield $\varepsilon$-stationary points after finitely many iterations for any $\varepsilon>0$.

\begin{lemma}\label{lem:lifted-fixed-penalty}
Suppose that Assumptions~\ref{ass:X-convex} and~\ref{ass:f-g-convex} hold. Fix $\sigma>0$ and $\rho\ge0$, and let $\{(\bs{x}^k,\bs{z}^k)\}_{k\geq0}$ be the sequence generated by Algorithm~\ref{alg:pdca-lifted}. Define $\hat y^k_s:=[g_s(\bs{x}^k)]_+$ for every $s\in[S]$. Then for every $\varepsilon>0$ there exists $K$ such that $(\bs{x}^k,\hat{\bs{y}}^k,\bs{z}^k)$ is an $\varepsilon$-stationary point of Problem~\eqref{p:ccp-bilinear-pen} for all $k\ge K$.
\end{lemma}

\begin{proof}
Let $\{(\bs{x}^k,\bs{y}^k,\bs{z}^k)\}_{k\ge0}$ be generated by Algorithm~\ref{alg:pdca-lifted} with the given $\sigma>0$ and $\rho\ge0$.
Let
\[
D:=\{(\bs{x},\bs{y})\in\X\times\Re_+^S:
g_s(\bs{x})\le y_s,\ s\in[S]\},
\qquad \hat y_s^k:=[g_s(\bs{x}^k)]_+.
\]
Here $\hat{\bs{y}}^k$ may differ from the solver output $\bs{y}^k$ in Algorithm~\ref{alg:pdca-lifted} at indices with $z^k_s=0$. By the proof of Proposition~\ref{prop:Psi-subgrad},
$(\bs{x}^k,\hat{\bs{y}}^k)$ solves the fixed-$\bs{z}^k$ inner
problem and $\bs{n}^k=-\sigma\hat{\bs{y}}^k$. Define
\[
v_k:=f(\bs{x}^k)+\sigma\langle\hat{\bs{y}}^k,\bs{z}^k\rangle
=-\Psi_\sigma(\bs{z}^k).
\]
Feasibility of $(\bs{x}^k,\hat{\bs{y}}^k)$ for the inner problem
at $\bs{z}^{k+1}$ and optimality of the $\bs{z}$-update imply
\begin{equation}\label{eq:lifted-descent}
v_{k+1}\le
f(\bs{x}^k)+\sigma\langle\hat{\bs{y}}^k,\bs{z}\rangle
+\frac{\rho}{2}\|\bs{z}-\bs{z}^k\|^2,
\qquad \forall\bs{z}\in\C.
\end{equation}
Taking $\bs{z}=\bs{z}^k$ in~\eqref{eq:lifted-descent} shows that $\{v_k\}$ is nonincreasing.
Since it is bounded below by $\min_{\bs{x}\in\X}f(\bs{x})$,
we have $v_{k+1}-v_k\to0$.

Define $\theta_k:=\sigma\bigl(\langle\hat{\bs{y}}^k,\bs{z}^k\rangle-\min_{\bs{z}\in\C}\langle\hat{\bs{y}}^k,\bs{z}\rangle\bigr)\ge0$, and let $\tilde{\bs{z}}^k$ attain this minimum. For any $\tau\in(0,1]$, the point $\bs{z}^k+\tau(\tilde{\bs{z}}^k-\bs{z}^k)$ belongs to $\C$, and $\|\tilde{\bs{z}}^k-\bs{z}^k\|^2\le S$ because $\C\subseteq[0,1]^S$. Substituting this point into~\eqref{eq:lifted-descent} gives
\[
    v_{k+1}\ \le\ v_k-\tau\theta_k+\frac{\rho\tau^2S}{2},
    \qquad\text{i.e.,}\qquad
    \theta_k\ \le\ \frac{v_k-v_{k+1}}{\tau}+\frac{\rho\tau S}{2}.
\]
Since $v_k-v_{k+1}\to0$, letting $k\to\infty$ and then $\tau\downarrow0$ yields $\theta_k\to0$. Hence $\theta_k\le\varepsilon$ for all sufficiently large $k$.

Fix such a $k$ and let $\bs{w}^k:=(\bs{x}^k,\hat{\bs{y}}^k,\bs{z}^k)\in\Omega_0$. Since $(\bs{x}^k,\hat{\bs{y}}^k)$ solves the convex problem~\eqref{p:inner-min-fixed-z} with $\bs{z}=\bs{z}^k$, whose feasible set is $D$, its first-order optimality condition gives
\[
    \langle\nabla f(\bs{x}^k),\bs{x}-\bs{x}^k\rangle+\sigma\langle\bs{z}^k,\bs{y}-\hat{\bs{y}}^k\rangle\ \ge\ 0,
    \qquad\forall(\bs{x},\bs{y})\in D.
\]
Moreover, $\sigma\langle\hat{\bs{y}}^k,\bs{z}-\bs{z}^k\rangle\ge-\theta_k\ge-\varepsilon$ for all $\bs{z}\in\C$. Adding the two inequalities and using $\Omega_0=D\times\C$ and $\nabla F_\sigma(\bs{w})=(\nabla f(\bs{x}),\sigma\bs{z},\sigma\bs{y})$, we obtain~\eqref{eq:stationarity-pen} at $\bs{w}^k$. Hence $\bs{w}^k$ is an $\varepsilon$-stationary point of Problem~\eqref{p:ccp-bilinear-pen}.
\end{proof}

\begin{algorithm}[ht]
    \caption{Proximal DCA for Problem~\eqref{p:ccp-bilinear-pen}}
    \label{alg:pdca-lifted}
    \begin{algorithmic}[1]
        \Input $\sigma>0$, $\rho \geq 0$, $\bs{z}^0 \in \C$.
        \For{$k = 0, 1, 2, \dots$}
        \State Solve subproblem~\eqref{p:inner-min-fixed-z} with $\bs{z}=\bs{z}^k$ to obtain $(\bs{x}^k,\bs{y}^k)$.
        \State Set $\bs{n}^k:=-\sigma\bigl([g_1(\bs{x}^k)]_+,\ldots,[g_S(\bs{x}^k)]_+\bigr)^\top\in\partial\Psi_{\sigma}(\bs{z}^k)$ by Proposition~\ref{prop:Psi-subgrad}.
        \State Update $\bs{z}^{k+1}
        \in
        \arg\min_{\bs{z}\in\C}
        \left\{ -\langle \bs{n}^k,\bs{z}\rangle
        +\frac{\rho}{2}\|\bs{z}-\bs{z}^k\|^2
        \right\}$.
        \EndFor
    \end{algorithmic}
\end{algorithm}

\begin{algorithm}[ht]
    \caption{Penalty Based Proximal DC Algorithm in the Lifted Space for Problem~\eqref{p:ccp-bilinear}}\label{alg:penalty-dc-lifted}
    \begin{algorithmic}[1]
        \Input $\bs{z}^0 \in \C$, $\sigma^1>0$, $\rho \geq 0$, $\beta>1$.
        \For{$t = 1,2,\dots$}
        \State Run Algorithm~\ref{alg:pdca-lifted} on Problem~\eqref{p:ccp-bilinear-pen} with input $(\sigma^t,\rho,\bs{z}^{t-1})$ to obtain $(\bs{x}^t,\bs{z}^t)=(\bs{x}^k,\bs{z}^{k+1})$, where $k$ is the index of the last inner iteration at penalty level $\sigma^t$.
        \State Set $y^t_s:=[g_s(\bs{x}^t)]_+$ for every $s\in[S]$.
        \State Update $\sigma^{t+1}\gets \beta\,\sigma^t$.
        \EndFor
    \end{algorithmic}
\end{algorithm}

We now define weakly stationary points of Problem~\eqref{p:ccp-bilinear}. Our goal is to show that, when each penalized subproblem is solved to $\varepsilon_t$-stationarity as in~\eqref{eq:stationarity-pen} with $\varepsilon_t\downarrow0$, every feasible limit point of the sequence $\{\bs{w}^t\}$ generated by Algorithm~\ref{alg:penalty-dc-lifted} is a weakly stationary point of Problem~\eqref{p:ccp-bilinear}. For $\bar{\bs{w}}\in\Omega$, define the tightened set
\[
    \Omega_{\mathrm{tn}}(\bar{\bs{w}}):=\bigl\{(\bs{x},\bs{y},\bs{z})\in\Omega_0:\ y_s=0 \text{ if } \bar y_s=0,\ \ z_s=0 \text{ if } \bar z_s=0,\ \ \forall s\in[S]\bigr\}.
\]
Following the MPCC literature \citep{scheel2000mathematical,jara2018study}, we define weak stationarity by the first-order optimality condition on this tightened set.

\begin{definition}[Weak stationarity]\label{def:weak-stationarity}
    We call a point $\bar{\bs{w}}\in\Omega$ a \emph{weakly stationary point} of Problem~\eqref{p:ccp-bilinear} if
    \begin{equation}\label{eq:weak-stationarity}
        \bs{0}\in\bigl(\nabla f(\bar{\bs{x}}),\bs{0},\bs{0}\bigr)+\N_{\Omega_{\mathrm{tn}}(\bar{\bs{w}})}(\bar{\bs{w}}).
    \end{equation}
\end{definition}

By convexity under Assumptions~\ref{ass:X-convex} and~\ref{ass:f-g-convex}, condition~\eqref{eq:weak-stationarity} holds if and only if $\bar{\bs{x}}$ minimizes $f$ over $\{\bs{x}\in\X: g_s(\bs{x})\le0 \text{ for all } s\in[S] \text{ with } \bar y_s=0\}$. This definition is equivalent to the multiplier-based definitions in \citet{scheel2000mathematical} and \citet{ye2005necessary} whenever the tightened program admits KKT multipliers, e.g., when $\X$ is a polytope and each $h_{i,s}$ is affine, or under a Slater condition.

\begin{theorem}\label{thm:lifted-outer}
    Suppose that Assumptions~\ref{ass:X-convex} and~\ref{ass:f-g-convex} hold, $\sigma^t\to\infty$, and $\varepsilon_t\downarrow0$. If, for every $t\ge1$, $\bs{w}^t=(\bs{x}^t,\bs{y}^t,\bs{z}^t)$ is an $\varepsilon_t$-stationary point of Problem~\eqref{p:ccp-bilinear-pen} with $\sigma=\sigma^t$ and $y^t_s=[g_s(\bs{x}^t)]_+$ for all $s\in[S]$, then every limit point $\bar{\bs{w}}=(\bar{\bs{x}},\bar{\bs{y}},\bar{\bs{z}})$ of $\{\bs{w}^t\}_{t\ge1}$ satisfies the following statements.
    \begin{itemize}
        \item[(a)] If $\bar{\bs{x}}\notin\hat\X$, then $\bar{\bs{w}}$ is a stationary point of $\min_{\bs{w}\in\Omega_0}V(\bs{y},\bs{z})$, i.e., $\bs{0}\in(\bs{0},\bar{\bs{z}},\bar{\bs{y}})+\N_{\Omega_0}(\bar{\bs{w}})$.
        \item[(b)] If $\bar{\bs{x}}\in\hat\X$, then $\bar{\bs{w}}\in\Omega$ and $\bar{\bs{w}}$ is a weakly stationary point of Problem~\eqref{p:ccp-bilinear}.
    \end{itemize}
\end{theorem}

\begin{proof}
Let $D$ be defined as in the proof of Lemma~\ref{lem:lifted-fixed-penalty}, so that $\Omega_0=D\times\C$ and $\N_{\Omega_0}(\bs{w})=\N_D(\bs{x},\bs{y})\times\N_{\C}(\bs{z})$. Passing to a subsequence, we may assume that $\bs{w}^t\to\bar{\bs{w}}$. By continuity of each $g_s$, we have $\bar y_s=[g_s(\bar{\bs{x}})]_+$ for every $s\in[S]$. Since $\nabla F_{\sigma}(\bs{w})=(\nabla f(\bs{x}),\sigma\bs{z},\sigma\bs{y})$, the $\varepsilon_t$-stationarity of $\bs{w}^t$ means that
\[
    \langle\nabla f(\bs{x}^t),\bs{x}-\bs{x}^t\rangle+\sigma^t\langle\bs{z}^t,\bs{y}-\bs{y}^t\rangle+\sigma^t\langle\bs{y}^t,\bs{z}-\bs{z}^t\rangle\ \ge\ -\varepsilon_t,
    \qquad\forall(\bs{x},\bs{y})\in D,\ \bs{z}\in\C.
\]
Taking $\bs{z}=\bs{z}^t$ and using the convexity of $f$, and then taking $(\bs{x},\bs{y})=(\bs{x}^t,\bs{y}^t)$, we obtain
\begin{equation}\label{eq:lifted-outer-blocks}
    f(\bs{x}^t)+\sigma^t\langle\bs{z}^t,\bs{y}^t\rangle\le f(\bs{x})+\sigma^t\langle\bs{z}^t,\bs{y}\rangle+\varepsilon_t\ \ \forall(\bs{x},\bs{y})\in D,
    \qquad
    \langle\bs{y}^t,\bs{z}^t\rangle\le\langle\bs{y}^t,\bs{z}\rangle+\frac{\varepsilon_t}{\sigma^t}\ \ \forall\bs{z}\in\C.
\end{equation}

\smallskip
\noindent\textbf{(a).} By the first relation in~\eqref{eq:lifted-outer-blocks}, for every $(\bs{x},\bs{y})\in D$,
\[
    \langle\bs{z}^t,\bs{y}^t\rangle+\frac{f(\bs{x}^t)}{\sigma^t}\le\langle\bs{z}^t,\bs{y}\rangle+\frac{f(\bs{x})+\varepsilon_t}{\sigma^t}.
\]
Since $f$ is bounded on $\X$, $\varepsilon_t\to0$, and $\sigma^t\to\infty$, letting $t\to\infty$ gives $\langle\bar{\bs{z}},\bar{\bs{y}}\rangle\le\langle\bar{\bs{z}},\bs{y}\rangle$ for all $(\bs{x},\bs{y})\in D$. Letting $t\to\infty$ in the second relation and using $\varepsilon_t/\sigma^t\to0$ gives $\langle\bar{\bs{y}},\bar{\bs{z}}\rangle\le\langle\bar{\bs{y}},\bs{z}\rangle$ for all $\bs{z}\in\C$. By convexity, these two inequalities are equivalent to $\bs{0}\in(\bs{0},\bar{\bs{z}})+\N_D(\bar{\bs{x}},\bar{\bs{y}})$ and $\bs{0}\in\bar{\bs{y}}+\N_{\C}(\bar{\bs{z}})$, respectively. Together they show that every limit point $\bar{\bs{w}}$ is a stationary point of $\min_{\bs{w}\in\Omega_0}V(\bs{y},\bs{z})$, which in particular proves part~(a).

\smallskip
\noindent\textbf{(b).} Let $\J:=\{s\in[S]:g_s(\bar{\bs{x}})\le0\}=\{s\in[S]:\bar y_s=0\}$. Since $\bar{\bs{x}}\in\hat\X$, Proposition~\ref{prop:equiv-cc} gives $|\J|\ge S-m$, so at least $S-m$ entries of $\bar{\bs{y}}\ge\bs{0}$ are zero. Hence $\min_{\bs{z}\in\C}\langle\bar{\bs{y}},\bs{z}\rangle=0$, and since $\bar{\bs{z}}$ attains this minimum by the proof of part~(a), $\langle\bar{\bs{y}},\bar{\bs{z}}\rangle=0$. Since $\bar{\bs{y}},\bar{\bs{z}}\ge\bs{0}$, we obtain $\bar y_s\bar z_s=0$ for all $s\in[S]$, i.e., $\bar{\bs{w}}\in\Omega$.

We next show that $\sigma^ta_t\to0$, where $a_t:=\sum_{s\notin\J}z^t_s$. This is trivial if $\J=[S]$. Otherwise, let $\delta:=\frac12\min_{s\notin\J}\bar y_s>0$. Since $\bs{y}^t\to\bar{\bs{y}}$ and $\bar y_{s'}=0$ for $s'\in\J$, there exists $t_0$ such that $y^t_s\ge\delta$ for all $s\notin\J$ and $y^t_{s'}\le\delta/2$ for all $s'\in\J$ whenever $t\ge t_0$. Fix $t\ge t_0$, and construct $\tilde{\bs{z}}$ from $\bs{z}^t$ by setting $\tilde z_s:=0$ for $s\notin\J$ and increasing the entries $z^t_{s'}$, $s'\in\J$, without exceeding one, by a total amount $b_t:=\max\{0,\,S-m-\sum_{s'\in\J}z^t_{s'}\}$. This is possible because $|\J|\ge S-m$ implies $\sum_{s'\in\J}(1-z^t_{s'})\ge b_t$. Then $\sum_{s=1}^S\tilde z_s\ge S-m$, so $\tilde{\bs{z}}\in\C$. Moreover, $b_t\le a_t$ because $\sum_{s'\in\J}z^t_{s'}=\sum_{s=1}^Sz^t_s-a_t\ge S-m-a_t$. Hence
\[
    \langle\bs{y}^t,\bs{z}^t\rangle-\langle\bs{y}^t,\tilde{\bs{z}}\rangle
    \ \ge\ \sum_{s\notin\J}y^t_sz^t_s-\frac{\delta}{2}\,b_t
    \ \ge\ \delta a_t-\frac{\delta}{2}\,a_t
    \ =\ \frac{\delta}{2}\,a_t.
\]
Combining this with the second relation in~\eqref{eq:lifted-outer-blocks} for $\bs{z}=\tilde{\bs{z}}$ gives $\sigma^ta_t\le2\varepsilon_t/\delta\to0$.

Now fix any $\bs{x}\in\X$ with $g_s(\bs{x})\le0$ for all $s\in\J$, and set $y_s:=[g_s(\bs{x})]_+$ for $s\in[S]$. Then $(\bs{x},\bs{y})\in D$, $y_s=0$ for $s\in\J$, and $y_s\le\bar Y:=\max_{s\in[S]}\max_{\bs{x}'\in\X}[g_s(\bs{x}')]_+<\infty$. By the first relation in~\eqref{eq:lifted-outer-blocks} and the nonnegativity of $\bs{y}^t$ and $\bs{z}^t$,
\[
    f(\bs{x}^t)\le f(\bs{x}^t)+\sigma^t\langle\bs{z}^t,\bs{y}^t\rangle\le f(\bs{x})+\sigma^t\sum_{s\notin\J}z^t_sy_s+\varepsilon_t\le f(\bs{x})+\bar Y\sigma^ta_t+\varepsilon_t.
\]
Letting $t\to\infty$ yields $f(\bar{\bs{x}})\le f(\bs{x})$. Hence $\bar{\bs{x}}$ minimizes $f$ over $\{\bs{x}\in\X:g_s(\bs{x})\le0 \text{ for all } s\in\J\}$. Since $\J=\{s\in[S]:\bar y_s=0\}$, the equivalence stated after~\eqref{eq:weak-stationarity} shows that $\bar{\bs{w}}$ is a weakly stationary point of Problem~\eqref{p:ccp-bilinear}.
\end{proof}

We make the following remarks on Algorithms~\ref{alg:pdca-lifted} and~\ref{alg:penalty-dc-lifted}:
\begin{itemize}
    \item[(i)]
   For any $\sigma>0$ and $\rho=0$, Algorithm~\ref{alg:pdca-lifted} for solving~\eqref{p:ccp-concave} terminates in finitely many iterations, provided that an extreme-point optimal solution is selected at each step and the loop is stopped when its objective value no longer strictly decreases. This is because the inner update reduces to the classical DCA step $\bs{z}^{k+1}\in\arg\min_{\bs{z}\in\C}\{-\langle \bs{n}^k,\bs{z}\rangle\}$, which is a linear program over the polytope $\C$. Hence, an optimal solution can always be chosen at an extreme point of $\C$. Moreover, strict decrease prevents revisiting any extreme point, and $\C$ has only finitely many extreme points. This property is structural and does not rely on the specific form of $\Psi_\sigma$ beyond Assumptions~\ref{ass:X-convex} and~\ref{ass:f-g-convex}.

    \item[(ii)]
    Similar to Algorithm~\ref{alg:penalty-dc-primal}, when oscillations occur with $\rho=0$ in our experiments, we use a small positive proximal parameter $\rho\in[10^{-4},10^{-3}]$ to mitigate these oscillations and improve numerical stability. A larger $\rho$ will lead to more conservative updates and noticeably longer runtime.

    \item[(iii)]
    The proposed framework admits effective warm starts in both the outer loop (Algorithm~\ref{alg:penalty-dc-lifted}) and the inner loop (Algorithm~\ref{alg:pdca-lifted}).
    \begin{itemize}
        \item[(a)] \emph{Warm start for $\bs{z}$ across penalty levels.}
        Algorithm~\ref{alg:penalty-dc-lifted} starts Algorithm~\ref{alg:pdca-lifted} at penalty level $\sigma^{t+1}$ from $\bs{z}^t$, the point returned at penalty level $\sigma^t$. The implementation sets $(\bs{x}^t,\bs{z}^t)=(\bs{x}^k,\bs{z}^{k+1})$, where $k$ is the index of the last inner iteration of Algorithm~\ref{alg:pdca-lifted} at penalty level $\sigma^t$.

        \item[(b)] \emph{Warm start for $(\bs{x},\bs{y})$ across inner iterations.}
        Within Algorithm~\ref{alg:pdca-lifted}, every $(\bs{x},\bs{y})$-subproblem~\eqref{p:inner-min-fixed-z} in Step~2 is convex and can be solved efficiently by modern solvers such as Gurobi~\citep{gurobi}. Across inner iterations, all these subproblems share exactly the same constraints, and only the linear objective coefficients of $\bs{y}$ change through $\bs{z}^k$. The solver can exploit this structure by building the model once and updating only these coefficients in each iteration, which reduces the time spent on model rebuilding and preprocessing. Moreover, the previous solution $(\bs{x}^{k-1},\bs{y}^{k-1})$ remains feasible for the current subproblem and can serve as a warm start when computing $(\bs{x}^k,\bs{y}^k)$. When $\rho>0$ is fixed, $\|\bs{z}^k-\bs{z}^{k-1}\|\to0$, so the objective coefficients in successive $(\bs{x},\bs{y})$-subproblems~\eqref{p:inner-min-fixed-z} become close. When the subproblems are linear programs, the simplex method warm-started in this way typically reaches an optimal solution within a small number of pivots.

        \item[(c)] \emph{Early termination for small penalties.}
        For early penalty levels, it is often unnecessary to solve the fixed-$\sigma$ subproblem to high accuracy, since constraint enforcement is weak when $\sigma$ is small. In practice, one may stop Algorithm~\ref{alg:pdca-lifted} early for small $\sigma$, and gradually tighten the stopping criterion as $\sigma$ increases.
    \end{itemize}
\end{itemize}

\section{Exact Penalty Frameworks}
\label{sec:exact-penalty-framework}

In this section, we establish exact penalty results for the primal problem~\eqref{p:ccp-dc} and its penalized counterpart~\eqref{p:ccp-dc-pen}, as well as the lifted problem~\eqref{p:ccp-bilinear} and its penalized counterpart~\eqref{p:ccp-bilinear-pen}. In Section~\ref{sec:exact-penalty-global-local}, we show that, under suitable regularity or polyhedral assumptions, both the primal and lifted penalty formulations preserve global optimality for sufficiently large penalty parameters, and that the primal penalty formulation preserves local optimality as well. In Section~\ref{sec:lifted-stationarity}, we further establish the connection between strongly stationary points of the original lifted problem~\eqref{p:ccp-bilinear} and stationary points of its penalized counterpart~\eqref{p:ccp-bilinear-pen} that satisfy the original complementarity constraints.

\subsection{Global and Local Exact Penalty Properties}
\label{sec:exact-penalty-global-local}

To establish the correspondence between global and local minimizers of the original problems and their penalized counterparts, we consider two alternative assumptions. The first imposes a generalized constraint qualification on the primal DC constraint, while the second exploits the polyhedral structure obtained when $\X$ is a polytope and each $h_{i,s}$ is affine.

To formulate the first assumption, recall that the primal chance constraint is represented by $G(\bs{x})\leq0$. Under Assumptions~\ref{ass:X-convex} and~\ref{ass:f-g-convex}, $G$ is locally Lipschitz continuous on an open neighborhood of $\X$. We introduce the following generalized Mangasarian--Fromovitz constraint qualification (MFCQ), expressed in terms of the Clarke directional derivative and subdifferential~\citep[Remark~(3) following Theorem~2.4 and Eq.~(2.8)]{burke1991exact}; see also \citet[Theorem~3.2]{borwein1986stability}.

\begin{assumption}[Generalized MFCQ]\label{ass:mfcq-dc}
    For every $\bar{\bs{x}}\in \X$ with $G(\bar{\bs{x}})=0$, there exists a direction $\bs{d}\in\T_{\X}(\bar{\bs{x}})$ such that $G^\circ(\bar{\bs{x}};\bs{d}) = \sup_{\bs{v}\in\partial G(\bar{\bs{x}})}\langle \bs{v},\bs{d}\rangle\;<\;0$.
\end{assumption}

We remark that Assumption~\ref{ass:mfcq-dc} is weaker than \citet[Assumption~2]{wang2025proximal}, which imposes a generalized MFCQ for a DC constraint of the form $G_1(\bs{x})-G_2(\bs{x})\le 0$. By \citet[Propositions~2.3.1 and~2.3.3]{clarke1990optimization}, we have $
\partial(G_1 - G_2)(\bs{x}) \subseteq \partial G_1(\bs{x}) + \partial (-G_2)(\bs{x}) = \partial G_1(\bs{x}) - \partial G_2(\bs{x}).
$ Therefore, for any $\bs{d}\in\T_{\X}(\bs{x})$, we have
\begin{align*}
    G^\circ(\bs{x};\bs{d})
    = \sup_{\bs{v}\in\partial G(\bs{x})}\langle \bs{v},\bs{d}\rangle
    &\leq \sup_{\substack{\bs{s}_{G_1}\in\partial G_1(\bs{x}),\\ \bs{s}_{G_2}\in\partial G_2(\bs{x})}}
    \big\langle \bs{s}_{G_1}-\bs{s}_{G_2},\bs{d}\big\rangle = \sup_{\bs{s}_{G_1}\in\partial G_1(\bs{x})}\langle \bs{s}_{G_1},\bs{d}\rangle
    -\inf_{\bs{s}_{G_2}\in\partial G_2(\bs{x})}\langle \bs{s}_{G_2},\bs{d}\rangle,
\end{align*}
 where the first equality follows from \citet[Proposition~2.1.2(b)]{clarke1990optimization}. If $G_2$ is differentiable at $\bs{x}$, the inequality in the preceding display holds with equality. The reader may refer to \citet[Remark~2]{wang2025proximal} for additional verifiable sufficient conditions under which both constraint qualifications hold.

The following theorem establishes global exact penalty properties for both the primal and lifted formulations. We also prove that the same property holds for polyhedral $\X$ and affine $h_{i,s}$, as in the first two numerical experiments, without requiring Assumption~\ref{ass:mfcq-dc}.

\begin{theorem} \label{thm:global-exact-dc}\label{thm:global-exact-bilinear}
    Suppose that Assumptions~\ref{ass:X-convex} and~\ref{ass:f-g-convex} hold, and either of the following conditions is satisfied:
    \begin{itemize}
        \item[(i)] Assumption~\ref{ass:mfcq-dc} holds.
        \item[(ii)] The set $\X$ is a polytope and each $h_{i,s}$ is affine.
    \end{itemize}
    Let $\bar{\bs{x}} \in\X$ be given. Then the following four statements are equivalent:
    \begin{itemize}
        \item[(a)] $\bar{\bs{x}}$ is a global minimizer of Problem~\eqref{p:ccp-dc}.
        \item[(b)] $G(\bar{\bs{x}})\le 0$ and there exists $\bar\sigma>0$ such that $\bar{\bs{x}}$ is a global minimizer of Problem~\eqref{p:ccp-dc-pen} for every $\sigma > \bar\sigma$.
        \item[(c)] There exist $\bar{\bs{y}}$ and $\bar{\bs{z}}$ such that $(\bar{\bs{x}},\bar{\bs{y}},\bar{\bs{z}})\in\Omega$ is a global minimizer of Problem~\eqref{p:ccp-bilinear}.
        \item[(d)] There exist $\bar{\bs{y}}$, $\bar{\bs{z}}$, and $\bar\sigma>0$ such that $(\bar{\bs{x}},\bar{\bs{y}},\bar{\bs{z}})\in\Omega$ is a global minimizer of Problem~\eqref{p:ccp-bilinear-pen} for every $\sigma>\bar\sigma$.
    \end{itemize}
\end{theorem}

\begin{proof}
We divide the proof into three parts.

\noindent\textbf{(a)$\Longleftrightarrow$(b).} By Lemma~\ref{thm:prelim-global-exact-penalty}, it suffices to establish the global error bound
\begin{equation}\label{eq:EB-global}
\dist(\bs{x},\hat\X)\le\kappa[G(\bs{x})]_+,\qquad \forall\bs{x}\in\X,
\end{equation}
for some constant $\kappa>0$.

\smallskip
\noindent\textbf{Condition (i).} For every $\bar{\bs{x}}\in\hat\X$ with $G(\bar{\bs{x}})=0$, Proposition~\ref{prop:prelim-gmfcq-mr} and Assumption~\ref{ass:mfcq-dc} give constants $\kappa_{\bar{\bs{x}}}>0$ and $r_{\bar{\bs{x}}}>0$ such that
\begin{equation}\label{eq:EB-local}
\dist(\bs{x},\hat\X)\le\kappa_{\bar{\bs{x}}}[G(\bs{x})]_+,\qquad \forall\bs{x}\in\X\cap B(\bar{\bs{x}},r_{\bar{\bs{x}}}).
\end{equation}
If $G(\bar{\bs{x}})<0$, continuity gives $r_{\bar{\bs{x}}}>0$ such that every point of $\X\cap B(\bar{\bs{x}},r_{\bar{\bs{x}}})$ is feasible, so both sides of~\eqref{eq:EB-local} are zero and~\eqref{eq:EB-local} holds there with $\kappa_{\bar{\bs{x}}}:=1$. Since $\hat\X$ is nonempty and compact, these neighborhoods admit a finite subcover $\{B(\bar{\bs{x}}^j,r_{\bar{\bs{x}}^j})\}_{j=1}^p$. Define
\[
\kappa_0:=\max_{j\in[p]}\kappa_{\bar{\bs{x}}^j},\qquad \mathcal{O}:=\X\cap\bigcup_{j=1}^p B(\bar{\bs{x}}^j,r_{\bar{\bs{x}}^j}).
\]
The error bound holds on $\mathcal{O}$ with constant $\kappa_0$. If $\mathcal{V}:=\X\setminus\mathcal{O}$ is empty, this proves~\eqref{eq:EB-global}. Otherwise, $\mathcal{V}$ is compact and contains no feasible point, so $\delta_0:=\min_{\bs{x}\in\mathcal{V}}G(\bs{x})>0$. Since $\hat\X\ne\varnothing$, we obtain
\[
\dist(\bs{x},\hat\X)\le\diam(\X)\le\frac{\diam(\X)}{\delta_0}[G(\bs{x})]_+,\qquad \forall\bs{x}\in\mathcal{V}.
\]
Taking $\kappa:=\max\{\kappa_0,\diam(\X)/\delta_0\}$ proves~\eqref{eq:EB-global}.

\smallskip
\noindent\textbf{Condition (ii).} For each $T\subseteq[S]$ with $|T|=S-m$, define $\mathcal{P}_T:=\{\bs{x}\in\X:g_s(\bs{x})\le0,\ s\in T\}$. Since each $h_{i,s}$ is affine and $\X$ is a polytope, every nonempty $\mathcal{P}_T$ is a polytope. The constraints $g_s(\bs{x})\le0$ are equivalent to $h_{i,s}(\bs{x})\le0$ for all $i\in[I]$. Hoffman's bound~\citep{hoffman1952approximate}, applied to these inequalities together with a linear description of $\X$, therefore gives a constant $\kappa_T>0$ such that
\begin{equation}\label{eq:branch-hoffman-bound}
\dist(\bs{x},\mathcal{P}_T)\le\kappa_T\max_{s\in T}[g_s(\bs{x})]_+,\qquad \forall\bs{x}\in\X.
\end{equation}
Here the residuals of the inequalities defining $\X$ vanish on $\X$, and the remaining residuals are bounded by $\max_{s\in T}[g_s(\bs{x})]_+$ after adjusting the constant. For every $T$ with $\mathcal{P}_T=\varnothing$, compactness and continuity imply
\[
\delta_T:=\min_{\bs{x}\in\X}\max_{s\in T}[g_s(\bs{x})]_+>0.
\]
Now let $T(\bs{x})$ index the $S-m$ smallest values among $\{g_s(\bs{x})\}_{s=1}^S$. By the definition of $G$, we have
\begin{equation}\label{eq:selected-scenario-residual}
\max_{s\in T(\bs{x})}[g_s(\bs{x})]_+=[g_{(S-m)}(\bs{x})]_+=[G(\bs{x})]_+.
\end{equation}
If $\mathcal{P}_{T(\bs{x})}\ne\varnothing$, then $\mathcal{P}_{T(\bs{x})}\subseteq\hat\X$ and combining~\eqref{eq:branch-hoffman-bound} and~\eqref{eq:selected-scenario-residual} yields $\dist(\bs{x},\hat\X)\le\kappa_{T(\bs{x})}[G(\bs{x})]_+$. If $\mathcal{P}_{T(\bs{x})}=\varnothing$, then \eqref{eq:selected-scenario-residual} and the definition of $\delta_{T(\bs{x})}$ give $[G(\bs{x})]_+\ge\delta_{T(\bs{x})}$, and hence
\[
\dist(\bs{x},\hat\X)\le\diam(\X)\le\frac{\diam(\X)}{\delta_{T(\bs{x})}}[G(\bs{x})]_+.
\]
There are finitely many such index sets $T$, so taking the maximum of their constants proves~\eqref{eq:EB-global}. Thus (a)$\Longleftrightarrow$(b) holds under either condition.

\smallskip
\noindent\textbf{(a)$\Longleftrightarrow$(c).}
For $\bs{x}\in\X$, define the set of satisfied scenarios by
\[
\I_{\le}(\bs{x}):=\{s\in[S]:g_s(\bs{x})\le0\}.
\]
Since $G(\bs{x})=g_{(S-m)}(\bs{x})$,
$\bs{x}\in\hat\X$ if and only if $|\I_{\le}(\bs{x})|\ge S-m$.

We first establish the correspondence between the feasible sets.
For any $\bs{x}\in\hat\X$, define
\begin{equation}\label{eq:feasible-lift-construction}
y_s(\bs{x}):=[g_s(\bs{x})]_+,
\qquad
z_s(\bs{x}):=
\begin{cases}
1,&s\in \I_{\le}(\bs{x}),\\
0,&s\notin \I_{\le}(\bs{x}).
\end{cases}
\end{equation}
Then $\bs{y}(\bs{x})\ge\bs{0}$ and
$g_s(\bs{x})\le y_s(\bs{x})$ for every $s\in[S]$.
Moreover, $\bs{z}(\bs{x})\in[0,1]^S$ and
\[
\sum_{s=1}^S z_s(\bs{x})=|\I_{\le}(\bs{x})|\ge S-m,
\]
so $\bs{z}(\bs{x})\in\C$.
If $s\in \I_{\le}(\bs{x})$, then $y_s(\bs{x})=0$;
otherwise, $z_s(\bs{x})=0$.
Hence $y_s(\bs{x})z_s(\bs{x})=0$ for every $s\in[S]$,
and therefore
$(\bs{x},\bs{y}(\bs{x}),\bs{z}(\bs{x}))\in\Omega$.

Conversely, let $(\bs{x},\bs{y},\bs{z})\in\Omega$.
For every $s\notin \I_{\le}(\bs{x})$, we have
$y_s\ge g_s(\bs{x})>0$.
The complementarity condition $y_sz_s=0$ then implies $z_s=0$.
Using $\bs{z}\in\C$, we obtain
\[
S-m
\le\sum_{s=1}^S z_s
=\sum_{s\in \I_{\le}(\bs{x})}z_s
\le |\I_{\le}(\bs{x})|.
\]
Thus $\bs{x}\in\hat\X$.

Now suppose that (a) holds. Define
$\bar{\bs{y}}:=\bs{y}(\bar{\bs{x}})$ and
$\bar{\bs{z}}:=\bs{z}(\bar{\bs{x}})$ using the construction in~\eqref{eq:feasible-lift-construction}.
Then $(\bar{\bs{x}},\bar{\bs{y}},\bar{\bs{z}})\in\Omega$.
For any $(\bs{x},\bs{y},\bs{z})\in\Omega$ with $\bs{x}\in\hat\X$,
global optimality of $\bar{\bs{x}}$ therefore implies $f(\bar{\bs{x}})\le f(\bs{x})$. Since the lifted problem also has objective $f(\bs{x})$,
$(\bar{\bs{x}},\bar{\bs{y}},\bar{\bs{z}})$ is a global minimizer
of Problem~\eqref{p:ccp-bilinear}, proving (c).

Finally, suppose that (c) holds.
The projection property gives $\bar{\bs{x}}\in\hat\X$.
For any $\bs{x}\in\hat\X$, the construction in~\eqref{eq:feasible-lift-construction} provides
a feasible lift
$(\bs{x},\bs{y}(\bs{x}),\bs{z}(\bs{x}))\in\Omega$.
Global optimality of
$(\bar{\bs{x}},\bar{\bs{y}},\bar{\bs{z}})$ consequently yields
\[
f(\bar{\bs{x}})\le f(\bs{x}),
\qquad\forall\bs{x}\in\hat\X.
\]
Hence $\bar{\bs{x}}$ is a global minimizer of
Problem~\eqref{p:ccp-dc}, proving (a).

\smallskip
\noindent\textbf{(c)$\Longleftrightarrow$(d).} For every $(\bs{x},\bs{y},\bs{z})\in\Omega_0$, we have $y_s\ge[g_s(\bs{x})]_+$ for all $s\in[S]$. Minimizing the resulting linear expression over $\bs{z}\in\C$ assigns unit weight to the $S-m$ smallest components and gives
\begin{equation}\label{eq:V-lower-phi}
V(\bs{y},\bs{z})\ge\sum_{s=1}^{S-m}[g_{(s)}(\bs{x})]_+\ge[G(\bs{x})]_+.
\end{equation}
Assume that (c) holds and choose a corresponding lifted point $(\bar{\bs{x}},\bar{\bs{y}},\bar{\bs{z}})\in\Omega$. Then $\bar{\bs{x}}$ is a global minimizer of Problem~\eqref{p:ccp-dc}. By (a)$\Longleftrightarrow$(b) and~\eqref{eq:V-lower-phi}, for every sufficiently large $\sigma$ and every $(\bs{x},\bs{y},\bs{z})\in\Omega_0$,
\[
F_\sigma(\bs{x},\bs{y},\bs{z})\ge f(\bs{x})+\sigma[G(\bs{x})]_+\ge f(\bar{\bs{x}})=F_\sigma(\bar{\bs{x}},\bar{\bs{y}},\bar{\bs{z}}).
\]
Thus the same lifted point satisfies (d). Conversely, restricting the inequality in (d) to $\Omega$, where $F_\sigma=f$, proves (c).
\end{proof}

The primal penalty formulation~\eqref{p:ccp-dc-pen} and the lifted penalty formulation~\eqref{p:ccp-bilinear-pen} penalize constraint violations differently. For the lifted formulation, projecting out the auxiliary variables $(\bs{y}, \bs{z})$ gives
\[
\min_{(\bs{x},\bs{y},\bs{z})\in\Omega_0}f(\bs{x})+\sigma V(\bs{y},\bs{z})
=\min_{\bs{x} \in \X} f(\bs{x}) + \sigma \sum_{s=1}^{S-m}[g_{(s)}(\bs{x})]_+.
\]
Thus, the lifted penalty aggregates violations among the $S-m$ scenarios with the smallest constraint violations, whereas the penalty for the primal formulation is $[G(\bs{x})]_+=[g_{(S-m)}(\bs{x})]_+$, which can be interpreted as the maximum violation among these same scenarios truncated below at zero. Theorem~\ref{thm:global-exact-dc} shows that both penalties preserve global optimality for sufficiently large penalty parameters.

The proof of Theorem~\ref{thm:global-exact-dc} relies on a global error bound for the primal constraint $G(\bs{x})\le0$. Since a global error bound implies a local one, we obtain local exact penalty properties for the primal formulation as a direct corollary. We next show that under suitable assumptions, local minimizers of the primal problem~\eqref{p:ccp-dc} remain locally optimal for its penalized counterpart~\eqref{p:ccp-dc-pen} when the penalty parameter is sufficiently large. Conversely, any local minimizer of the penalized problem that is feasible for the original problem is also locally optimal for the latter.

\begin{corollary}
    \label{thm:local-exact-dc}
    Suppose that Assumptions~\ref{ass:X-convex} and~\ref{ass:f-g-convex} hold, and either of the following conditions is satisfied:
    \begin{itemize}
        \item[(i)] Assumption~\ref{ass:mfcq-dc} holds.
        \item[(ii)] The set $\X$ is a polytope and each $h_{i,s}$ is affine.
    \end{itemize}
    Let $\bar{\bs{x}}\in\X$ be given. Then the following two statements are equivalent:
    \begin{itemize}
        \item[(a)] $\bar{\bs{x}} \in \hat \X$ is a local minimizer of Problem~\eqref{p:ccp-dc}, i.e., there exists $r_0>0$ such that $f(\bar{\bs{x}})\le f(\bs{x})$ for all $\bs{x}\in \hat\X\cap B(\bar{\bs{x}},r_0)$.
        \item[(b)] $G(\bar{\bs{x}})\le 0$ and there exists $\bar\sigma>0$ such that, for every $\sigma>\bar\sigma$, $\bar{\bs{x}}$ is a local minimizer of the penalized problem~\eqref{p:ccp-dc-pen}, i.e., there exists $r_\sigma>0$ such that $f(\bar{\bs{x}})+\sigma[G(\bar{\bs{x}})]_+\le f(\bs{x})+\sigma[G(\bs{x})]_+$ for all $\bs{x}\in \X\cap B(\bar{\bs{x}},r_\sigma)$.
    \end{itemize}
\end{corollary}

\begin{proof}
Under either condition (i) or (ii), the global error
bound~\eqref{eq:EB-global} established in the proof of
Theorem~\ref{thm:global-exact-dc} implies a local error
bound at $\bar{\bs{x}}$. Applying Lemma~\ref{thm:prelim-local-exact-penalty} immediately yields the desired result.
\end{proof}

We next identify sufficient conditions under which a critical point of the penalized problem is a local minimizer. Under Assumptions~\ref{ass:X-convex} and~\ref{ass:f-g-convex}, if each $h_{i,s}(\bs{x})$ is affine in $\bs{x}$, then $\bar{\bs{x}}\in\X$ is a local minimizer of Problem~\eqref{p:ccp-dc-pen} if and only if it is a $\mathrm{d}$-stationary point, i.e., $P_\sigma'(\bar{\bs{x}};\bs{x}-\bar{\bs{x}})\ge0$ for all $\bs{x}\in\X$~\citep[Proposition~6.1.1]{cui2021modern}. A sufficient condition for a critical point of a DC program to be $\mathrm{d}$-stationary is that the concave part is differentiable at that point~\citep[Proposition~6.1.10]{cui2021modern}. Thus, in the affine setting above, a critical point $\bar{\bs{x}}$ of Problem~\eqref{p:ccp-dc-pen} is a local minimizer whenever $G_2$ is differentiable at $\bar{\bs{x}}$. Conditions~\eqref{eq:unique-top-m} and~\eqref{eq:unique-active-piece} ensure the required differentiability of $G_2$.

For the lifted formulation, however, global exactness does not automatically imply analogous relationships between local minimizers or stationary points of Problems~\eqref{p:ccp-bilinear} and~\eqref{p:ccp-bilinear-pen}. We examine these relationships in the next subsection.

\subsection{Stationarity and Local Optimality in the Lifted Formulation}
\label{sec:lifted-stationarity}
In this subsection, we investigate the local connection between the MPCC~\eqref{p:ccp-bilinear} in the lifted space and its smooth penalized counterpart~\eqref{p:ccp-bilinear-pen}. A well-known subtlety in MPCC is that even when the penalty parameter is sufficiently large, the set of local minimizers of an MPCC generally does not admit a one-to-one correspondence with that of its penalized counterpart~\citep{scheel2000mathematical, jara2018study}. By Lemma~\ref{lem:lifted-fixed-penalty}, for any $\sigma>0$, Algorithm~\ref{alg:pdca-lifted} yields stationary points of Problem~\eqref{p:ccp-bilinear-pen}  as $\varepsilon\to0$. We first show that, whenever a stationary point is feasible for Problem~\eqref{p:ccp-bilinear}, it is a strongly stationary point for that problem. We then provide sufficient conditions under which, for sufficiently large penalty parameters, a feasible point is a strongly stationary point of Problem~\eqref{p:ccp-bilinear} if and only if it is a stationary point of Problem~\eqref{p:ccp-bilinear-pen}.

We first introduce the definition of strongly stationary points for Problem~\eqref{p:ccp-bilinear}. Let $\bar{\bs{w}} =(\bar{\bs{x}},\bar{\bs{y}},\bar{\bs{z}})\in\Omega$ be given. Define the index sets
    \begin{equation}\label{eq:mpcc-index-sets}
    \begin{aligned}
        \I_y(\bar{\bs{w}}):=\{s\in[S]: \bar y_s=0, \bar z_s>0\},\,
        \I_z(\bar{\bs{w}}):=\{s\in[S]: \bar z_s=0, \bar y_s>0\}.
    \end{aligned}
    \end{equation}

Using these index sets, we define the \emph{relaxed program} at $\bar{\bs{w}}$ by fixing the zero component of each strictly complementary pair and retaining both nonnegativity inequalities at biactive indices (those with $\bar y_s=\bar z_s=0$):
    \begin{equation}
        \label{p:bilinear-relaxed}
        \min_{(\bs{x},\bs{y},\bs{z}) \in \Omega_0}\left\{ f(\bs{x}) : y_s=0,\forall s\in \I_y(\bar{\bs{w}}), \; z_s=0,\forall s\in \I_z(\bar{\bs{w}})\right\}.
    \end{equation}
Let us denote the feasible region of Problem~\eqref{p:bilinear-relaxed} by
\begin{equation*}
    \Omega_{\mathrm{re}}(\bar{\bs{w}})\ :=\ \Bigl\{(\bs{x},\bs{y},\bs{z})\in\Omega_0:\
    y_s=0\ \forall s\in \I_y(\bar{\bs{w}}),\ z_s=0\ \forall s\in \I_z(\bar{\bs{w}})\Bigr\}.
\end{equation*}

Following the MPCC literature \citep{scheel2000mathematical,jara2018study}, we define strong stationarity by the first-order optimality condition for the relaxed problem~\eqref{p:bilinear-relaxed}. Since the relaxed set contains the tightened set of Definition~\ref{def:weak-stationarity}, strong stationarity implies weak stationarity.

\begin{definition}[Strong stationarity]\label{def:strong-stationarity}
    Under Assumptions~\ref{ass:X-convex} and~\ref{ass:f-g-convex}, we call a point $\bar{\bs{w}}\in\Omega$ a \emph{strongly stationary point} of Problem~\eqref{p:ccp-bilinear} if
    \begin{equation}\label{eq:relaxed-stationarity-nc}
        \bs{0}\in\bigl(\nabla f(\bar{\bs{x}}),\bs{0},\bs{0}\bigr)+\N_{\Omega_{\mathrm{re}}(\bar{\bs{w}})}(\bar{\bs{w}}).
    \end{equation}
\end{definition}

Since~\eqref{p:bilinear-relaxed} is a convex program, condition~\eqref{eq:relaxed-stationarity-nc} is equivalent to $\bar{\bs{w}}$ being a globally optimal solution of the relaxed program~\eqref{p:bilinear-relaxed}. The next proposition shows that this global optimality over $\Omega_{\mathrm{re}}(\bar{\bs{w}})$ implies local optimality for Problem~\eqref{p:ccp-bilinear}.

\begin{proposition}\label{prop:strong-stat-localmin}
    Suppose that Assumptions~\ref{ass:X-convex} and~\ref{ass:f-g-convex} hold. If $\bar{\bs{w}}\in\Omega$ is a strongly stationary point of Problem~\eqref{p:ccp-bilinear}, then $\bar{\bs{w}}$ is a local minimizer of Problem~\eqref{p:ccp-bilinear}.
\end{proposition}

\begin{proof}
Let $\bar{\bs{w}}=(\bar{\bs{x}},\bar{\bs{y}},\bar{\bs{z}})\in\Omega$ be a strongly stationary point of Problem~\eqref{p:ccp-bilinear}. Under Assumptions~\ref{ass:X-convex} and~\ref{ass:f-g-convex}, the set $\Omega_{\mathrm{re}}(\bar{\bs{w}})$ is convex, and $\bar{\bs{w}}\in\Omega_{\mathrm{re}}(\bar{\bs{w}})$ because $\bar{\bs{w}}\in\Omega$. By~\eqref{eq:relaxed-stationarity-nc} and the definition of the normal cone, $\langle\nabla f(\bar{\bs{x}}),\bs{x}-\bar{\bs{x}}\rangle\ge0$ for all $(\bs{x},\bs{y},\bs{z})\in\Omega_{\mathrm{re}}(\bar{\bs{w}})$. Since $f$ is convex, this yields
\begin{equation}\label{eq:relaxed-global-opt}
    f(\bs{x})\ \ge\ f(\bar{\bs{x}})+\langle\nabla f(\bar{\bs{x}}),\bs{x}-\bar{\bs{x}}\rangle\ \ge\ f(\bar{\bs{x}}),\qquad \forall\,(\bs{x},\bs{y},\bs{z})\in\Omega_{\mathrm{re}}(\bar{\bs{w}}).
\end{equation}

Define
\[
    \delta:=\min\Bigl\{\min_{s\in\I_y(\bar{\bs{w}})}\bar z_s,\ \min_{s\in\I_z(\bar{\bs{w}})}\bar y_s\Bigr\},
\]
with the convention that the minimum over an empty index set is $+\infty$. By the definition of the index sets in~\eqref{eq:mpcc-index-sets}, we have $\delta>0$. Let $\bs{w}=(\bs{x},\bs{y},\bs{z})\in\Omega$ satisfy $\|\bs{w}-\bar{\bs{w}}\|<\delta$. For every $s\in\I_y(\bar{\bs{w}})$, we have $z_s\ge\bar z_s-\|\bs{w}-\bar{\bs{w}}\|>\bar z_s-\delta\ge0$, so the complementarity condition $y_sz_s=0$ in the definition of $\Omega$ gives $y_s=0$. Similarly, for every $s\in\I_z(\bar{\bs{w}})$, we have $y_s>\bar y_s-\delta\ge0$, and hence $z_s=0$. Since $\bs{w}\in\Omega\subseteq\Omega_0$, it follows that $\bs{w}\in\Omega_{\mathrm{re}}(\bar{\bs{w}})$, and~\eqref{eq:relaxed-global-opt} yields $f(\bs{x})\ge f(\bar{\bs{x}})$. Therefore, $\bar{\bs{w}}$ is a local minimizer of Problem~\eqref{p:ccp-bilinear}.
\end{proof}

We next establish the connection between stationary points of the penalized problem~\eqref{p:ccp-bilinear-pen} and strongly stationary points of Problem~\eqref{p:ccp-bilinear}. For this purpose, we impose a relative Slater condition on the scenarios whose slack variables $y_s$ are fixed to zero in the relaxed program~\eqref{p:bilinear-relaxed}.

\begin{assumption}[Relative Slater condition]
\label{ass:relative-slater}
Let $\bar{\bs{w}}=(\bar{\bs{x}},\bar{\bs{y}},\bar{\bs{z}})\in\Omega$ be given. There exists $\bs{x}^{\mathrm{0}}\in\operatorname{ri}(\X)$ such that
\begin{equation}
    \label{eq:relative-slater}
    h_{i,s}(\bs{x}^{\mathrm{0}})<0,
    \qquad
    \forall i\in[I],\quad
    \forall s\in\I_y(\bar{\bs{w}}).
\end{equation}
\end{assumption}

The relative Slater condition above ensures the existence of multipliers that allow us to connect strongly stationary points of the original problem~\eqref{p:ccp-bilinear} with stationary points of its penalized counterpart~\eqref{p:ccp-bilinear-pen}. Classical MPCC constraint qualifications provide alternative sufficient conditions for the existence of multipliers; see \citet{scheel2000mathematical}, \citet{hoheisel2013theoretical}, and \citet{steffensen2010new}. Alternatively, when $\X$ is a polytope and each $h_{i,s}$ is affine, multiplier existence follows from the polyhedral structure. The following theorem establishes the equivalence between strong stationarity and stationarity for Problems~\eqref{p:ccp-bilinear} and~\eqref{p:ccp-bilinear-pen}, respectively, at feasible points of the former and for sufficiently large penalty parameters.

\begin{theorem}
    \label{thm:stationarity-equivalence-lifted}
    Let $\bar{\bs{w}}=(\bar{\bs{x}},\bar{\bs{y}},\bar{\bs{z}})\in\Omega$. Suppose that Assumptions~\ref{ass:X-convex} and~\ref{ass:f-g-convex} hold and that one of the following conditions is satisfied:
    \begin{itemize}
        \item[(i)] Assumption~\ref{ass:relative-slater} holds at $\bar{\bs{w}}$.
        \item[(ii)] The set $\X$ is a polytope and each $h_{i,s}$ is affine.
    \end{itemize}
    Then the following statements are equivalent:
    \begin{itemize}
        \item[(a)] $\bar{\bs{w}}$ is a strongly stationary point of Problem~\eqref{p:ccp-bilinear}.
        \item[(b)] There exists $\bar\sigma>0$ such that, for every $\sigma > \bar\sigma$, $\bar{\bs{w}}$ is a stationary point of Problem~\eqref{p:ccp-bilinear-pen}.
    \end{itemize}
\end{theorem}

\begin{proof}
Let $\bar{\bs{w}}=(\bar{\bs{x}},\bar{\bs{y}},\bar{\bs{z}})\in\Omega$.

\smallskip
\noindent\textbf{(a)$\Longrightarrow$(b).} We first show that
\[
\operatorname{proj}_{\bs{x}}\Omega_{\mathrm{re}}(\bar{\bs{w}})
=\{\bs{x}\in\X:h_{i,s}(\bs{x})\le0,\ i\in[I],\ s\in \I_y(\bar{\bs{w}})\}.
\]
Indeed, every feasible point $(\bs{x},\bs{y},\bs{z})$ of the relaxed program~\eqref{p:bilinear-relaxed} satisfies $y_s=0$ for $s\in \I_y(\bar{\bs{w}})$, so its $\bs{x}$-component belongs to the set on the right-hand side. Conversely, any $\bs{x}$ in this set can be lifted to the relaxed feasible region by setting $\bs{z}=\bar{\bs{z}}$, $y_s=0$ for $s\in \I_y(\bar{\bs{w}})$, and $y_s=[g_s(\bs{x})]_+$ otherwise.

Strong stationarity~\eqref{eq:relaxed-stationarity-nc} implies that $\bar{\bs{w}}$ is a global minimizer of the convex relaxed program~\eqref{p:bilinear-relaxed}. Since its objective depends only on $\bs{x}$, the projection identity shows that $\bar{\bs{x}}$ solves
\[
\min_{\bs{x}\in\X}
\left\{f(\bs{x}):h_{i,s}(\bs{x})\le0,\ i\in[I],\ s\in \I_y(\bar{\bs{w}})\right\}.
\]
The existence of KKT multipliers follows from either the relative Slater condition~\eqref{eq:relative-slater} at $\bar{\bs{w}}$ under condition (i) or the polyhedral setting under condition (ii). Thus, there exist $\bs{v}\in\N_{\X}(\bar{\bs{x}})$ and $\lambda_{i,s}\ge0$ such that
\begin{equation}\label{eq:strong-stat-kkt}
\bs{0}
=\nabla f(\bar{\bs{x}})
+\sum_{s\in \I_y(\bar{\bs{w}})}\sum_{i=1}^I
\lambda_{i,s}\nabla h_{i,s}(\bar{\bs{x}})
+\bs{v},
\qquad
\lambda_{i,s}h_{i,s}(\bar{\bs{x}})=0.
\end{equation}
By convexity of $f$ and $h_{i,s}$ for all $i,s$ under Assumption~\ref{ass:f-g-convex} and~\eqref{eq:strong-stat-kkt},
\begin{equation}\label{eq:strong-stat-lagrangian-bound}
f(\bs{x})
+\sum_{s\in \I_y(\bar{\bs{w}})}\sum_{i=1}^I\lambda_{i,s}h_{i,s}(\bs{x})
\ge f(\bar{\bs{x}}),
\qquad \forall\bs{x}\in\X.
\end{equation}

Choose $\bar\sigma>0$ such that $\bar\sigma>\max_{s\in \I_y(\bar{\bs{w}})}\bigl(\sum_{i=1}^I\lambda_{i,s}\bigr)/\bar z_s$. Fix any $\sigma>\bar\sigma$ and define
\[
D:=\{(\bs{x},\bs{y})\in\X\times\Re_+^S:
h_{i,s}(\bs{x})\le y_s,\ i\in[I],\ s\in[S]\}.
\]
For every $(\bs{x},\bs{y})\in D$, we obtain
\[
\begin{aligned}
f(\bs{x})+\sigma\langle\bar{\bs{z}},\bs{y}\rangle
\ge
f(\bs{x})
+\sum_{s\in \I_y(\bar{\bs{w}})}\sum_{i=1}^I\lambda_{i,s}h_{i,s}(\bs{x})
+\sum_{s\in \I_y(\bar{\bs{w}})}\left(\sigma\bar z_s-\sum_{i=1}^I\lambda_{i,s}\right)y_s\ge f(\bar{\bs{x}}),
\end{aligned}
\]
where the last inequality follows from~\eqref{eq:strong-stat-lagrangian-bound}, $\sigma\bar z_s-\sum_{i=1}^I\lambda_{i,s}\ge0$ for all $s\in\I_y(\bar{\bs{w}})$, and $\bs{y}\ge\bs{0}$.
Since $\langle\bar{\bs{z}},\bar{\bs{y}}\rangle=0$, the pair $(\bar{\bs{x}},\bar{\bs{y}})$ minimizes $f(\bs{x})+\sigma\langle\bar{\bs{z}},\bs{y}\rangle$ over $D$. The first-order optimality condition is
\begin{equation}\label{eq:strong-stat-xy-optimality}
\bs{0}\in
\begin{pmatrix}
\nabla f(\bar{\bs{x}})\\
\sigma\bar{\bs{z}}
\end{pmatrix}
+\N_D(\bar{\bs{x}},\bar{\bs{y}}).
\end{equation}
Moreover,
\[
\langle\sigma\bar{\bs{y}},\bs{z}-\bar{\bs{z}}\rangle
=\sigma\langle\bar{\bs{y}},\bs{z}\rangle\ge0,
\qquad \forall\bs{z}\in\C,
\]
so $\bar{\bs{z}}$ minimizes $\sigma\langle\bar{\bs{y}},\bs{z}\rangle$ over $\C$. The first-order optimality condition is
\begin{equation}\label{eq:strong-stat-z-optimality}
\bs{0}\in\sigma\bar{\bs{y}}+\N_{\C}(\bar{\bs{z}}).
\end{equation}
Combining~\eqref{eq:strong-stat-xy-optimality} and~\eqref{eq:strong-stat-z-optimality} and using $\Omega_0=D\times\C$ yields $\bs{0}\in\nabla F_\sigma(\bar{\bs{w}})+\N_{\Omega_0}(\bar{\bs{w}})$, i.e.,~\eqref{eq:stationarity-pen} with $\varepsilon=0$. Since $\sigma>\bar\sigma$ was arbitrary, $\bar{\bs{w}}$ is a stationary point of Problem~\eqref{p:ccp-bilinear-pen} for every $\sigma>\bar\sigma$, which proves (b).

\smallskip
\noindent\textbf{(b)$\Longrightarrow$(a).} Suppose that $\bar{\bs{w}}\in\Omega$ is a stationary point of Problem~\eqref{p:ccp-bilinear-pen} for some $\sigma>0$. By~\eqref{eq:stationarity-pen} with $\varepsilon=0$,
\[
\langle\nabla F_\sigma(\bar{\bs{w}}),\bs{w}-\bar{\bs{w}}\rangle\ge0,\qquad \forall\bs{w}\in\Omega_0.
\]
For every $\bs{w}=(\bs{x},\bs{y},\bs{z})\in\Omega_{\mathrm{re}}(\bar{\bs{w}})$, we have $y_s=\bar y_s=0$ whenever $\bar z_s>0$ and $z_s=\bar z_s=0$ whenever $\bar y_s>0$. At biactive indices, both $\bar y_s$ and $\bar z_s$ vanish. Therefore,
\[
\langle\nabla F_\sigma(\bar{\bs{w}}),\bs{w}-\bar{\bs{w}}\rangle=\langle\nabla f(\bar{\bs{x}}),\bs{x}-\bar{\bs{x}}\rangle\ge0,\qquad \forall\bs{w}\in\Omega_{\mathrm{re}}(\bar{\bs{w}}).
\]
This is exactly~\eqref{eq:relaxed-stationarity-nc}, proving strong stationarity.
\end{proof}

\section{From Lifted to Primal Local Optimality}
\label{sec:iden-localmin}

In the previous section, we established local exactness of the primal penalty formulation and showed in Proposition~\ref{prop:strong-stat-localmin} that every strongly stationary point of the lifted problem~\eqref{p:ccp-bilinear} is a local minimizer. Given a local minimizer $(\bar{\bs{x}},\bar{\bs{y}},\bar{\bs{z}})$ of Problem~\eqref{p:ccp-bilinear}, we ask whether its $\bs{x}$-component $\bar{\bs{x}}$ is a local minimizer of Problem~\eqref{p:ccp-dc}.

We first show that every local minimizer of the primal formulation~\eqref{p:ccp-dc} can be lifted to a local minimizer of Problem~\eqref{p:ccp-bilinear}. We then demonstrate that the converse fails: the lifted formulation may admit spurious local minimizers whose projections are not locally optimal for the primal problem. To explain this discrepancy, we characterize local optimality of both formulations and derive a sufficient condition under which lifted local optimality is preserved for the primal.

\subsection{Failure of Primal Local Optimality}

The following proposition gives the counterpart of \citet[Theorem~3.4]{burdakov2016mathematical} for the present formulation. It shows that every local minimizer of the primal problem~\eqref{p:ccp-dc} admits a lifted feasible point that is a local minimizer of the lifted problem~\eqref{p:ccp-bilinear}.

\begin{proposition}\label{prop:equiv-local-P2Q}
Let $\bar{\bs{x}}\in\hat{\X}$ be a local minimizer of Problem~\eqref{p:ccp-dc}. Then there exist $\bar{\bs{y}}\in\Re^S_+$ and $\bar{\bs{z}}\in\C$ such that $(\bar{\bs{x}},\bar{\bs{y}},\bar{\bs{z}})\in\Omega$ is a local minimizer of Problem~\eqref{p:ccp-bilinear}.
\end{proposition}

\begin{proof}
Let $r_0>0$ be such that $f(\bar{\bs{x}})\le f(\bs{x})$ for all $\bs{x}\in\hat\X\cap B(\bar{\bs{x}},r_0)$, and let $\bar J:=\{s\in[S]:g_s(\bar{\bs{x}})\le0\}$. Define
\[
\bar y_s:=[g_s(\bar{\bs{x}})]_+,
\qquad
\bar z_s:=
\begin{cases}
1, & s\in\bar J,\\
0, & s\notin\bar J,
\end{cases}
\qquad \forall s\in[S].
\]
Since $\bar{\bs{x}}\in\hat\X$, Proposition~\ref{prop:equiv-cc} gives $|\bar J|\ge S-m$, and hence $\bar{\bs{z}}\in\C$. Moreover, $\bar{\bs{y}}\in\Re^S_+$ and $h_{i,s}(\bar{\bs{x}})\le g_s(\bar{\bs{x}})\le\bar y_s$ for all $i\in[I]$ and $s\in[S]$. Complementarity also holds: $\bar y_s=0$ for $s\in\bar J$ and $\bar z_s=0$ for $s\notin\bar J$, so $\bar y_s\bar z_s=0$ for all $s\in[S]$. Thus $\bar{\bs{w}}:=(\bar{\bs{x}},\bar{\bs{y}},\bar{\bs{z}})\in\Omega$.

We next show that the $\bs{x}$-component of every point in $\Omega$ is primal feasible. Let $(\bs{x},\bs{y},\bs{z})\in\Omega$ and $K:=\{s\in[S]:z_s>0\}$. Since $0\le z_s\le1$ for all $s\in[S]$, we have $|K|\ge\sum_{s\in K}z_s=\bs{1}^\top\bs{z}\ge S-m$. For each $s\in K$, complementarity gives $y_s=0$, and hence $g_s(\bs{x})\le y_s=0$. Therefore at least $S-m$ scenarios satisfy $g_s(\bs{x})\le0$, and Proposition~\ref{prop:equiv-cc} gives $\bs{x}\in\hat\X$.

Set $\varepsilon:=r_0$. For every $\bs{w}=(\bs{x},\bs{y},\bs{z})\in\Omega\cap B(\bar{\bs{w}},\varepsilon)$, we have $\bs{x}\in\hat\X$ and $\|\bs{x}-\bar{\bs{x}}\|\le\|\bs{w}-\bar{\bs{w}}\|<r_0$, so $f(\bs{x})\ge f(\bar{\bs{x}})$. Since the objective of Problem~\eqref{p:ccp-bilinear} is $f(\bs{x})$, the point $\bar{\bs{w}}$ is a local minimizer of Problem~\eqref{p:ccp-bilinear}.
\end{proof}

However, the converse of Proposition~\ref{prop:equiv-local-P2Q} does not hold without additional assumptions. This phenomenon also occurs in the cardinality-constrained reformulation studied by \citet{burdakov2016mathematical}. We use the following simple one-dimensional example to illustrate this failure.

\begin{example}\label{exa:failure-Q2P}
Consider the one-dimensional setting with $S=2$, $m=1$, and $\X=[-1,1]$. Let $f(x)=x$, $g_1(x)=[x]_+$, and $g_2(x)=[-x]_+$. Then
\[
    G_1(x)-G_2(x)
    =\min\{[x]_+,[-x]_+\}
    =0,
    \qquad \forall x\in\X.
\]
Therefore, the DC formulation~\eqref{p:ccp-dc} reduces to $\min\{x:-1\leq x\leq1\}$, whose unique global minimizer is $x=-1$.

The point $x=0$ can nevertheless appear as the $\bs{x}$-component of a local minimizer of the lifted formulation
\[ \min_{x\in\X,\;\bs{y}\geq\bs{0},\;\bs{z}\in\C}
    \left\{
        x:
        g_1(x)\leq y_1,\;
        g_2(x)\leq y_2,\;
        y_1z_1+y_2z_2=0
    \right\}.
\]
Indeed, the point $(\bar{x},\bar{\bs{y}},\bar{\bs{z}})=(0,(0,0),(1,1))$ is a local minimizer of this problem. Any nearby feasible point has $z_1,z_2>0$, so complementarity forces $y_1=y_2=0$ and hence $x=0$. However, $\bar{x}=0$ is not a local minimizer of Problem~\eqref{p:ccp-dc}, since every $x\in[-1,0)$ is feasible and has a smaller objective value. A detailed verification is given after the example.
\end{example}

\begin{proof}[Verification of the example]
The point $(\bar{x},\bar{\bs{y}},\bar{\bs{z}})=(0,(0,0),(1,1))$ is feasible because $g_1(0)=g_2(0)=0$, $\bar z_1+\bar z_2=2\ge S-m=1$, and $\bar y_1\bar z_1+\bar y_2\bar z_2=0$. Choose $\varepsilon\in(0,1/2)$. Every $(x,\bs{y},\bs{z})\in\Omega$ within distance $\varepsilon$ of this point satisfies $z_1>1/2$ and $z_2>1/2$. Nonnegativity and complementarity imply $y_1=y_2=0$. Feasibility then gives $[x]_+=[-x]_+=0$, so $x=0$. Thus every feasible point in this neighborhood has objective value $f(x)=0$, proving lifted local optimality. On the other hand, every $x\in[-1,1]$ is primal feasible and $f(x)<f(0)$ whenever $x<0$, so $\bar x=0$ is not a primal local minimizer.
\end{proof}

\subsection{Characterization of Local Optimality}

We now derive necessary and sufficient conditions for local optimality in the primal and lifted formulations~\eqref{p:ccp-dc} and~\eqref{p:ccp-bilinear}, which can help clarify the distinction between the two formulations. To state the results, we first introduce some notation.
For each $J\subseteq[S]$, define the set $\X_J:=\{\bs{x}\in\X:g_s(\bs{x})\leq0,\;s\in J\}$.
For $\bs{x}\in\hat\X$, define the strictly satisfied and active scenario sets by
\[
    \I_-(\bs{x})
    :=\{s\in[S]:g_s(\bs{x})<0\},
    \qquad
    \I_0(\bs{x})
    :=\{s\in[S]:g_s(\bs{x})=0\},
\]
and let $\I_{\le}(\bs{x}):=\I_-(\bs{x})\cup \I_0(\bs{x})$ be the set of all satisfied scenarios. Since each $g_s$ is continuous under Assumption~\ref{ass:f-g-convex}, scenarios in $\I_-({\bs{x}})$ remain strictly satisfied under sufficiently small perturbations of $\bs{x}$ within $\X$. For $\bar{\bs{x}}\in\hat\X$, define $r:=\max\{0,S-m-|\I_-(\bar{\bs{x}})|\}$, which is the minimum number of active scenarios that must remain satisfied to maintain primal feasibility locally at $\bar{\bs{x}}$. Then $0\leq r\leq|\I_0(\bar{\bs{x}})|$. The following theorem characterizes lifted and primal local optimality through convex subproblems involving active scenarios.

\begin{theorem}
\label{thm:local-optimality-scenario-branches}
Suppose that Assumptions~\ref{ass:X-convex} and~\ref{ass:f-g-convex} hold. Let $\bar{\bs{w}}=(\bar{\bs{x}},\bar{\bs{y}},\bar{\bs{z}})\in\Omega$ and define $\I_{0+}(\bar{\bs{w}}):=\{s\in \I_0(\bar{\bs{x}}):\bar z_s>0\}$. Then the following assertions hold.
\begin{itemize}
    \item[(a)] The point $\bar{\bs{w}}$ is a local minimizer of Problem~\eqref{p:ccp-bilinear} if and only if
    \[
        \bar{\bs{x}}
        \in\arg\min_{\bs{x}\in\X_{\I_{0+}(\bar{\bs{w}})}}f(\bs{x}).
    \]

    \item[(b)] The point $\bar{\bs{x}}$ is a local minimizer of Problem~\eqref{p:ccp-dc} if and only if
    \[
        \bar{\bs{x}}
        \in\arg\min_{\bs{x}\in\X_{J}}f(\bs{x})
        \qquad
        \text{for every }J\subseteq \I_0(\bar{\bs{x}})
        \text{ with }|J|=r.
    \]
\end{itemize}
\end{theorem}

\begin{proof}
\noindent\textbf{(a).} Recall that $\bar{\bs{w}}\in\Omega$ implies $\bar{\bs{x}}\in\X$ and
\[
\bar y_s\ge0,\qquad \bar z_s\ge0,\qquad
\bar y_s\bar z_s=0,\qquad g_s(\bar{\bs{x}})\le\bar y_s,
\qquad \forall s\in[S].
\]
Thus, $\bar z_s>0$ implies $\bar y_s=0$, so the definition of $\I_y(\bar{\bs{w}})$ in~\eqref{eq:mpcc-index-sets} gives $\I_y(\bar{\bs{w}})=\{s\in[S]:\bar z_s>0\}$ and hence $\I_{0+}(\bar{\bs{w}})=\I_y(\bar{\bs{w}})\cap\I_0(\bar{\bs{x}})$.
For every $s\in\I_y(\bar{\bs{w}})$, we have $g_s(\bar{\bs{x}})\le\bar y_s=0$, and hence $\bar{\bs{x}}\in\X_{\I_y(\bar{\bs{w}})}$.
For each $s\in\I_y(\bar{\bs{w}})\setminus\I_{0+}(\bar{\bs{w}})$, we have $g_s(\bar{\bs{x}})\le0$, and equality is excluded by the definition of $\I_{0+}(\bar{\bs{w}})$. Hence, $g_s(\bar{\bs{x}})<0$. By continuity, these inequalities remain strict in a neighborhood of $\bar{\bs{x}}$, so there exists $\varepsilon>0$ such that
\begin{equation}\label{eq:local-selected-active-sets}
\X_{\I_y(\bar{\bs{w}})}\cap B(\bar{\bs{x}},\varepsilon)=\X_{\I_{0+}(\bar{\bs{w}})}\cap B(\bar{\bs{x}},\varepsilon).
\end{equation}
By~\eqref{eq:local-selected-active-sets}, local optimality at $\bar{\bs{x}}$ over $\X_{\I_y(\bar{\bs{w}})}$ is equivalent to local optimality over $\X_{\I_{0+}(\bar{\bs{w}})}$.

Suppose first that $\bar{\bs{w}}$ is a local minimizer of Problem~\eqref{p:ccp-bilinear}. For any $\bs{x}\in\X_{\I_y(\bar{\bs{w}})}$, define $y_s(\bs{x})=0$ for $s\in\I_y(\bar{\bs{w}})$ and $y_s(\bs{x})=\max\{\bar y_s,g_s(\bs{x})\}$ otherwise. Then $(\bs{x},\bs{y}(\bs{x}),\bar{\bs{z}})\in\Omega$. This is because for $s\in\I_y(\bar{\bs{w}})$, we have $g_s(\bs{x})\le0=y_s(\bs{x})$, whereas for $s\notin\I_y(\bar{\bs{w}})$, we have $\bar z_s=0$ and $y_s(\bs{x})\ge\bar y_s\ge0$. Moreover, continuity and feasibility of $\bar{\bs{w}}$ give $\bs{y}(\bs{x})\to\bar{\bs{y}}$ as $\bs{x}\to\bar{\bs{x}}$ within $\X_{\I_y(\bar{\bs{w}})}$. Thus, local optimality of $\bar{\bs{w}}$ implies that $\bar{\bs{x}}$ is a local minimizer of $f$ on $\X_{\I_y(\bar{\bs{w}})}$, and hence on $\X_{\I_{0+}(\bar{\bs{w}})}$ by~\eqref{eq:local-selected-active-sets}. Since $f$ and $\X_{\I_{0+}(\bar{\bs{w}})}$ are convex, $\bar{\bs{x}}$ is a global minimizer on this set.

Conversely, suppose that $\bar{\bs{x}}$ minimizes $f$ over $\X_{\I_{0+}(\bar{\bs{w}})}$. Since $\I_{0+}(\bar{\bs{w}})\subseteq\I_y(\bar{\bs{w}})$, the definition of $\X_J$ gives $\X_{\I_y(\bar{\bs{w}})}\subseteq\X_{\I_{0+}(\bar{\bs{w}})}$. As $\bar{\bs{x}}\in\X_{\I_y(\bar{\bs{w}})}$, it also minimizes $f$ over $\X_{\I_y(\bar{\bs{w}})}$. Every $(\bs{x},\bs{y},\bs{z})\in\Omega$ sufficiently close to $\bar{\bs{w}}$ satisfies $z_s>0$ for every $s\in\I_y(\bar{\bs{w}})$. Complementarity then gives $y_s=0$ and $g_s(\bs{x})\le0$ for these indices. Hence $\bs{x}\in\X_{\I_y(\bar{\bs{w}})}$ and $f(\bs{x})\ge f(\bar{\bs{x}})$, proving lifted local optimality.

\smallskip
\noindent\textbf{(b).} By continuity of the finitely many functions $g_s$, there exists $\varepsilon>0$ such that every $\bs{x}\in\X\cap B(\bar{\bs{x}},\varepsilon)$ satisfies $g_s(\bs{x})<0$ for $s\in \I_-(\bar{\bs{x}})$ and $g_s(\bs{x})>0$ for $s\notin \I_{\le}(\bar{\bs{x}})$. Hence such a point is primal feasible if and only if at least $r$ scenarios in $\I_0(\bar{\bs{x}})$ are satisfied. Consequently,
\begin{equation}\label{eq:boundary-local-union}
\hat\X\cap B(\bar{\bs{x}},\varepsilon)=\bigcup_{J\subseteq \I_0(\bar{\bs{x}}),\ |J|=r}\bigl(\X_{J}\cap B(\bar{\bs{x}},\varepsilon)\bigr).
\end{equation}
Equation~\eqref{eq:boundary-local-union} also holds when $r=0$, since the union then consists of the single set $\X_{\varnothing}=\X$ intersected with the ball.

If $\bar{\bs{x}}$ is a primal local minimizer, then~\eqref{eq:boundary-local-union} implies that it is a local minimizer on every $\X_{J}$ with $J\subseteq \I_0(\bar{\bs{x}})$ and $|J|=r$. Each such set is convex and contains $\bar{\bs{x}}$, so convexity of $f$ gives global optimality on $\X_{J}$. Conversely, suppose that $\bar{\bs{x}}$ minimizes $f$ over $\X_J$ for every $J\subseteq\I_0(\bar{\bs{x}})$ with $|J|=r$. Then~\eqref{eq:boundary-local-union} gives $f(\bs{x})\ge f(\bar{\bs{x}})$ for every $\bs{x}\in\hat\X\cap B(\bar{\bs{x}},\varepsilon)$. This proves part (b).
\end{proof}

Theorem~\ref{thm:local-optimality-scenario-branches} clarifies the relationship illustrated by Proposition~\ref{prop:equiv-local-P2Q} and Example~\ref{exa:failure-Q2P}. Feasibility of $\bar{\bs{w}}$ ensures that at least $S-m$ scenarios have $\bar z_s>0$ and satisfy $g_s(\bar{\bs{x}})\le0$. Since at most $|\I_-(\bar{\bs{x}})|$ of them are strictly satisfied, $|\I_{0+}(\bar{\bs{w}})|\ge r$. We can therefore choose $J\subseteq\I_{0+}(\bar{\bs{w}})$ with $|J|=r$. Part~(b) requires $\bar{\bs{x}}$ to minimize $f$ over $\X_J$, so it also minimizes $f$ over the subset $\X_{\I_{0+}(\bar{\bs{w}})}$ containing $\bar{\bs{x}}$. This gives the condition in part~(a).

By Theorem~\ref{thm:stationarity-equivalence-lifted}, every stationary point $\bar{\bs{w}}=(\bar{\bs{x}},\bar{\bs{y}},\bar{\bs{z}})$ of Problem~\eqref{p:ccp-bilinear-pen} that is feasible for Problem~\eqref{p:ccp-bilinear} is a strongly stationary point and hence, by Proposition~\ref{prop:strong-stat-localmin}, a local minimizer of Problem~\eqref{p:ccp-bilinear}. However, its projection $\bar{\bs{x}}$ need not be a local minimizer of Problem~\eqref{p:ccp-dc}. Theorem~\ref{thm:local-optimality-scenario-branches}(b) gives a necessary and sufficient condition for this projected point to be a primal local minimizer. The following corollary gives a sufficient condition under which lifted local optimality implies primal local optimality.

\begin{corollary}
\label{thm:equiv-local-Q2P}
Let $\bar{\bs{w}}=(\bar{\bs{x}},\bar{\bs{y}},\bar{\bs{z}})\in\Omega$ be a local minimizer of Problem~\eqref{p:ccp-bilinear}. Suppose that there exists $\delta>0$ such that
\begin{equation}\label{eq:selected-scenario-persistence}
    \hat\X\cap B(\bar{\bs{x}},\delta)
    \subseteq\X_{\I_y(\bar{\bs{w}})}.
\end{equation}
Then $\bar{\bs{x}}$ is a local minimizer of Problem~\eqref{p:ccp-dc}.
\end{corollary}

\begin{proof}
For any $\bs{x}\in\hat\X\cap B(\bar{\bs{x}},\delta)$, keep $\bs{z}=\bar{\bs{z}}$ and define
\[
y_s(\bs{x}):=\max\{\bar y_s,[g_s(\bs{x})]_+\},
\qquad s\in[S].
\]
By construction, $y_s(\bs{x})\ge0$ and $g_s(\bs{x})\le y_s(\bs{x})$. For $s\in \I_y(\bar{\bs{w}})$, the definition in~\eqref{eq:mpcc-index-sets} gives $\bar y_s=0$, while~\eqref{eq:selected-scenario-persistence} gives $g_s(\bs{x})\le0$. Hence $y_s(\bs{x})=0$. For $s\notin \I_y(\bar{\bs{w}})$, we have $\bar z_s=0$. Thus $y_s(\bs{x})\bar z_s=0$ for every $s$, and consequently $(\bs{x},\bs{y}(\bs{x}),\bar{\bs{z}})\in\Omega$.

Moreover, feasibility of $\bar{\bs{w}}$ for Problem~\eqref{p:ccp-bilinear} gives $\bar y_s\ge[g_s(\bar{\bs{x}})]_+$, so $\bs{y}(\bar{\bs{x}})=\bar{\bs{y}}$. Since each $g_s$ is continuous, the constructed lifted point satisfies
\[
(\bs{x},\bs{y}(\bs{x}),\bar{\bs{z}})
\longrightarrow
(\bar{\bs{x}},\bar{\bs{y}},\bar{\bs{z}})
\qquad\text{as }\bs{x}\to\bar{\bs{x}}.
\]
Therefore, for every sufficiently close primal feasible point $\bs{x}$, this feasible lifted point $(\bs{x},\bs{y}(\bs{x}),\bar{\bs{z}})$ lies in the local optimality neighborhood of the given lifted point $(\bar{\bs{x}},\bar{\bs{y}},\bar{\bs{z}})$. Since the objective depends only on $\bs{x}$, lifted local optimality yields $f(\bs{x})\ge f(\bar{\bs{x}})$. Hence $\bar{\bs{x}}$ is a local minimizer of Problem~\eqref{p:ccp-dc}.
\end{proof}

Since $\X_{\I_y(\bar{\bs{w}})}\subseteq\hat\X$, condition~\eqref{eq:selected-scenario-persistence} means that $\X_{\I_y(\bar{\bs{w}})}$ and $\hat\X$ coincide locally, which holds whenever $|\I_{\le}(\bar{\bs{x}})|=S-m$, that is, exactly $S-m$ scenarios satisfy $g_s(\bar{\bs{x}})\le0$. Condition~\eqref{eq:selected-scenario-persistence} is sufficient but not necessary for primal local optimality. For instance, replacing $f(x)=x$ by $f(x)=x^2$ in Example~\ref{exa:failure-Q2P} makes $x=0$ a primal global minimizer, although~\eqref{eq:selected-scenario-persistence} still fails for $\bar{\bs{z}}=(1,1)$.

\section{Numerical Experiments}
\label{sec:numerical}
In this section, we conduct experiments to test the proposed algorithms on real and synthetic datasets. We denote Algorithm~\ref{alg:penalty-dc-primal} by PenDC-P(rimal) and Algorithm~\ref{alg:penalty-dc-lifted} by PenDC-L(ifted), and compare them with the CVaR approximation~\citep{rockafellar2002conditional,nemirovski2007convex}, the mixed-integer formulation (MIP) of \citet{ahmed2008solving}, the DCA of \citet{wang2025proximal}, and ALSO-X+~\citep{jiang2022also}.

\phantomsection\label{app:numerical-details}
We use Gurobi (v13.0.3) to solve all linear, quadratic, and mixed-integer (sub)problems. The solver uses its default primal feasibility and integrality tolerances, $10^{-6}$ and $10^{-5}$, respectively. We set a time limit of 1800 seconds and a default 0.01\% optimality gap tolerance for the MIP model. For instances not solved to optimality within the time limit, we report the average of the relative gaps returned by Gurobi, computed for each instance as $\textrm{gap}(\%)=|\textrm{UB}-\textrm{LB}|/|\textrm{UB}|\times100$, where $\textrm{UB}$ is the incumbent objective value and $\textrm{LB}$ is the best lower bound reported by the solver. For PenDC-P and PenDC-L, we stop the outer loop when the current iterate $\bs{x}^t$ satisfies the sample-based chance constraint or when the time limit is reached, whichever occurs first. For DCA and the inner loops of PenDC-P and PenDC-L, we terminate when $|v_k-v_{k+1}|/\max\{1,|v_{k+1}|\}\le 10^{-6}$, where $v_k$ is the objective value at the $k$th iteration. For PenDC-P and PenDC-L, it is the penalized objective in~\eqref{p:ccp-dc-pen} and~\eqref{p:ccp-bilinear-pen}, respectively. Initial points for PenDC-P and PenDC-L are chosen randomly. To obtain a cheap warm start, we stop the inner loop after one iteration in the first outer iteration and after two iterations in the second. For DCA, we use the CVaR solution as the initial point whenever the CVaR approximation is feasible. If the CVaR approximation is infeasible, we initialize DCA at a random point. ALSO-X+ terminates when $|t^U-t^L|\le 10^{-6}$. All non-MIP methods have an 1800-second time limit; because time is checked only after each iteration, actual runtime may slightly exceed this limit. All experiments are implemented in Python 3.12.14 and run on a Mac mini with a 12-core Apple M6 processor and 24 GB of RAM, running macOS 27.0.

All reported values are averages over five instances. We use ``fval'' for objective value, ``time'' for wall-clock time in seconds, ``prob'' for empirical satisfaction probability, ``gap'' for the MIP optimality gap, ``solved'' for the number of feasible outputs, and ``/'' when at least one instance fails to return a valid feasible solution. For a returned point $\bs{x}$ and a scenario tolerance $\epsilon\ge0$, we compute the empirical probability $\hat p_{\epsilon}(\bs{x}) := \frac{1}{S}\sum_{s=1}^S \ind\{g_s(\bs{x})\le \epsilon\}$
and count the output as feasible if $\hat p_{\epsilon}(\bs{x})\ge 1-\alpha$. We use $\epsilon=10^{-4}$ for all six methods in all three experiments.

\subsection{A VaR-Constrained Portfolio Optimization Problem}

In this subsection, we consider a VaR-constrained mean--variance portfolio selection model studied by \citet{bai2021augmented} and \citet{wang2025proximal}. Let $\bs{\mu}\in\Re^n$ and $\boldsymbol{\Sigma}\in\Re^{n\times n}$ denote the estimated mean vector and covariance matrix of returns for $n$ risky assets, and let $\gamma>0$ be a risk-aversion parameter. Let $\bs{x}\in\Re^n_{+}$ denote the portfolio weight vector. The problem can be formulated as follows:
\[
\min_{\bs{x}\in\Re^n}\left\{\gamma \bs{x}^\top \boldsymbol{\Sigma}\bs{x} - \bs{\mu}^\top \bs{x}: \Pr\{\txi^\top \bs{x} \geq R\} \geq 1-\alpha, \sum^n_{i=1}x_i = 1, 0 \leq x_i \leq u_i, \; i\in[n]\right\}.
\]
The problem can be interpreted as follows: the objective minimizes a quadratic variance penalty minus expected return, subject to a chance constraint requiring that the portfolio return $\txi^\top \bs{x}$ is at least a prespecified target $R$ with probability at least $1-\alpha$. In addition, we impose the budget constraint $\sum_{i=1}^n x_i=1$ and box constraints $0\le x_i\le u_i$ for each $i\in[n]$ to avoid overconcentration.

To construct test instances from real data, we use a dataset of 2523 daily returns of 435 stocks in the S\&P 500 index from March 2006 to March 2016, which can be downloaded from \url{https://github.com/INFORMSJoC/2024.0648}. Specifically, we consider four problem sizes $n\in\{100,200,300,400\}$ and set the sample size to $S=3n$. For each pair $(n,S)$, we run the algorithms on five independent instances and report the average performance metrics over these five instances. For the PenDC-L method, we set $\sigma^1 = 5\cdot 10^{-3}$, $\beta  = 4.0$, and $\rho = 10^{-4}$. For the PenDC-P method, we set $\sigma^1 = 3\cdot 10^{-3}$, $\beta  = 1.5$, and $\rho = 0$.

\begin{table}[!htbp]
    \centering
    \caption{Comparison of methods on the portfolio optimization problem.}
    \label{tab:portfolio-avg}
    \scriptsize
    \setlength{\tabcolsep}{2pt}
    \renewcommand{\arraystretch}{0.92}
    \resizebox{\textwidth}{!}{%
        \begin{tabular}{cc@{\hskip 5pt}cccccc@{\hskip 10pt}cccccc}
            \toprule
            \multicolumn{2}{c}{} &
            \multicolumn{6}{c}{$\alpha=0.05$} &
            \multicolumn{6}{c}{$\alpha=0.10$}\\
            \cmidrule(lr){3-8}\cmidrule(l){9-14}
            $\left(\begin{array}{c}n\\S\end{array}\right)$ & &
            MIP & CVaR & PenDC-L & PenDC-P & DCA & ALSO-X+ &
            MIP & CVaR & PenDC-L & PenDC-P & DCA & ALSO-X+\\
            \midrule
            \multirow{4}{*}{$\left(\begin{array}{c}100\\300\end{array}\right)$}
            & fval & -0.013550 & -0.011861 & \textbf{-0.013387} & -0.013123 & -0.012863 & -0.013268 &
            -0.014429 & -0.012284 & \textbf{-0.014289} & -0.014204 & -0.013771 & -0.014179\\
            & time (gap) & 55.1 (0.0\%) & 0.0432 & \textbf{0.0366} & 0.0600 & 0.2021 & 1.1703 &
            11.5 (0.0\%) & \textbf{0.0407} & 0.0488 & 0.0615 & 0.2499 & 1.0899\\
            & prob & 0.9500 & 0.9887 & 0.9500 & 0.9507 & 0.9620 & 0.9500 &
            0.9000 & 0.9687 & 0.9033 & 0.9007 & 0.9120 & 0.9007\\
            & solved & 5/5 & 5/5 & 5/5 & 5/5 & 5/5 & 5/5 &
            5/5 & 5/5 & 5/5 & 5/5 & 5/5 & 5/5\\
            \midrule
            \multirow{4}{*}{$\left(\begin{array}{c}200\\600\end{array}\right)$}
            & fval & -0.013522 & -0.011914 & \textbf{-0.013355} & -0.013081 & -0.012862 & -0.013208 &
            -0.014231 & -0.012371 & \textbf{-0.014169} & -0.013941 & -0.013850 & -0.014080\\
            & time (gap) & 1800.0 (3.5\%) & 0.2570 & \textbf{0.1658} & 0.2596 & 1.0604 & 5.5068 &
            1800.0 (0.8\%) & 0.2322 & \textbf{0.1202} & 0.2250 & 2.4652 & 4.7441\\
            & prob & 0.9500 & 0.9873 & 0.9503 & 0.9507 & 0.9597 & 0.9527 &
            0.9000 & 0.9620 & 0.9027 & 0.9007 & 0.9077 & 0.9003\\
            & solved & 5/5 & 5/5 & 5/5 & 5/5 & 5/5 & 5/5 &
            5/5 & 5/5 & 5/5 & 5/5 & 5/5 & 5/5\\
            \midrule
            \multirow{4}{*}{$\left(\begin{array}{c}300\\900\end{array}\right)$}
            & fval & -0.013470 & -0.011830 & \textbf{-0.013375} & -0.012870 & -0.012887 & -0.013165 &
            -0.014376 & -0.012284 & \textbf{-0.014319} & -0.013964 & -0.014105 & -0.014235\\
            & time (gap) & 1800.0 (7.1\%) & 0.7240 & \textbf{0.3107} & 0.6730 & 6.0277 & 15.0359 &
            1800.0 (2.4\%) & 0.7248 & \textbf{0.2820} & 0.5750 & 15.5848 & 15.0896\\
            & prob & 0.9500 & 0.9853 & 0.9500 & 0.9500 & 0.9529 & 0.9513 &
            0.9000 & 0.9633 & 0.9029 & 0.9011 & 0.9033 & 0.9009\\
            & solved & 5/5 & 5/5 & 5/5 & 5/5 & 5/5 & 5/5 &
            5/5 & 5/5 & 5/5 & 5/5 & 5/5 & 5/5\\
            \midrule
            \multirow{4}{*}{$\left(\begin{array}{c}400\\1200\end{array}\right)$}
            & fval & -0.013673 & -0.011939 & \textbf{-0.013626} & -0.013255 & -0.013172 & -0.013285 &
            -0.014667 & -0.012467 & \textbf{-0.014503} & -0.014383 & -0.014361 & -0.014456\\
            & time (gap) & 1800.0 (8.8\%) & 1.9028 & \textbf{0.5639} & 1.5429 & 12.3198 & 43.2008 &
            1800.0 (2.7\%) & 1.6066 & \textbf{0.5281} & 1.1882 & 35.4503 & 42.2871\\
            & prob & 0.9500 & 0.9860 & 0.9503 & 0.9505 & 0.9547 & 0.9507 &
            0.9000 & 0.9653 & 0.9073 & 0.9008 & 0.9025 & 0.9018\\
            & solved & 5/5 & 5/5 & 5/5 & 5/5 & 5/5 & 5/5 &
            5/5 & 5/5 & 5/5 & 5/5 & 5/5 & 5/5\\
            \bottomrule
        \end{tabular}
    }
    \vspace{2pt}
    \begin{flushleft}\footnotesize
        Each average uses all five feasible outputs; otherwise / is reported.
    \end{flushleft}
\end{table}

The results in Table~\ref{tab:portfolio-avg} show that PenDC-L performs well in both solution quality and computational efficiency. It produces objective values close to those of the MIP benchmark in all eight settings while requiring substantially less runtime. Among the non-MIP methods, PenDC-L achieves the lowest average objective value in all eight settings and the shortest average runtime in seven of them. The CVaR approximation is fast but more conservative, yielding higher empirical satisfaction probabilities and worse objective values. Compared with DCA, PenDC-P achieves lower average objective values in six of the eight settings and shorter runtimes in all eight settings.

\subsection{A Probabilistic Resource Planning Problem}

We next consider a chance constrained linear resource planning problem, which has been studied by \citet{luedtke2010integer} and \citet{bai2021augmented}. The model concerns shipping goods from $N$ suppliers to $M$ customers at minimum total transportation cost, where customer demands are uncertain. Specifically, the demand of customer $j\in[M]$ is modeled by a random variable $\txi_j$, while each supplier $i\in[N]$ is subject to a deterministic capacity limit $\theta_i>0$. Let $c_{ij}\ge 0$ denote the unit shipping cost from supplier $i$ to customer $j$, and let $x_{ij}\ge 0$ be the amount shipped along arc $(i,j)$. The shipment plan is decided in advance of demand realization, and feasibility is enforced in a probabilistic sense by requiring that all customer demands are met simultaneously with probability at least $1-\alpha$. The resulting formulation is
\begin{align*}
    \min_{\bs{x}\in\Re^{N\times M}} \left\{ \sum_{i=1}^N \sum_{j=1}^M c_{ij}x_{ij}:
    \begin{array}{l}
        \displaystyle \Pr\left\{\sum_{i=1}^N x_{ij}\ge \txi_j,\ \forall j\in[M]\right\}\ge 1-\alpha, \\
        \displaystyle \sum_{j=1}^M x_{ij}\le \theta_i,\ \forall i\in[N],
        \quad x_{ij}\ge 0, \quad \forall i\in[N], \;j\in[M]
    \end{array}\right\}.
\end{align*}

We use the public dataset provided at \url{http://homepages.cae.wisc.edu/~luedtkej/}. Throughout, we let $N = 40$ and consider four settings $(M, S)\in\{(100,1000),(100, 2000), (200, 2000), (200, 3000)\}$. For both the PenDC-L and PenDC-P methods, we set $\sigma^1 = 5$, $\beta  = 4.5$, and $\rho = 10^{-3}$.

\begin{table}[!htbp]
    \centering
    \caption{Comparison of methods on the probabilistic resource planning problem.}
    \label{tab:resource-avg}
    \scriptsize
    \setlength{\tabcolsep}{2pt}
    \renewcommand{\arraystretch}{0.92}
    \resizebox{\textwidth}{!}{%
        \begin{tabular}{cc@{\hskip 5pt}cccccc@{\hskip 10pt}cccccc}
            \toprule
            \multicolumn{2}{c}{} &
            \multicolumn{6}{c}{$\alpha=0.05$} &
            \multicolumn{6}{c}{$\alpha=0.10$}\\
            \cmidrule(lr){3-8}\cmidrule(l){9-14}
            $\left(\begin{array}{c}M\\S\end{array}\right)$ & &
            MIP & CVaR & PenDC-L & PenDC-P & DCA & ALSO-X+ &
            MIP & CVaR & PenDC-L & PenDC-P & DCA & ALSO-X+\\
            \midrule
            \multirow{4}{*}{$\left(\begin{array}{c}100\\1000\end{array}\right)$}
            & fval & 4.1309 & / & \textbf{4.1422} & 4.2552 & 4.2081 & 4.1740 &
            4.0540 & / & \textbf{4.0667} & 4.2283 & 4.1440 & 4.0840\\
            & time (gap) & 60.2 (0.0\%) & / & \textbf{7.0915} & 10.1877 & 7.7487 & 22.0296 &
            157.9 (0.0\%) & / & \textbf{6.3999} & 9.4313 & 8.2337 & 28.7908\\
            & prob & 0.9500 & / & 0.9502 & 0.9500 & 0.9500 & 0.9502 &
            0.9000 & / & 0.9000 & 0.9000 & 0.9000 & 0.9004\\
            & solved & 5/5 & 4/5 & 5/5 & 5/5 & 5/5 & 5/5 &
            5/5 & 4/5 & 5/5 & 5/5 & 5/5 & 5/5\\
            \midrule
            \multirow{4}{*}{$\left(\begin{array}{c}100\\2000\end{array}\right)$}
            & fval & 4.4168 & 4.6538 & \textbf{4.4264} & 4.6087 & 4.4946 & 4.4582 &
            4.3212 & 4.5638 & \textbf{4.3325} & 4.5741 & 4.3843 & 4.3585\\
            & time (gap) & 188.5 (0.0\%) & \textbf{5.7784} & 17.1390 & 19.0424 & 22.3438 & 59.1460 &
            465.4 (0.0\%) & \textbf{7.3069} & 19.1842 & 23.5982 & 22.9861 & 84.1294\\
            & prob & 0.9500 & 0.9848 & 0.9500 & 0.9500 & 0.9500 & 0.9500 &
            0.9000 & 0.9636 & 0.9002 & 0.9000 & 0.9000 & 0.9000\\
            & solved & 5/5 & 5/5 & 5/5 & 5/5 & 5/5 & 5/5 &
            5/5 & 5/5 & 5/5 & 5/5 & 5/5 & 5/5\\
            \midrule
            \multirow{4}{*}{$\left(\begin{array}{c}200\\2000\end{array}\right)$}
            & fval & 8.3550 & / & \textbf{8.3741} & / & / & 8.4300 &
            8.2001 & / & \textbf{8.2174} & 8.5983 & / & 8.2655\\
            & time (gap) & 809.2 (0.0\%) & / & \textbf{64.3117} & / & / & 206.1994 &
            1800.3 (0.1\%) & / & 72.6735 & \textbf{40.9038} & / & 263.9521\\
            & prob & 0.9500 & / & 0.9500 & / & / & 0.9501 &
            0.9000 & / & 0.9000 & 0.9000 & / & 0.9001\\
            & solved & 5/5 & 4/5 & 5/5 & 4/5 & 4/5 & 5/5 &
            5/5 & 4/5 & 5/5 & 5/5 & 4/5 & 5/5\\
            \midrule
            \multirow{4}{*}{$\left(\begin{array}{c}200\\3000\end{array}\right)$}
            & fval & 8.5813 & / & \textbf{8.6001} & 8.9437 & 8.7778 & 8.6595 &
            8.4040 & 8.8778 & \textbf{8.4276} & 8.8959 & 8.5255 & 8.4731\\
            & time (gap) & 1676.0 (0.5\%) & / & 108.8975 & \textbf{57.1522} & 79.5629 & 390.1066 &
            1800.2 (6.7\%) & \textbf{37.1995} & 106.2359 & 62.9275 & 103.9728 & 432.0110\\
            & prob & 0.9500 & / & 0.9500 & 0.9500 & 0.9500 & 0.9501 &
            0.9000 & 0.9648 & 0.9001 & 0.9000 & 0.9000 & 0.9001\\
            & solved & 5/5 & 4/5 & 5/5 & 5/5 & 5/5 & 5/5 &
            5/5 & 5/5 & 5/5 & 5/5 & 5/5 & 5/5\\
            \bottomrule
        \end{tabular}
    }
    \vspace{2pt}
    \begin{flushleft}\footnotesize
        Objective values are divided by $10^{7}$.
        Each average uses all five feasible outputs; otherwise / is reported.
    \end{flushleft}
\end{table}

Table~\ref{tab:resource-avg} shows that DCA achieves lower average objective values than PenDC-P in all six settings for which both methods return feasible solutions for all five instances. Their runtime ordering depends on the setting, with DCA faster in three of these six settings. Due to the conservatism of the CVaR approximation, CVaR fails on some instances. DCA requires a feasible initial point. Of the five instances initialized randomly after CVaR is found infeasible, DCA returns feasible outputs on three of them. PenDC-P does not require a feasible initial point, but on one instance with $(M,S)=(200,2000)$ and $\alpha=0.05$ Gurobi reports numerical difficulties on one of its subproblems, so no solution is returned. In contrast, PenDC-L and ALSO-X+ return feasible solutions on all instances.

Among the iterative methods, PenDC-L achieves the lowest average objective value in every setting. In the settings where PenDC-L, PenDC-P, and DCA all return feasible solutions for all five instances, PenDC-L is faster than both other methods when $M=100$, but requires more time when $(M,S)=(200,3000)$. Even in these larger settings, however, PenDC-L remains substantially more efficient than ALSO-X+.

\subsection{Linear Objective with Joint Quadratic Chance Constraint}
\label{sec:numerical-nonlinear}

We further test our algorithms on a problem with a linear objective and a joint nonlinear (convex) chance constraint, which has been used as a benchmark by \citet{hong2011sequential} and \citet{wang2025proximal}. The decision variable is $\bs{x}\in\Re^d_+$ and the objective is to minimize the negative sum of components. The feasibility requirement is imposed through a joint quadratic chance constraint, and the problem reads
\[
\min_{\bs{x}\in[0,U]^d} \left\{ -\sum_{i=1}^d x_i
:
\Pr\left\{\sum_{i=1}^d \txi_{ij}^2 x_i^2 \le \theta,\ \forall j\in[M]\right\}\ \ge\ 1-\alpha\right\},
\]
where the random coefficients $\{\txi_{ij}\}_{i\in[d],\,j\in[M]}$ follow a dependent Gaussian structure. To ensure that the deterministic feasible set is compact, we impose the componentwise upper bounds $x_i \le U$ for all $i \in [d]$, where $U>0$ is a sufficiently large constant, e.g., $U=10^4$. The reader may refer to~\citet{wang2025proximal} for the detailed data generation process. In the numerical study, we fix $d=20$, $M=20$, and $\theta=100$, and we vary the sample size $S$. Specifically, we consider $S\in\{500,1000,2000\}$. For the PenDC-P method, we set $\sigma^1 = 4\cdot10^{-3}$, $\beta = 15$, and $\rho = 0$. For the PenDC-L method, we set $\sigma^1 = 8\cdot10^{-5}$, $\beta = 10$, and $\rho = 10^{-3}$.

\begin{table}[!htbp]
    \centering
    \caption{Comparison of methods on the problem with a joint quadratic chance constraint.}
    \label{tab:nonlinear-avg}
    \scriptsize
    \setlength{\tabcolsep}{2pt}
    \renewcommand{\arraystretch}{0.92}
    \resizebox{\textwidth}{!}{%
        \begin{tabular}{cc@{\hskip 5pt}cccccc@{\hskip 10pt}cccccc}
            \toprule
            \multicolumn{2}{c}{}
            & \multicolumn{6}{c}{$\alpha=0.05$}
            & \multicolumn{6}{c}{$\alpha=0.10$}\\
            \cmidrule(lr){3-8}\cmidrule(l){9-14}
            $S$ & & MIP & CVaR & PenDC-L & PenDC-P & DCA & ALSO-X+ &
            MIP & CVaR & PenDC-L & PenDC-P & DCA & ALSO-X+\\
            \midrule
            \multirow{4}{*}{500}
            & fval & -16.5535 & -14.9194 & \textbf{-16.5204} & -15.5383 & -16.4723 & -16.5027 &
            -17.5002 & -15.7334 & \textbf{-17.4413} & -15.8349 & -17.4187 & -17.4250\\
            & time (gap) & 165.7 (0.0\%) & \textbf{2.1293} & 10.9637 & 470.5908 & 178.8133 & 94.0002 &
            1800.0 (2.7\%) & \textbf{1.4300} & 10.9380 & 621.4288 & 231.2255 & 103.1960\\
            & prob & 0.9500 & 0.9796 & 0.9500 & 0.9500 & 0.9500 & 0.9500 &
            0.9000 & 0.9624 & 0.9000 & 0.9000 & 0.9000 & 0.9004\\
            & solved & 5/5 & 5/5 & 5/5 & 5/5 & 5/5 & 5/5 &
            5/5 & 5/5 & 5/5 & 5/5 & 5/5 & 5/5\\
            \midrule
            \multirow{4}{*}{1000}
            & fval & / & -15.0201 & \textbf{-16.3818} & -15.0013 & -16.3427 & -16.3758 &
            / & -15.7932 & \textbf{-17.3788} & -14.7983 & -17.3461 & -17.3490\\
            & time (gap) & / & \textbf{5.6491} & 26.2365 & 1307.6159 & 448.5037 & 228.1529 &
            / & \textbf{4.4321} & 25.8602 & 1593.2785 & 597.2679 & 274.9111\\
            & prob & / & 0.9820 & 0.9500 & 0.9502 & 0.9500 & 0.9500 &
            / & 0.9636 & 0.9002 & 0.9004 & 0.9000 & 0.9002\\
            & solved & 4/5 & 5/5 & 5/5 & 5/5 & 5/5 & 5/5 &
            2/5 & 5/5 & 5/5 & 5/5 & 5/5 & 5/5\\
            \midrule
            \multirow{4}{*}{2000}
            & fval & -16.1677 & -14.8413 & \textbf{-16.1401} & / & -16.1067 & -16.1226 &
            -17.2106 & -15.6311 & \textbf{-17.1955} & / & -17.1484 & -17.1606\\
            & time (gap) & 1800.0 (16.5\%) & \textbf{12.8155} & 69.9200 & / & 1381.0310 & 605.3849 &
            1800.0 (38.4\%) & \textbf{12.8300} & 75.4159 & / & 1472.1644 & 705.7701\\
            & prob & 0.9500 & 0.9808 & 0.9500 & / & 0.9500 & 0.9500 &
            0.9000 & 0.9635 & 0.9003 & / & 0.9000 & 0.9004\\
            & solved & 5/5 & 5/5 & 5/5 & 4/5 & 5/5 & 5/5 &
            5/5 & 5/5 & 5/5 & 0/5 & 5/5 & 5/5\\
            \bottomrule
        \end{tabular}
    }
    \vspace{2pt}
    \begin{flushleft}\footnotesize
        Each average uses all five feasible outputs; otherwise / is reported.
    \end{flushleft}
\end{table}

Table~\ref{tab:nonlinear-avg} shows that DCA and PenDC-P require substantially longer runtimes than PenDC-L. Due to the big-$M$ formulation and the solver's integrality tolerance, MIP incumbents reported as feasible by the solver may slightly violate the chance constraint, so MIP returns feasible outputs on only four and two of the five instances at $S=1000$ for $\alpha=0.05$ and $\alpha=0.10$, respectively. DCA returns feasible solutions for all five instances in every setting, but its average runtime is 16 to 23 times that of PenDC-L. PenDC-P also returns feasible solutions for all five instances when $S\in\{500,1000\}$. At $S=2000$, however, it returns feasible solutions on only four instances for $\alpha=0.05$ and on none for $\alpha=0.10$. PenDC-L returns feasible solutions for all instances and has the shortest runtime among the iterative methods in every setting. It achieves the lowest average objective value among these methods in all six settings and remains considerably faster than ALSO-X+ throughout the table.

\section{Conclusion}
\label{sec:conclusion}
In this paper, we developed penalty based DC algorithms for the SAA formulation of chance constrained programs in a convex setting. Using a rank-based DC representation of the empirical chance constraint, we proposed a primal-space method that avoids feasible initialization and a method based on an equivalent lifted formulation. We established exact penalty and stationarity guarantees for both formulations under suitable regularity or polyhedral assumptions. We further characterized when lifted local optimality implies primal local optimality. Numerical studies demonstrated the efficiency of the proposed methods. Future work includes extending the method to distributionally robust chance constrained programs under Wasserstein ambiguity sets.

\phantomsection
\label{opre:references-start}


\bibliography{project}

\clearpage
\appendix
\numberwithin{equation}{section}

\section{Preliminaries on Exact Penalization of Constrained Optimization}
\label{sec:preliminaries}

In this section, we list some classical results on exact penalization of constrained optimization problems and provide their proofs directly after the corresponding statements. The reader may also refer to \citet{burke1991exact} for a detailed review. Consider the following constrained problem
\begin{equation} \label{p:prelim-constrained}
    \min_{\bs{x} \in \Re^d} \left\{ f(\bs{x}) : g(\bs{x}) \leq 0, \; \bs{x} \in \X\right\},
\end{equation}
where $f,g:\Re^d\to\Re$ are locally Lipschitz and $\X\subseteq\Re^d$ is convex and compact. Denote the feasible set by
\[
\mathcal{F}:=\{\bs{x}\in\X: g(\bs{x})\le 0\}.
\]

Let $g^\circ(\bar{\bs{x}};\bs{d})$ be the Clarke directional derivative of $g$ at $\bar{\bs{x}}$ along $\bs{d}$. We say that the generalized Mangasarian--Fromovitz constraint qualification (GMFCQ) holds at a boundary feasible point $\bar{\bs{x}}\in\mathcal{F}$ with $g(\bar{\bs{x}})=0$ if there exists a direction $\bs{d}\in\T_{\X}(\bar{\bs{x}})$ such that
\begin{equation} \label{eq:prelim-gmfcq}
    g^\circ(\bar{\bs{x}};\bs{d})\;<\;0.
\end{equation}

Let $\mathcal{G}(\bs{x}):=g(\bs{x})+\Re_+ = \{t + g(\bs{x}):t \geq 0\}$ for $\bs{x}\in\X$ and $\mathcal{G}^{-1}(y) = \{\bs{x} \in \X:y \in \mathcal{G}(\bs{x})\}$ for $y \in \Re$. Under the GMFCQ, the mapping $\mathcal{G}$ can be shown to be metrically regular at $(\bar{\bs{x}},0)$ relative to $\X$~\citep[Definition~2.2]{burke1991exact}. That is, there exist $\kappa, \varepsilon > 0$ such that $\dist(\bs{x}, \mathcal{G}^{-1}(y)) \leq \kappa\,\dist(y, \mathcal{G}(\bs{x}))$ for all $\bs{x} \in \X\cap B(\bar{\bs{x}}, \varepsilon)$ and $y \in B(0, \varepsilon)$. We state this in the following proposition.

\begin{proposition} \label{prop:prelim-gmfcq-mr}
    Suppose GMFCQ~\eqref{eq:prelim-gmfcq} holds at $\bar{\bs{x}}$. Define $\mathcal{G}(\bs{x}):=g(\bs{x})+\Re_+$ for $\bs{x}\in\X$. Then $\mathcal{G}$ is metrically regular at $(\bar{\bs{x}},0)$ relative to $\X$. Consequently, there exist $\kappa>0$ and $\varepsilon>0$ such that the local linear error bound holds:
    \begin{equation}\label{eq:prelim-error-bound}
        \dist(\bs{x},\mathcal{F})
        \ \le\ \kappa\,[g(\bs{x})]_+,
        \qquad \forall\,\bs{x}\in \X\cap B(\bar{\bs{x}},\varepsilon).
    \end{equation}
\end{proposition}

\begin{proof}
Applying~\eqref{eq:notation-clarke-support} to $g$, we obtain
    \begin{equation}\label{eq:clarke-dd-support}
        g^\circ(\bar{\bs{x}};\bs{d})\;=\;\max_{\bs{v}\in\partial g(\bar{\bs{x}})}\ \langle \bs{v},\bs{d}\rangle.
    \end{equation}

    \smallskip
    \noindent\textbf{Step 1: GMFCQ $\Longrightarrow$ $0\notin \partial g(\bar{\bs{x}})+\N_{\X}(\bar{\bs{x}})$.}
    Assume GMFCQ holds at $\bar{\bs{x}}$, i.e., there exists $\bs{d}\in\T_{\X}(\bar{\bs{x}})$ such that $g^\circ(\bar{\bs{x}};\bs{d})<0$. By~\eqref{eq:clarke-dd-support}, this implies $\langle \bs{v},\bs{d}\rangle<0$ for all $\bs{v}\in\partial g(\bar{\bs{x}})$. If $0\in \partial g(\bar{\bs{x}})+\N_{\X}(\bar{\bs{x}})$, then there exist $\bs{v}\in\partial g(\bar{\bs{x}})$ and $\bs{z}\in\N_{\X}(\bar{\bs{x}})$ with $\bs{v}+\bs{z}=0$. Since $\bs{d}\in\T_{\X}(\bar{\bs{x}})$ and $\bs{z}\in\N_{\X}(\bar{\bs{x}})$, we have $\langle \bs{z},\bs{d}\rangle\le 0$ and hence $\langle \bs{v},\bs{d}\rangle=-\langle \bs{z},\bs{d}\rangle\ge 0$, a contradiction. Therefore $0\notin \partial g(\bar{\bs{x}})+\N_{\X}(\bar{\bs{x}})$.

    \smallskip
    \noindent\textbf{Step 2: $0\notin \partial g(\bar{\bs{x}})+\N_{\X}(\bar{\bs{x}})$ $\Longrightarrow$ metric regularity.}
    Define
    \[
    \ker\!\bigl([\partial g(\bar{\bs{x}}), I_d]\bigr)
    :=\bigl\{(\lambda,\bs{z})\in\Re\times\Re^d:\ 0\in \lambda\,\partial g(\bar{\bs{x}})+\bs{z}\bigr\}.
    \]
    Since $g(\bar{\bs{x}})=0$, we have $\N_{\Re_-}(g(\bar{\bs{x}}))=\N_{\Re_-}(0)=\Re_+$.
    We claim that
    \begin{equation*}
        0\ \notin\ \partial g(\bar{\bs{x}})+\N_{\X}(\bar{\bs{x}})
        \quad\Longleftrightarrow\quad
        \ker\!\bigl([\partial g(\bar{\bs{x}}), I_d]\bigr)\ \cap\ \bigl(\N_{\Re_-}(g(\bar{\bs{x}}))\times \N_{\X}(\bar{\bs{x}})\bigr)
        \;=\;\{(0,\bs{0})\}.
    \end{equation*}
    The condition on the right-hand side is the nonsmooth regularity condition of \citet[Remark~(3) following Theorem~2.4 and Eq.~(2.8)]{burke1991exact}; see also \citet[Eq.~(33)]{borwein1986stability}. Since $g$ is locally Lipschitz and $\X$ is closed and convex, this condition yields~\citep[Theorem~3.2]{borwein1986stability} metric regularity of the mapping $\mathcal{G}(\bs{x})=g(\bs{x})+\Re_+$ at $(\bar{\bs{x}},0)$ relative to $\X$.

    Finally, metric regularity at $(\bar{\bs{x}},0)$ yields constants $\kappa>0$ and $\varepsilon>0$ such that
    \[
    \dist\bigl(\bs{x},\mathcal{G}^{-1}(0)\bigr)
    \ \le\ \kappa\,\dist\bigl(0,\mathcal{G}(\bs{x})\bigr)
    \ =\ \kappa [g(\bs{x})]_+,
    \qquad \forall\,\bs{x}\in \X\cap B(\bar{\bs{x}},\varepsilon),
    \]
    which is exactly~\eqref{eq:prelim-error-bound}.

    It remains to prove the claim. Suppose first that $0\notin\partial g(\bar{\bs{x}})+\N_{\X}(\bar{\bs{x}})$, and let $(\lambda,\bs{z})$ belong to the intersection on the right-hand side. Then $\lambda\ge0$, $\bs{z}\in\N_{\X}(\bar{\bs{x}})$, and $\lambda\bs{v}+\bs{z}=0$ for some $\bs{v}\in\partial g(\bar{\bs{x}})$. If $\lambda=0$, then $\bs{z}=\bs{0}$. If $\lambda>0$, the conic property of $\N_{\X}(\bar{\bs{x}})$ gives
    $ \bs{v}+\frac{\bs{z}}{\lambda}=0$ and $
         \frac{\bs{z}}{\lambda}\in\N_{\X}(\bar{\bs{x}})$,
    contradicting $0\notin\partial g(\bar{\bs{x}})+\N_{\X}(\bar{\bs{x}})$. Hence the intersection contains only $(0,\bs{0})$, which indeed belongs to it.

    Conversely, if $0\in\partial g(\bar{\bs{x}})+\N_{\X}(\bar{\bs{x}})$, there exist $\bs{v}\in\partial g(\bar{\bs{x}})$ and $\bs{z}\in\N_{\X}(\bar{\bs{x}})$ with $\bs{v}+\bs{z}=0$. Taking $\lambda=1$ yields a nonzero element $(1,\bs{z})$ of the intersection. This proves the claim.
\end{proof}

The error bound~\eqref{eq:prelim-error-bound} together with the Lipschitz continuity of $f$ implies local exact penalization. Similar proofs can be found in \citet{le2012exact}, \citet{chen2024penalty}, and \citet{luo1996exact}.

\begin{lemma}[Local exact penalty] \label{thm:prelim-local-exact-penalty}
    Suppose $f$ is Lipschitz continuous on the compact and convex set $\X$ with constant $L_f$, and that the local error bound~\eqref{eq:prelim-error-bound} holds at some feasible point $\bar{\bs{x}}\in\mathcal{F}$. For $\sigma>0$, consider the $\ell_1$-penalty problem
    \begin{equation}\label{p:prelim-penalty}
        \min_{\bs{x}\in\X}\ \bigl\{ f(\bs{x})+\sigma[g(\bs{x})]_+\bigr\}.
    \end{equation}
    Then the following statements hold:
    \begin{itemize}
        \item[(i)] For any $\sigma>0$, if $\bar{\bs{x}}$ is a local minimizer of the penalty problem~\eqref{p:prelim-penalty} and satisfies $g(\bar{\bs{x}})\le 0$, then $\bar{\bs{x}}$ is a local minimizer of the constrained problem~\eqref{p:prelim-constrained}.
        \item[(ii)] If $\bar{\bs{x}}$ is a local minimizer of the constrained problem~\eqref{p:prelim-constrained}, then there exists $\bar\sigma>0$ such that, for every $\sigma>\bar\sigma$, $\bar{\bs{x}}$ is a local minimizer of the penalty problem~\eqref{p:prelim-penalty}.
    \end{itemize}
\end{lemma}

\begin{proof}
\noindent\textbf{(i)} Let $\sigma>0$ and assume that $\bar{\bs{x}}$ is a local minimizer of~\eqref{p:prelim-penalty} with $g(\bar{\bs{x}})\le 0$.
    Then $[g(\bar{\bs{x}})]_+=0$. By local optimality, there exists $r>0$ such that
    \[
    f(\bar{\bs{x}})+\sigma[g(\bar{\bs{x}})]_+
    \le
    f(\bs{x})+\sigma[g(\bs{x})]_+,
    \qquad \forall\,\bs{x}\in \X\cap B(\bar{\bs{x}},r).
    \]
    For any $\bs{x}\in \mathcal{F}\cap B(\bar{\bs{x}},r)$, we have $[g(\bs{x})]_+=0$ and thus $f(\bar{\bs{x}})\le f(\bs{x})$. Hence $\bar{\bs{x}}$ is a local minimizer of~\eqref{p:prelim-constrained}.

    \medskip
    \noindent\textbf{(ii)} Assume that $\bar{\bs{x}}$ is a local minimizer of~\eqref{p:prelim-constrained}.
    By the local form of Clarke's exact penalty principle~\citep[Proposition~2.4.3]{clarke1990optimization}, $\bar{\bs{x}}$ locally minimizes $f(\bs{x})+L_f\dist(\bs{x},\mathcal{F})$ over $\X$. Combining this with the error bound~\eqref{eq:prelim-error-bound}, for any $\sigma>L_f\kappa$ and all $\bs{x}\in\X$ sufficiently close to $\bar{\bs{x}}$, we obtain
    \[
    f(\bar{\bs{x}})
    \le f(\bs{x})+L_f\dist(\bs{x},\mathcal{F})
    \le
    f(\bs{x})+L_f\kappa[g(\bs{x})]_+
    \le
    f(\bs{x})+\sigma[g(\bs{x})]_+.
    \]
    Since $[g(\bar{\bs{x}})]_+=0$, this proves~(ii).
\end{proof}

When the constrained system admits a global linear error bound, i.e.,
\begin{equation}\label{eq:prelim-global-error-bound}
    \dist(\bs{x},\mathcal{F}) \;\le\; \kappa\,[g(\bs{x})]_+,
    \qquad \forall\,\bs{x}\in \X,
\end{equation}
the same argument as in Lemma~\ref{thm:prelim-local-exact-penalty} yields a global exact penalization result.

\begin{lemma}[Global exact penalty] \label{thm:prelim-global-exact-penalty}
    Suppose that $f$ is Lipschitz continuous on the compact and convex set $\X$ with constant $L_f$, and that the global error bound~\eqref{eq:prelim-global-error-bound} holds. Then the following statements hold:
    \begin{itemize}
        \item[(i)] For any $\sigma>0$, if $\bar{\bs{x}}$ is a global minimizer of the penalty problem~\eqref{p:prelim-penalty} and satisfies $g(\bar{\bs{x}})\le 0$, then $\bar{\bs{x}}$ is a global minimizer of the constrained problem~\eqref{p:prelim-constrained}.
        \item[(ii)] If $\bar{\bs{x}}$ is a global minimizer of the constrained problem~\eqref{p:prelim-constrained}, then there exists $\bar{\sigma}>0$ such that for every $\sigma>\bar\sigma$ the point $\bar{\bs{x}}$ is a global minimizer of the penalty problem~\eqref{p:prelim-penalty}.
    \end{itemize}
\end{lemma}

\begin{proof}
The proof follows the same argument as that of
Lemma~\ref{thm:prelim-local-exact-penalty},
using the global error bound and global optimality
in place of their local counterparts.
\end{proof}

\end{document}